\documentclass[oneside,english]{amsart}
\usepackage[T1]{fontenc}
\usepackage{geometry}
\usepackage{babel}
\usepackage{mathrsfs}
\usepackage{mathtools}
\usepackage{amstext}
\usepackage{amsthm}
\usepackage{amssymb}
\usepackage{stmaryrd}
\usepackage{setspace}
\usepackage[unicode=true]
 {hyperref}

\makeatletter

\providecommand{\tabularnewline}{\\}

\numberwithin{equation}{section}
\numberwithin{figure}{section}
\theoremstyle{plain}
\newtheorem*{thm*}{\protect\theoremname}
\theoremstyle{plain}
\newtheorem{thm}{\protect\theoremname}[section]
\theoremstyle{definition}
\newtheorem{defn}[thm]{\protect\definitionname}
\theoremstyle{remark}
\newtheorem{rem}[thm]{\protect\remarkname}
\theoremstyle{plain}
\newtheorem{lem}[thm]{\protect\lemmaname}
\theoremstyle{plain}
\newtheorem{prop}[thm]{\protect\propositionname}
\theoremstyle{definition}
\newtheorem{example}[thm]{\protect\examplename}
\theoremstyle{plain}
\newtheorem{cor}[thm]{\protect\corollaryname}
\theoremstyle{plain}
\newtheorem{conjecture}[thm]{\protect\conjecturename}

\usepackage{tikz-cd}
\usepackage{comment}
\allowdisplaybreaks

\makeatother

\providecommand{\conjecturename}{Conjecture}
\providecommand{\corollaryname}{Corollary}
\providecommand{\definitionname}{Definition}
\providecommand{\examplename}{Example}
\providecommand{\lemmaname}{Lemma}
\providecommand{\propositionname}{Proposition}
\providecommand{\remarkname}{Remark}
\providecommand{\theoremname}{Theorem}

\begin{document}
\title{On the theory of $q$-characters for quantum affine superalgebras
of type $A$}
\author{Sin-Myung Lee}
\address{School of Mathematics, Korea Institute for Advanced Study, Seoul 02455,
Republic of Korea}
\email{sinmyunglee@kias.re.kr}
\thanks{This work is supported by a KIAS Individual Grant (MG095001, MG095002)
at Korea Institute for Advanced Study. }
\begin{abstract}
We develop the theory of $q$-characters for quantum affine superalgebras
of type $A$ in connection with deformed Cartan matrices. To achieve
this, we establish a Khoroshkin--Tolstoy-type multiplicative formula
of the universal $R$-matrix of the associated generalized quantum
group, from which one can read off a 2-parameter deformation of Cartan
matrices of super type $A$. We also propose a Frenkel--Mukhin-type
algorithm for $q$-characters of finite-dimensional simple modules
with integral highest $\ell$-weights.
\end{abstract}

\maketitle
\global\long\def\ket#1{\left|#1\right\rangle }%

\section{Introduction}

The study of finite-dimensional modules of quantum affine algebras,
especially of their monoidal structure, has long been an intriguing
subject with a variety of approaches and is still enjoying new ideas
from rich connections with other branches of mathematics and physics.
As a quantum affine analogue of the character of modules over finite-dimensional
simple Lie algebras, the theory of \textit{$q$-characters} \cite{FR1}
has played an essential role in the investigation of the intricate
structure of tensor products of two simple modules in both theoretical
and computational aspects. Furthermore, it provides a natural place
where various perspectives meet and hence exciting interactions may
arise, such as geometry of quiver variety \cite{Nak}, integrable
systems \cite{FR1} and cluster algebras \cite{HL}, to name a few.

\subsection{Quantum affine superalgebras and generalized quantum groups}

It is a natural problem to extend such a fruitful theory to variations
of quantum affine algebras. In this paper we are concerned with the
\textit{quantum affine superalgebra} of type $A$, a quantum group
associated to the affine Lie superalgebra $\widehat{\mathfrak{sl}}_{M|N}=\mathfrak{sl}_{M|N}\otimes\mathbb{C}[t^{\pm1}]\oplus\mathbb{C}K$
for the special linear Lie superalgebra $\mathfrak{sl}_{M|N}$. Surprisingly
enough, their representation theory is far less understood compared
with the non-super case and suffers from conceptual and technical
complications, let alone the difficulties already presented in the
study of Lie superalgebras. There has been a remarkable series of
works by Zhang \cite{Z1,Z2,Z3,Z4,Z5}, sometimes giving even new insights
to the original theory for quantum affine algebras; at other times
relying on the analysis of specific representations such as fundamental
representations and Kirillov--Reshetikhin modules. It is hence desirable
to have tools to understand finite-dimensional modules beyond them,
including a practical method to compute $q$-characters.

Recently, such a difficulty has been partially resolved by considering
a parallel theory for a \textit{generalized quantum group} \cite{KOS}.
It is a Hopf algebra which is not super but very much resembles the
quantum affine superalgebra. Roughly, the super structure of the latter
is confined by the systematic introduction of an `additional' quantum
parameter $\widetilde{q}=-q^{-1}$ in the former and it turns out
that their module categories are equivalent as abelian categories.
Although one has to be careful when dealing with tensor products as
it is not a monoidal equivalence a priori, the generalized quantum
group has an advantage that many constructions known for ordinary
quantum affine algebras can be applied. For instance, in the author's
past work with Kwon \cite{KL} we could make use of the universal
$R$-matrix of the generalized quantum group to reproduce key results
in \cite{KKKO}, which enables us to study the structure of tensor
products of polynomial modules in terms of the denominator of the
normalized $R$-matrices. Therefore, we aim to develop the theory
of $q$-characters for the quantum affine superalgebras by means of
generalized quantum groups in this paper.

\subsection{Main results}

We fix $M,N>0$ such that $M\neq N$. Let $\epsilon=(\epsilon_{1},\dots,\epsilon_{n})$
be a sequence of $M$ $0$'s and $N$ $1$'s which represents a choice
of a Borel subalgebra of $\mathfrak{sl}_{M|N}$ (see \cite[Section 1.3.2]{CW})
or the parities of simple root vectors. We associate to each $\epsilon$
a generalized quantum group $\mathcal{U}(\epsilon)$ of affine type
$A$ and also a quantum affine superalgebra $U(\epsilon)$.

\subsubsection{Braid action and a multiplicative formula for the universal $R$-matrix}

We begin by establishing several important properties of a braid action
on $\mathcal{U}(\epsilon)$ along the lines of \cite{Lb}. It is a
key tool to obtain Drinfeld presentation of $\mathcal{U}(\epsilon)$
(Theorem \ref{thm:GQG-Drinfeld}) and also to construct affine root
vectors and PBW bases (Section \ref{subsec:aff-PBW}) as in \cite{B2}
for non-super types and \cite{Y} for $U(\epsilon)$. In our super
case the isomorphism $T_{i}$ is defined between $\mathcal{U}(\epsilon)$
and $\mathcal{U}(\epsilon s_{i})$ where $\epsilon s_{i}$ is obtained
from $\epsilon$ by permuting $i$-th and $(i+1)$-st entries and
so not an automorphism if $\epsilon_{i}\neq\epsilon_{i+1}$. Nevertheless,
the original arguments for usual quantum groups \cite{Lb} work for
$\mathcal{U}(\epsilon)$ without any essential modification, which
we find it advantageous over $U(\epsilon)$. We also recall from \cite{KL}
an algebra isomorphism $\tau$ between $\mathcal{U}(\epsilon)$ and
$U(\epsilon)$ (up to a slight extension of both algebras) to compare
the generators in Drinfeld presentations of each algebras. This enables
us to lift known results for $U(\epsilon)$ to $\mathcal{U}(\epsilon)$
such as a triangular decomposition (\ref{eq:GQGDr-tri-decomp}) and
the evaluation homomorphism (Proposition \ref{prop:eval-hom}). In
short, the generalized quantum group is an intermediate object between
the quantum affine superalgebra and rather well-studied quantum groups.

Let $q$ be an indeterminate and $q_{i}=q$ if $\epsilon_{i}=0$ and
$q_{i}=-q^{-1}$ if $\epsilon_{i}=1$ for $i=1,\dots,n$. The first
main result of this paper is the following Khoroshkin--Tolstoy-type
multiplicative formula \cite{KT} of the universal $R$-matrix for
$\mathcal{U}(\epsilon)$ (Theorem \ref{thm:mult-formula-R}). 
\begin{thm*}
If $\epsilon$ is such that $(\epsilon_{i}+\epsilon_{i+1})(\epsilon_{i+1}+\epsilon_{i+2})=0$
for $1\leq i\leq n-2$, then the universal $R$-matrix $\mathcal{R}$
for $\mathcal{U}(\epsilon)$ has the following decomposition:
\begin{align*}
\mathcal{R} & =\mathcal{R}^{+}\mathcal{R}^{0}\mathcal{R}^{-}\overline{\Pi}_{\mathbf{q}},\\
\mathcal{R}^{\pm} & =\prod_{\beta\in\Phi_{+}(\pm\infty)}\exp_{\beta}\left((q^{-1}-q)e_{\beta}\otimes f_{\beta}\right),\\
\mathcal{R}^{0} & =\exp\left(-\sum_{r>0}\sum_{i,j=1}^{n-1}\frac{r(q-q^{-1})^{2}}{q_{i}^{r}-q_{i}^{-r}}\widetilde{C}_{ji}^{r}h_{i,r}\otimes h_{j,-r}\right)
\end{align*}
where $\widetilde{C}^{r}=(\widetilde{C}_{ij}^{r})_{i,j}$ is the inverse
matrix of 
\[
C^{r}=\left(\frac{\mathbf{q}(\alpha_{i},\alpha_{j})^{r}-\mathbf{q}(\alpha_{i},\alpha_{j})^{-r}}{q_{i}^{r}-q_{i}^{-r}}\right)_{i,j=1,\dots,n-1}.
\]
\end{thm*}
Here $e_{\beta}$ and $f_{\beta}$ (resp. $h_{i,r}$) are quantum
analogues in $\mathcal{U}(\epsilon)$ of real (resp. imaginary) root
vectors of the affine Lie superalgebra $\widehat{\mathfrak{sl}}_{M|N}$.
The formula, up to the input $\widetilde{C}^{r}$, looks the same
as those in the non-super types \cite{D1}. Indeed, this result relies
on a number of preliminary results (Section \ref{subsec:aff-PBW}--\ref{subsec:Pairing-imag-rvector})
to compute a Hopf pairing between the PBW basis vectors. Since the
root system with a convex ordering of $\widehat{\mathfrak{sl}}_{M|N}$
is identical to that of $\widehat{\mathfrak{sl}}_{M+N}$, almost every
proofs in \cite[Section 4-10]{D1} for non-super cases, mutatis mutandis,
work for $\mathcal{U}(\epsilon)$. The only exception occurs in the
computation of the Hopf pairing $(h_{i,r},h_{j,s})$ when $\epsilon_{j}+\epsilon_{j+1}=1$,
where we invoke a relation between the Hopf pairing and a symmetric
bilinear form on the negative half of $\mathcal{U}^{-}(\epsilon)$
and a $q$-bracket realization of imaginary root vectors (Section
\ref{subsec:Lusztig-f-bil-form}).

\subsubsection{The $q$-character map and its transfer matrix construction}

Now we assume $\epsilon=\epsilon_{M|N}$ is the standard $(01)$-sequence
(\ref{eq:01seq-std}) and begin to study finite-dimensional $\mathcal{U}(\epsilon)$-modules
in view of the Drinfeld presentation. We introduce basic notions from
the usual highest $\ell$-weight theory \cite{CP}. This approach
for $\mathcal{U}(\epsilon)$-modules is compatible with the one for
$U(\epsilon)$-modules \cite{Z1} through $\tau$, in the sense that
the pullback by $\tau$ maps simple highest $\ell$-weight modules
over $U(\epsilon)$ to those over $\mathcal{U}(\epsilon)$ and the
$\ell$-weights of the latter can be read easily from those of the
former, and vice versa (Remark \ref{rem:compare-l-wt-GQG-qasa}, \ref{rem:compare-fund-GQG-qasa}).
In particular, the fundamental problem of computing the $q$-character
of simple $U(\epsilon)$-modules is equivalent to that of $\mathcal{U}(\epsilon)$-modules,
which justifies our approach toward the theory of $q$-characters
for quantum affine superalgebras via generalized quantum groups.

To be precise, we are mostly focused on the category $\mathcal{F}(\epsilon)$
of the finite-dimensional $\mathcal{U}(\epsilon)$-modules with \textit{integral}
$\ell$-weights, that is every $\ell$-weight $\Psi$ is a product
of the \textit{fundamental $\ell$-weights} $Y_{i,a}^{\pm1}$ and
$\widetilde{Y}_{j,a}^{\pm1}$ for $i\leq M\leq j$, $a\in\mathbb{C}^{\times}$
defined as usual with the quantum parameter $q$ and $-q^{-1}$ respectively
(Definition \ref{def:fund-l-wt}). Beyond such modules are there more
finite-dimensional simple $\mathcal{U}(\epsilon)$-modules, parametrized
by $n-2$ Drinfeld polynomials (for each $i\neq M$) and one degree
zero rational function (for $i=M$) (\textit{cf.} highest weight classification
of finite-dimensional simple $\mathfrak{sl}_{M|N}$-modules \cite[Section 2.1.1]{CW}),
but we believe $\mathcal{F}(\epsilon)$ is of reasonable size to expect
nice theories on monoidal structures and also applications to mathematical
physics. For example, it is the smallest rigid monoidal Serre subcategory
that contains the natural representation, and contains all the \textit{fundamental
representations} $L(Y_{i,a})$, $L(\widetilde{Y}_{j,b})$ whose module
structure is explicitly described in Section \ref{subsec:fund-repn}.
We remark that the latter follows from the cone property (Corollary
\ref{cor:q-char-cone-prop}) for $\ell$-weights of highest $\ell$-weight
modules. This important property is a result of a general observation
\cite[Proposition 3.9]{MY} originally established for a highest $\ell$-weight
module over a quantum affine algebra $U_{q}(\widehat{\mathfrak{g}})$
that belongs to the BGG category as a $U_{q}(\mathfrak{g})$-module,
which is in accordance with the philosophy of the super duality \cite[Introduction]{CLW}. 

The \textit{$q$-character} of $V\in\mathcal{F}(\epsilon)$ is now
defined as the generating function of the $\ell$-weights of $V$.
Although it is in general very difficult to compute the $q$-character
in this setting, we give several examples of fundamental representations
and Kirillov--Reshetikhin-type modules in the rank 2 case $\epsilon=(001)$
that can be worked out directly. This results, together with another
case $\epsilon=(011)$ which can be studied in a parallel manner,
are crucially used in the description of Frenkel--Mukhin type algorithm
below.

By the standard argument \cite{FR2} the resulting $q$-character
map can be understood as an embedding
\[
\chi_{q}:K(\mathcal{F}(\epsilon))\longrightarrow\mathscr{Y}(\epsilon)=\frac{\mathbb{Z}[Y_{i,a}^{\pm},\widetilde{Y}_{j,a}^{\pm}]_{i\leq M\leq j,a\in\mathbb{C}^{\times}}}{(Y_{M,a}\widetilde{Y}_{M,-a}-Y_{M,b}\widetilde{Y}_{M,-b})_{a,b\in\mathbb{C}^{\times}}}
\]
of the Grothendieck ring $K(\mathcal{F}(\epsilon))$ of $\mathcal{F}(\epsilon)$
into the Laurent polynomial ring $\mathscr{Y}(\epsilon)$ (the relation
comes from the definition of $Y_{M,a}$ and $\widetilde{Y}_{M,a}$).
Using the multiplicative formula for the universal $R$-matrix above,
we can redefine this map along the transfer matrix method \cite{FR2}
to interpret $Y_{i,a}$, $\widetilde{Y}_{j,a}$ as formal series with
coefficients in the subalgebra generated by $h_{i,r}\in\mathcal{U}(\epsilon)$
($i=1,\dots,n-1$, $r\in\mathbb{Z}$), or free fields in the sense
of Remark \ref{rem:hh-q-Heis}. This construction should serve as
a natural background to understand $q$-characters for $\mathcal{U}(\epsilon)$,
and hence for quantum affine superalgebras, in connection with integrable
systems (see for example \cite{T,Z5}) and also the map $\chi_{q}$
with a free field realization of the $q$-deformed $Y$-algebra $Y_{0,N,M}$
as a `super' analogue of \cite[Section 8]{FR2} (see Section \ref{subsec:intro-def-Cartan}
below).

\subsubsection{Applications}

The transfer matrix construction of $\chi_{q}$ explains how the matrix
$\widetilde{C}^{r}$ controls the $q$-characters, and hence the representation
theory of $\mathcal{U}(\epsilon)$. As a demonstration, we compute
the universal coefficient of the universal $R$-matrices on $V\otimes W$
as in \cite[Section 4.3]{FR2}. It is recently used to define an invariant
$\Lambda(V,W)$ that plays a significant role in establishing a monoidal
categorification of cluster algebras \cite{KKOP} whose super generalization
will be discussed in the future work. Moreover, by comparing the result
with another method using the \textit{denominator} of the normalized
$R$-matrix \cite[Appendix A]{AK}, we can even deduce a formula for
the denominator for modules of the form $L(Y_{i,a})\otimes L(\widetilde{Y}_{j,b})$
(Theorem \ref{thm:denom-R-+-}). Such a tensor product has complicated
structures that does not arise in non-super cases, even as a module
over the subalgebra of finite type (\textit{i.e.} corresponding to
$U_{q}(\mathfrak{sl}_{M|N})$). Since the denominator is known to
contain much information on the tensor product structure, we hope
that this result will shed a new light on this problem.

At last, we propose an analogue of Frenkel--Mukhin algorithm \cite{FM}
to compute the $q$-character of simple $\mathcal{U}(\epsilon)$-modules
in $\mathcal{F}(\epsilon)$ by means of restrictions to each rank
1 subalgebra. Let $J=\{p,p+1,\dots,p^{\prime}-1\}\subset\{1,\dots,n\}$,
$\epsilon_{J}=(\epsilon_{p},\dots,\epsilon_{p^{\prime}})$ and $\mathcal{U}(\epsilon)_{J}$
be the subalgebra of $\mathcal{U}(\epsilon)$ which is essentially
$\mathcal{U}(\epsilon_{J})$. Given a $\mathcal{U}(\epsilon)$-module
$V$, the $q$-character of its restriction to $\mathcal{U}(\epsilon)_{J}$
can be obtained by substituting $Y_{i,a}=1=\widetilde{Y}_{j,b}$,
namely its image under the naive restriction map $\beta_{J}:\mathscr{Y}(\epsilon)\rightarrow\mathscr{Y}(\epsilon_{J})$.
As in \cite[Section 3]{FM} it can be refined to $\tau_{J}$ so that
it is injective if $J=\{p\}$ for $p\neq M$ and hence the restriction
to each \textit{even} rank 1 subalgebra $\mathcal{U}(\epsilon_{\{p\}})\cong U_{q_{p}}(\widehat{\mathfrak{sl}}_{2})$
can be inverted in the sense of the commutative diagram (\ref{eq:comm-diag-restriction}).
In the \textit{odd} rank 1 case $J=\{M\}$ where $\tau_{J}$ is not
injective we should consider the rank 2 subalgebras $J=\{M,M\pm1\}$
at the same time, which is the reason why we need to know the $q$-character
of several $\mathcal{U}(001)$- or $\mathcal{U}(011)$-modules to
determine the algorithm.

Now the algorithm is obtained from the original one by revising the
rule of the `expansion in the $M$-th direction' according to the
representation theory of $\mathcal{U}(\epsilon_{\{M\}})=\mathcal{U}(01)$
(Appendix \ref{sec:GQG-RT-rank1}). It is not very hard to find simple
modules for which the algorithm fails (Example \ref{exa:FM-alg-fails},
\textit{cf.} \cite[Example 5.6]{HL}), whereas it is also not an easy
problem to determine for which simples it operates well and produces
the correct $q$-character even in the non-super case. We can show
that the fundamental representations fall into the latter as in the
proof of \cite[Theorem 5.9]{FM}, but we do not have a handy sufficient
condition at this stage. For instance, the usual notion of `dominance'
as finite-dimensionality is not strong enough in the super case and
hence the same for the minuscule property which was originally assumed.
Nevertheless, the algorithm is verified by a lot of examples beyond
fundamentals and Kirillov--Reshetikhin modules and also serves as
an extremely useful method to deduce partial information on $q$-characters
even for those it fails.

\subsection{\label{subsec:intro-def-Cartan}Deformed Cartan matrices for $\mathfrak{sl}_{M|N}$}

Let $C$ be the Cartan matrix associated with a finite-dimensional
simple complex Lie algebra $\mathfrak{g}$. A 2-parameter deformation
$C(q,t)$ of $C$ was introduced by Frenkel and Reshetikhin \cite{FR1}
as an ingredient to define a $(q,t)$-deformation $W_{q,t}(\mathfrak{g})$
of the (classical principal) $W$-algebra of $\mathfrak{g}$ inside
a $(q,t)$-deformed Heisenberg algebra as the joint kernel of screening
operators. In pursuit of Langlands-type dualities they considered
various limits on the parameters $q,t$, where one also expects a
connection with representation theory of quantum groups through quantum
integrable models. A classical limit $t\rightarrow1$ is especially
well-understood as the algebra $W_{q,1}(\mathfrak{g})$ coincides
with the image of the Grothendieck ring of a module category of the
quantum affine algebra $U_{q}(\widehat{\mathfrak{g}})$ under the
$q$-character map \cite{FR2,FM}. In that sense, the $q$-deformed
Cartan matrix $C(q,1)$ can be seen as a numerical input to the theory
of $q$-characters and one may push this viewpoint further to construct
(quantum) virtual Grothendieck rings from other limits of $C(q,t)$,
see for example \cite{KO}. To find a representation theoretic realization
of the latters is an active area of research and we refer the reader
to \cite{FHR} for recent prospects.

From our discussion we observe that the matrix $C^{r}$ appears in
the multiplicative formula of the universal $R$-matrix of $\mathcal{U}(\epsilon)$
and so plays the role of $C(q^{r},1)$ in the non-super case. Interestingly,
$C^{r}$ already hints the existence of further deformation as it
cannot be written as $C(q^{r})$ for some 1-parameter matrix $C(z)$,
which also agrees with our viewpoint to generalized quantum groups
as a kind of 2-parameter quantum group. Indeed, we may read off a
(symmetrization of) matrix $C(q,\widetilde{q})$ from a careful examination
of our computation of $(h_{i,1},h_{j,-1})$, regarding $\widetilde{q}$
as a formal variable independent of $q$, so that $C(q^{r},\widetilde{q}^{r})$
specializes to $C^{r}$ by taking $\widetilde{q}=-q^{-1}$. We call
$C(q,\widetilde{q})$ a \textit{$(q,\widetilde{q})$-deformed Cartan
matrix} of $\mathfrak{sl}_{M|N}$ associated with $\epsilon$.

Our $(q,\widetilde{q})$-deformed Cartan matrix coincides with the
one recently introduced from the $W$-algebra side by Feigin, Jimbo,
Mukhin and Vilkoviskiy \cite{FJMV}. It is based on a `super' generalization
of a realization of $W_{q,t}(\mathfrak{sl}_{M})$ as a truncation
of the quantum toroidal algebra $\mathcal{E}$ associated to $\mathfrak{gl}_{1}$
\cite[Section 5]{BFM} which is also a $q$-deformation of \cite{SV}.
Namely, $\mathcal{E}$ has three kinds of Fock modules associated
to three quantum parameters, and $W_{q,t}(\mathfrak{sl}_{M})$ is
obtained as the image of the $\mathcal{E}$-action on the tensor product
of $M$ Fock modules of the same kind. Our super case is exactly when
one takes $M$ Fock modules of one kind and $N$ of another, where
the $(01)$-sequence $\epsilon$ indicates the order of the tensor
product. We thus have a natural interpretation of the 2-parameter
nature of $C^{r}$ and hence of the representation theory of $\mathcal{U}(\epsilon)$
and $U(\epsilon)$. Let us also remark that the resulting (deformed)
$W$-algebra in the super case is the so-called (deformed) $Y$-algebra
$Y_{0,N,M}$ which is quite different from the usual $W$-superalgebra
and not super a priori.

Furthermore, the last construction also provides the screening operator
description that induces a Frenkel--Mukhin-type algorithm to produce
an element of the deformed $Y_{0,N,M}$, a \textit{$qq$-character},
in a combinatorial way \cite{FJM1,FJM2}. While they work with a quite
general assumption, it is pointed out that for a special class of
$qq$-characters they should give rise to the $q$-characters of finite-dimensional
simple modules of the quantum affine superalgebras after a specialization
that corresponds to $\widetilde{q}=\pm q^{-1}$ in our setting (and
$t=1$ in the non-super theory). Since our algorithm is essentially
the specialization of theirs, this work provides a representation
theoretic explanation of the story at the $W$-algebra side. It also
strongly suggests that the image of the $q$-character map should
be given as the joint kernel of the screening operators for $q$-deformed
$Y_{0,N,M}$ \cite[Section 5.4]{BFM} at $\widetilde{q}=-q^{-1}$.
We hope to come back to this problem in the near future. 

\subsection{Organization}

This article is organized as follows. In Section \ref{sec:GQG} we
recall the definition of generalized quantum groups of affine type
$A$ and discuss fundamental structures, which include braid actions,
non-degenerate bilinear form on the negative half of the algebra,
Drinfeld presentation and a comparison result with quantum affine
superalgebras of type $A$. In Section \ref{sec:mult-formula-R} we
establish the multiplicative formula of the universal $R$-matrix,
from which we introduce a 2-parameter deformation of Cartan matrices
of $\mathfrak{sl}_{M|N}$. In Section \ref{sec:RT-GQG} we begin to
study finite-dimensional $\mathcal{U}(\epsilon)$-modules with respect
to Drinfeld presentation. In particular, we introduce fundametal representations
with an explicit description of their module structure and provide
examples of $q$-characters for Kirillov--Reshetikhin modules in
the rank 2 case. Finally in Section \ref{sec:FR-FM-qchar} we develop
the theory of $q$-characters for $\mathcal{U}(\epsilon)$ along the
lines of \cite{FR2,FM} using the multiplicative formula above. As
an application we compute the universal coefficient of the $R$-matrices
from which we can deduce a denominator formula for fundamental representations.
We also propose a Frenkel--Mukhin-type algorithm followed by examples
and non-examples. In the first two appendices we give a proof of the
existence of braid action in the rank 2 case and a construction of
Drinfeld presentation for $\mathcal{U}(\epsilon)$. The last Appendix
\ref{sec:GQG-RT-rank1} is devoted to the structure and representation
theory of $\mathcal{U}(\epsilon)$ for $\epsilon=(01)$ which is essential
in the description of the algorithm but not covered by our general
discussions.

\subsection*{Acknowledgement}

The author would like to express his sincere gratitude to Ryo Fujita
for a number of stimulating discussions and inspiring comments. He
is also grateful to Il-Seung Jang for his interests in this work and
discussions, and to Evgeny Mukhin for his valuable comments. This
work was partly done while the author was visiting Research Institute
for Mathematical Sciences, Kyoto University. He would like to thank
RIMS for their warm hospitality during his stay. 

\section{Generalized quantum groups of type $A$\label{sec:GQG}}

\subsection*{Notations and conventions}

Throughout this article, we fix $M,N\in\mathbb{Z}_{>0}$ such that
$M\neq N$ and $n=M+N\geq3$. Let $\Bbbk=\mathbb{Q}(q)$. The following
notations are frequently used:
\begin{itemize}
\item $\epsilon=(\epsilon_{1},\dots,\epsilon_{n})\in\{0,1\}^{n}$: a $(01)$-sequence
with $M$ $0$'s and $N$ $1$'s 
\item $\mathbb{I}=\{1<2<\cdots<n\}$: an ordered set with a partition $\mathbb{I}=\mathbb{I}_{0}\sqcup\mathbb{I}_{1}$,
where
\[
\mathbb{I}_{0}=\{i\,|\,\epsilon_{i}=0\},\quad\mathbb{I}_{1}=\{i\,|\,\epsilon_{i}=1\}.
\]
\item $\mathring{P}=\bigoplus_{i\in\mathbb{I}}\mathbb{Z}\delta_{i}$: the
($\mathfrak{gl}_{n}$-)weight lattice with a symmetric bilinear form
$(\,\cdot\,,\,\cdot\,)_{\epsilon}$ defined by
\[
(\delta_{i},\delta_{j})_{\epsilon}=(-1)^{\epsilon_{i}}\delta_{ij}
\]
\item $I=\{0,\dots,n-1\}$, $\mathring{I}=I\setminus\{0\}$.
\item $\alpha_{i}=\delta_{i}-\delta_{i+1}\in\mathring{P}$ for $i\in I$.
In particular, we have $\sum_{i\in I}\alpha_{i}=0$.
\item $p_{\epsilon}(i)=\epsilon_{i}+\epsilon_{i+1}\in\mathbb{Z}/2\mathbb{Z}$
for $i\in I$, which extends to a \textit{parity} function 
\[
p_{\epsilon}:\mathring{Q}=\sum\mathbb{Z}\alpha_{i}\rightarrow\mathbb{Z}/2\mathbb{Z},\quad p_{\epsilon}(\sum c_{i}\alpha_{i})=\sum c_{i}p_{\epsilon}(i).
\]
\item $q_{i}=(-1)^{\epsilon_{i}}q^{(-1)^{\epsilon_{i}}}$ ($i\in\mathbb{I}$).
We also write $\widetilde{q}=-q^{-1}$.
\item $\mathbf{q}_{\epsilon}(\,\cdot\,,\,\cdot\,):\mathring{P}\times\mathring{P}\longrightarrow\Bbbk^{\times}$:
a symmetric biadditive function defined by
\[
\mathbf{q}_{\epsilon}\left(\sum\lambda_{i}\delta_{i},\sum\mu_{i}\delta_{i}\right)=\prod q_{i}^{\lambda_{i}\mu_{i}}.
\]
\item For $a\in\Bbbk^{\times}$ or $\mathbb{C}^{\times}$, $n\in\mathbb{Z}_{\geq0}$
and $0\leq k\leq n$, we let 
\[
[n]_{a}=a^{n-1}+a^{n-3}+\cdots+a^{1-n},\quad[n]_{a}!=[n]_{a}[n-1]_{a}\cdots[1]_{a},\quad\begin{bmatrix}n\\
k
\end{bmatrix}_{a}=\frac{[n]_{a}!}{[k]_{a}![n-k]_{a}!}.
\]
\end{itemize}
In Section 2 and 3, we will also consider the affine weight lattice
$P$ and the affine root lattice $Q$ given by
\[
P=\mathbb{Z}\mathbf{\delta}\oplus\bigoplus_{i\in I}\mathbb{Z}\mathbf{\delta}_{i}\supset Q=\bigoplus_{i\in I}\mathbb{Z}\mathbf{\alpha}_{i},\quad\alpha_{i}=\delta_{i}-\delta_{i+1}+\delta_{i,0}\delta
\]
which are related to $\mathring{P}$ and $\mathring{Q}$ under the
projection 
\begin{align*}
\mathrm{cl}: & P\longrightarrow\mathring{P}\\
 & \mathbf{\delta}_{i}\longmapsto\delta_{i}\\
 & \mathbf{\delta}\longmapsto0.
\end{align*}
The above definitions naturally extend to $P$ and $Q$, for example
$\mathbf{q}_{\epsilon}(\lambda,\mu)\coloneqq\mathbf{q}_{\epsilon}(\mathrm{cl}(\lambda),\mathrm{cl}(\mu))$
for $\lambda,\mu\in P$ and $p_{\epsilon}(\beta)\coloneqq p_{\epsilon}(\mathrm{cl}(\beta))$
for $\beta\in Q$. Thus, for notational convenience, we will use the
same symbol $\delta_{i}$ and $\alpha_{i}$ to denote those in either
$P$ or $\mathring{P}$ (which will be made clear in each context). 

In Section 4 and 5, we will fix $\epsilon=\epsilon_{M|N}$, the standard
$(01)$-sequence
\begin{equation}
\epsilon_{M|N}=(\underbrace{0,0,\dots,0}_{M},\underbrace{1,1,\dots1}_{N}).\label{eq:01seq-std}
\end{equation}
In particular, $p(i)=1$ if and only if $i=0,M$. Whenever $\epsilon$
is obvious from the context, we suppress it from the notations.

\subsection{Generalized quantum groups and a braid action\label{subsec:GQG-braid}}
\begin{defn}[\cite{KOS,Machida}]
\label{def:GQG}The generalized quantum group $\mathcal{U}(\epsilon)$
of affine type $A$ associated with $\epsilon$ is defined to be a
$\Bbbk$-algebra generated by $e_{i},\,f_{i}$ and $k_{i}^{\pm1}$
for $i\in I$ subject to the following relations:
\begin{align}
 & k_{i}k_{i}^{-1}=k_{i}^{-1}k_{i}=1,\quad k_{i}k_{j}=k_{j}k_{i},\label{eq:GQG-DJ-kk}\\
 & k_{i}e_{j}k_{i}^{-1}=\mathbf{q}(\alpha_{i},\alpha_{j})e_{j},\quad k_{i}f_{j}k_{i}^{-1}=\mathbf{q}(\alpha_{i},\alpha_{j})^{-1}f_{j},\label{eq:GQG-DJ-kekf}\\
 & e_{i}f_{j}-f_{j}e_{i}=\delta_{ij}\frac{k_{i}-k_{i}^{-1}}{q_{i}-q_{i}^{-1}},\label{eq:GQG-DJ-ef}\\
 & [e_{i},e_{j}]_{\mathbf{q}(\alpha_{i},\alpha_{j})^{-1}}=[f_{i},f_{j}]_{\mathbf{q}(\alpha_{i},\alpha_{j})^{-1}}=0\quad\quad\text{if }j\neq i\ensuremath{\pm1},\label{eq:GQG-DJ-quadSerre}\\
 & \begin{array}{c}
[e_{i},[e_{i},e_{j}]_{\mathbf{q}(\alpha_{i},\alpha_{j})^{-1}}]_{\mathbf{q}(\alpha_{i},\alpha_{i}+\alpha_{j})^{-1}}=0\\{}
[f_{i},[f_{i},f_{j}]_{\mathbf{q}(\alpha_{i},\alpha_{j})^{-1}}]_{\mathbf{q}(\alpha_{i},\alpha_{i}+\alpha_{j})^{-1}}=0
\end{array}\quad\quad\text{if }p(i)=0\text{ and }j=i\pm1,\label{eq:GQG-DJ-evenSerre}\\
 & \begin{array}{c}
[[[e_{i-1},e_{i}]_{\mathbf{q}(\alpha_{i-1},\alpha_{i})^{-1}},e_{i+1}]_{\mathbf{q}(\alpha_{i-1}+\alpha_{i},\alpha_{i+1})^{-1}},e_{i}]=0\\{}
[[[f_{i-1},f_{i}]_{\mathbf{q}(\alpha_{i-1},\alpha_{i})^{-1}},f_{i+1}]_{\mathbf{q}(\alpha_{i-1}+\alpha_{i},\alpha_{i+1})^{-1}},f_{i}]=0
\end{array}\quad\quad\text{if }p(i)=1\text{ and }n>3\label{eq:GQG-DJ-oddSerre}
\end{align}
where $[x,y]_{a}=xy-ayx$. If $n=3$, there is exactly one $i\in I$
such that $p(i)p(i+1)=1$, and the quartic relation (\ref{eq:GQG-DJ-oddSerre})
is replaced by the following quintic one
\begin{align}
 & \begin{array}{c}
[e_{i+1},[e_{i},[e_{i+1},[e_{i},e_{i-1}]_{q_{i}}]_{q_{i+1}q_{i+2}}]]_{q_{i+1}}=[e_{i},[e_{i+1},[e_{i},[e_{i+1},e_{i-1}]_{q_{i-1}}]_{q_{i}q_{i+1}}]]_{q_{i+1}},\\{}
[f_{i+1},[f_{i},[f_{i+1},[f_{i},f_{i-1}]_{q_{i}}]_{q_{i+1}q_{i+2}}]]_{q_{i+1}}=[f_{i},[f_{i+1},[f_{i},[f_{i+1},f_{i-1}]_{q_{i-1}}]_{q_{i}q_{i+1}}]]_{q_{i+1}}.
\end{array}\label{eq:GQG-DJ-deg5Serre}
\end{align}
We often call (\ref{eq:GQG-DJ-oddSerre}) and (\ref{eq:GQG-DJ-deg5Serre})
\textit{odd} Serre relations in constrast with the usual \textit{even}
Serre relation (\ref{eq:GQG-DJ-evenSerre}), although there are variety
of Serre relations at odd simple roots beyond type $A$ (see \cite[Proposition 6.7.1]{Y}).

It is a Hopf algebra with the following comultiplication and antipode:
\begin{align*}
\Delta: & e_{i}\mapsto e_{i}\otimes1+k_{i}\otimes e_{i},\quad f_{i}\mapsto f_{i}\otimes k_{i}^{-1}+1\otimes f_{i},\quad k_{i}\mapsto k_{i}\otimes k_{i},\\
S: & e_{i}\mapsto-k_{i}^{-1}e_{i},\quad f_{i}\mapsto-f_{i}k_{i},\quad k_{i}\mapsto k_{i}^{-1}.
\end{align*}
Moreover, we assign a $Q$-grading to $\mathcal{U}(\epsilon)$ by
$\deg e_{i}=\mathbf{\alpha}_{i}=-\deg f_{i}$ and $\deg k_{\lambda}=0$.
Then for $x\in\mathcal{U}(\epsilon)_{\beta}$, we have $k_{i}xk_{i}^{-1}=\mathbf{q}(\alpha_{i},\mathrm{cl}(\beta))x.$
\end{defn}

Let $\mathcal{U}^{+}(\epsilon)$, $\mathcal{U}^{-}(\epsilon)$ and
$\mathcal{U}^{0}(\epsilon)$ be the subalgebras of $\mathcal{U}(\epsilon)$
generated by $e_{i}$'s, $f_{i}$'s and $k_{i}$'s for $i\in I$,
respectively. By a standard argument, we have a triangular decomposition
\begin{equation}
\mathcal{U}(\epsilon)\cong\mathcal{U}^{\mp}(\epsilon)\otimes\mathcal{U}^{0}(\epsilon)\otimes\mathcal{U}^{\pm}(\epsilon)\quad\text{as a vector space.}\label{eq:GQG-tri-decomp}
\end{equation}
Note that $\mathcal{U}^{\geq0}(\epsilon)\coloneqq\mathcal{U}^{0}(\epsilon)\mathcal{U}^{+}(\epsilon)$
and $\mathcal{U}^{\leq0}(\epsilon)\coloneqq\mathcal{U}^{-}(\epsilon)\mathcal{U}^{0}(\epsilon)$
are subalgebras as well.

There is a $\mathbb{Q}$-linear anti-involution $\Omega$ defined
by
\[
\Omega:e_{i}\mapsto f_{i},\quad f_{i}\mapsto e_{i},\quad k_{i}\mapsto k_{i}^{-1},\quad q\mapsto q^{-1}
\]
which exchanges $\mathcal{U}^{+}(\epsilon)$ and $\mathcal{U}^{-}(\epsilon)$
and satisfies $\Delta\circ\Omega=(\Omega\otimes\Omega)\circ\Delta^{\mathrm{op}}$.
\begin{rem}
The generalized quantum group $\mathcal{U}(\epsilon)$ was first introduced
in \cite{KOS} without Serre relations which were subsequently provided
in \cite{Machida}, implicitly assuming $n>3$. For completeness,
we propose the definition for $n=3$ by supplementing (\ref{eq:GQG-DJ-deg5Serre}),
motivated from the one for the quantum affine superalgebra $U_{q}(\widehat{\mathfrak{sl}}_{2|1})$
in \cite[(QS4)(10)]{Y}. The case $M=N=1$ is explained separately
in Appendix \ref{sec:GQG-RT-rank1} as it cannot be defined by the
above Drinfeld--Jimbo-type presentation. 
\end{rem}

The Hopf algebra $\mathcal{U}(\epsilon)$ can be seen as an analogue
of the quantum affine superalgebra $U_{q}(\widehat{\mathfrak{sl}}_{M|N})$
associated with the choice of a Borel subalgebra associated with $\epsilon$
\cite{Y} (see Section \ref{subsec:Drinfeld-pres-GQG} below for a
precise presentation, where it is denoted by $U(\epsilon)$), while
$\mathcal{U}(\epsilon)$ has no superalgebra structure a priori. A
precise relation between $\mathcal{U}(\epsilon)$ and $U_{q}(\widehat{\mathfrak{sl}}_{M|N})$
is given in \cite[Section 2.3]{KL} and will be recalled later.

on, $W$ acts from the right on the set of the $(01)$-sequences $\epsilon$
with $M$ $0$'s and $N$ $1$'s, where $s_{0}$ permutes the first
and the $n$-th entries of $\epsilon$. We always assume that $W$
acts on the affine root lattice $Q$ as in the non-super type $A$:
\begin{equation}
s_{i}(\alpha_{j})=\begin{cases}
-\alpha_{i} & \text{if }i=j\\
\alpha_{j}+\alpha_{i} & \text{if }i=j\pm1\\
\alpha_{j} & \text{otherwise}.
\end{cases}\label{eq:affW-action}
\end{equation}

There is a `braid action' $T_{i}$ ($i\in I$) which plays exactly
the same role as the one by Lusztig on the usual quantum groups ($T_{i}$
here corresponds to $T_{i,1}^{\prime\prime}$ in \cite[Chapter 37]{Lb}).
The proof for the remaining case of $n=3$ can be found in Appendix
\ref{sec:well-defd-T_i-rank2}.
\begin{thm}[\cite{Machida,Yu}]
 For each $i\in I$, there exists an algebra isomorphism $T_{i}:\mathcal{U}(\epsilon)\longrightarrow\mathcal{U}(\epsilon s_{i})$
given by
\begin{equation}
T_{i}(k_{j})=\begin{cases}
k_{i}^{-1}\\
k_{j}k_{i}\\
k_{j}
\end{cases},\quad T_{i}(e_{j})=\begin{cases}
-f_{i}k_{i}\\{}
[e_{i},e_{j}]_{\mathbf{q}_{\epsilon s_{i}}(\alpha_{i},\alpha_{j})}\\
e_{j}
\end{cases},\quad T_{i}(f_{j})=\begin{cases}
-k_{i}^{-1}e_{i} & \text{if }j=i,\\{}
[f_{j},f_{i}]_{\mathbf{q}_{\epsilon s_{i}}(\alpha_{i},\alpha_{j})^{-1}} & \text{if }j=i\pm1,\\
f_{j} & \text{otherwise}
\end{cases}\label{eq:def-braid-action}
\end{equation}
which satisfies the type $A$ braid relations:
\begin{align*}
T_{i}T_{j} & =T_{j}T_{i}:\mathcal{U}(\epsilon)\longrightarrow\mathcal{U}(\epsilon s_{i}s_{j})\quad\quad\text{if }\left|i-j\right|>1,\\
T_{i}T_{i+1}T_{i} & =T_{i+1}T_{i}T_{i+1}:\mathcal{U}(\epsilon)\longrightarrow\mathcal{U}(\epsilon s_{i}s_{i+1}s_{i}).
\end{align*}
The inverse $T_{i}^{-1}:\mathcal{U}(\epsilon s_{i})\longrightarrow\mathcal{U}(\epsilon)$
is given by
\[
T_{i}^{-1}(k_{j})=\begin{cases}
k_{i}^{-1}\\
k_{j}k_{i}\\
k_{j}
\end{cases},\quad T_{i}^{-1}(e_{j})=\begin{cases}
-k_{i}^{-1}f_{i}\\{}
[e_{j},e_{i}]_{\mathbf{q}_{\epsilon}(\alpha_{i},\alpha_{j})}\\
e_{j}
\end{cases},\quad T_{i}^{-1}(f_{j})=\begin{cases}
-e_{i}k_{i} & \text{if }j=i,\\{}
[f_{i},f_{j}]_{\mathbf{q}_{\epsilon}(\alpha_{i},\alpha_{j})^{-1}} & \text{if }j=i\pm1,\\
f_{j} & \text{otherwise}.
\end{cases}
\]
Moreover, $T_{i}^{\pm1}$ commutes with $\Omega$.
\end{thm}

\begin{rem}
Since $T_{i}$ is not an automorphism whenever $p(i)=1$, it defines
a braid \textit{groupoid} action on the collection $\{\mathcal{U}(\epsilon)\,|\,\epsilon\text{ has \ensuremath{M} 0's and \ensuremath{N} 1's}\}$.
Below we will establish analogues for $\mathcal{U}(\epsilon)$ of
various results for quantum groups and quantum affine algebras whose
proofs rely on braid group actions. Thanks to the groupoid analogue
of Matsumoto theorem \cite[Theorem 5]{HY} and the following easy
observation, almost all arguments still work for our braid action
by keeping track of how $\epsilon$ varies.
\end{rem}

\begin{lem}
\label{lem:W-pres-q}For $\alpha,\beta\in Q$, we have $\mathbf{q}_{\epsilon}(\alpha,\beta)=\mathbf{q}_{\epsilon s_{i}}(s_{i}\alpha,s_{i}\beta)$.
\end{lem}

By the braid relation, each $w\in W$ defines without ambiguity
\[
T_{w}\coloneqq T_{i_{1}}T_{i_{2}}\cdots T_{i_{l}}:\mathcal{U}(\epsilon)\rightarrow\mathcal{U}(\epsilon w^{-1})
\]
 for any reduced expression $w=s_{i_{1}}s_{i_{2}}\cdots s_{i_{l}}$. 
\begin{prop}[{cf. \cite[Proposition 1.8 (d)]{L}}]
\label{prop:Tw-on-ef} If $i,j\in I$ and $w\in W$ is such that
$w(\alpha_{i})=\alpha_{j}$, then we have $T_{w}(e_{i})=e_{j}$ and
$T_{w}(f_{i})=f_{j}$.
\end{prop}

\begin{proof}
Following the induction argument in the proof of \cite[Proposition 1.8 (d)]{L},
we only need to check whether the following identity 
\[
T_{i}T_{i\pm1}(e_{i})=e_{i\pm1}
\]
holds in $\mathcal{U}(\epsilon)$ for any $\epsilon$. Indeed,
\begin{align*}
T_{i}T_{i-1}(e_{i}) & =T_{i}\left(e_{i-1}e_{i}-\mathbf{q}_{\epsilon s_{i}}(\alpha_{i-1},\alpha_{i})e_{i}e_{i-1}\right)\\
 & =(e_{i}e_{i-1}-\mathbf{q}_{\epsilon}(\alpha_{i},\alpha_{i-1})e_{i-1}e_{i})(-f_{i}k_{i})-q_{i+1}^{-1}(-f_{i}k_{i})(e_{i}e_{i-1}-\mathbf{q}_{\epsilon}(\alpha_{i},\alpha_{i-1})e_{i-1}e_{i})\\
 & =-q_{i}e_{i}f_{i}k_{i}e_{i-1}+q_{i}^{-1}e_{i-1}e_{i}f_{i}k_{i}+q_{i}f_{i}e_{i}k_{i}e_{i-1}-q_{i}^{-1}e_{i-1}f_{i}e_{i}k_{i}\\
 & =-q_{i}\frac{k_{i}-k_{i}^{-1}}{q-q^{-1}}k_{i}e_{i-1}+q_{i}^{-1}e_{i-1}\frac{k_{i}-k_{i}^{-1}}{q-q^{-1}}k_{i}=e_{i-1}
\end{align*}
where in the first line we compute $T_{i-1}(e_{i})$ in $\mathcal{U}(\epsilon s_{i})$
as we are verifying the identity in $\mathcal{U}(\epsilon)$. The
other one is parallel.
\end{proof}
Since $T_{i}$'s satisfy the type $A$ braid relation, this computation
also proves Lemma 40.1.1 in \cite{Lb} for $\mathcal{U}(\epsilon)$,
from which the following statement follows by induction on the length
$l(w)$ of $w\in W$.
\begin{prop}[{cf. \cite[Lemma 40.1.2]{Lb}}]
\label{prop:Lus-Lemma-40.1.2} For $w\in W$ and $i\in I$ satisfying
$l(ws_{i})=l(w)+1$, we have $T_{w}(e_{i})\in\mathcal{U}^{+}(\epsilon)$.
\end{prop}

\subsection{Lusztig's algebra $\mathbf{f}(\epsilon)$ and a bilinear form\label{subsec:Lusztig-f-bil-form}}

Let $^{\prime}\mathbf{f}(\epsilon)$ be the free associative $\Bbbk$-algebra
generated by $\theta_{i}$ for $i\in I$ with a $Q$-grading $\deg\theta_{i}=\alpha_{i}$.
We also use the notation $\left|x\right|=\mathrm{cl}\deg(x)\in Q$
for homogeneous $x\in{}^{\prime}\mathbf{f}(\epsilon)$. Define a twisted
multiplication on $^{\prime}\mathbf{f}(\epsilon)\otimes{}^{\prime}\mathbf{f}(\epsilon)$
by
\[
(x_{1}\otimes x_{2})(y_{1}\otimes y_{2})=\mathbf{q}(\left|x_{2}\right|,\left|y_{1}\right|)(x_{1}y_{1})\otimes(x_{2}y_{2})
\]
for homogeneous $x_{2}$ and $y_{1}$, and let $r:^{\prime}\mathbf{f}(\epsilon)\rightarrow{}^{\prime}\mathbf{f}(\epsilon)\otimes{}^{\prime}\mathbf{f}(\epsilon)$
be the unique algebra homomorphism defined by $r(\theta_{i})=\theta_{i}\otimes1+1\otimes\theta_{i}$.
We also define an anti-automorphism $\sigma$ by $\sigma(\theta_{i})=\theta_{i}$.

Following \cite[Proposition 1.2.3]{Lb}, there exists a unique $\Bbbk$-valued
symmetric bilinear form on $^{\prime}\mathbf{f}(\epsilon)$ satisfying
\begin{align*}
(1,1)=1 & ,\quad(\theta_{i},\theta_{j})=\delta_{ij}\frac{1}{q_{i}^{-1}-q_{i}}\\
(x,yy^{\prime})= & (r(x),y\otimes y^{\prime}),\quad(xx^{\prime},y)=(x\otimes x^{\prime},r(y))
\end{align*}
where the form is naturally extended to $^{\prime}\mathbf{f}(\epsilon)\otimes{}^{\prime}\mathbf{f}(\epsilon)$
by $(x_{1}\otimes x_{2},y_{1}\otimes y_{2})=(x_{1},y_{1})(x_{2},y_{2})$.
Note that for $x\in{}^{\prime}\mathbf{f}(\epsilon)_{\beta}$ and $y\in{}^{\prime}\mathbf{f}(\epsilon)_{\gamma}$,
$(x,y)=0$ unless $\beta=\gamma$. Moreover, we have $(\sigma(x),\sigma(y))=(x,y)$
(see \cite[Lemma 1.2.8]{Lb}).

From these properties, one can readily see that its radical $\mathcal{I}$
is a two-sided ideal of $^{\prime}\mathbf{f}(\epsilon)$ and $r(\mathcal{I})\subset\mathcal{I}\otimes{}^{\prime}\mathbf{f}(\epsilon)+{}^{\prime}\mathbf{f}(\epsilon)\otimes\mathcal{I}$,
$\sigma(\mathcal{I})\subset\mathcal{I}$. Hence $r$ and $\sigma$
descends to the quotient algebra $\mathbf{f}(\epsilon)={}^{\prime}\mathbf{f}(\epsilon)/\mathcal{I}$.
The following proposition, essentially proved in \cite[Theorem 6.8.1]{Y},
says that $\mathcal{I}$ is generated by the Serre relations in Definition
\ref{def:GQG}.
\begin{prop}[{\cite[Theorem 2.9]{KL}}]
 Whenever $M\neq N$, there exist algebra isomorphisms
\[
^{\pm}:\mathbf{f}(\epsilon)\longrightarrow\mathcal{U}^{\pm}(\epsilon),\quad\theta_{i}^{+}=e_{i},\,\theta_{i}^{-}=f_{i}.
\]
\end{prop}

\begin{proof}
The proof in \cite[Theorem 2.9]{KL} also works for the remaining
case of $n=3$ once we verify that $\theta_{i}$'s in $\mathbf{f}(\epsilon)$
satisfy the relation (\ref{eq:GQG-DJ-deg5Serre}). Namely, we have
to show the following element
\[
[\theta_{i+1},[\theta_{i},[\theta_{i+1},[\theta_{i},\theta_{i-1}]_{q_{i}}]_{q_{i+1}q_{i+2}}]]_{q_{i+1}}-[\theta_{i},[\theta_{i+1},[\theta_{i},[\theta_{i+1},\theta_{i-1}]_{q_{i-1}}]_{q_{i}q_{i+1}}]]_{q_{i+1}}
\]
is contained in the radical $\mathcal{I}$. Since $\theta_{i}^{2},\theta_{i+1}^{2}\in\mathcal{I}$,
this is equivalent to that
\begin{align*}
 & \:\theta_{i+1}\theta_{i}\theta_{i+1}\theta_{i+1}\theta_{i-1}-(q+q^{-1})\theta_{i+1}\theta_{i}\theta_{i+1}\theta_{i-1}\theta_{i}-(q+q^{-1})\theta_{i+1}\theta_{i}\theta_{i-1}\theta_{i}\theta_{i+1}\\
 & +(q+q^{-1})\theta_{i+1}\theta_{i-1}\theta_{i}\theta_{i+1}\theta_{i}+(q+q^{-1})\theta_{i}\theta_{i-1}\theta_{i+1}\theta_{i}\theta_{i-1}+\theta_{i-1}\theta_{i}\theta_{i+1}\theta_{i}\theta_{i+1}\\
 & -\theta_{i}\theta_{i+1}\theta_{i}\theta_{i+1}\theta_{i-1}+(q+q^{-1})\theta_{i}\theta_{i+1}\theta_{i-1}\theta_{i+1}\theta_{i}-\theta_{i-1}\theta_{i+1}\theta_{i}\theta_{i+1}\theta_{i}
\end{align*}
belongs to $\mathcal{I}$, which can be verified by straightforward
but tedious computation.
\end{proof}
For each $i\in I$, let $_{i}r:{}^{\prime}\mathbf{f}(\epsilon)\longrightarrow{}^{\prime}\mathbf{f}(\epsilon)$
be the $\Bbbk$-linear map uniquely defined by $_{i}r(1)=0$, $_{i}r(\theta_{j})=\delta_{ij}$
and 
\[
_{i}r(xy)={}_{i}r(x)\cdot y+\mathbf{q}(\left|x\right|,\left|y\right|)x\cdot{}_{i}r(y)
\]
for homogeneous $x,y$. Similarly, $r_{i}:{}^{\prime}\mathbf{f}(\epsilon)\longrightarrow{}^{\prime}\mathbf{f}(\epsilon)$
is defined to satisfy
\[
r_{i}(1)=0,\,r_{i}(\theta_{j})=\delta_{ij},\quad r_{i}(xy)=x\cdot r_{i}(y)+\mathbf{q}(\left|x\right|,\left|y\right|)r_{i}(x)\cdot y.
\]
Then we have 
\[
(\theta_{i}x,y)=(\theta_{i},\theta_{i})(x,{}_{i}r(y)),\quad(x\theta_{i},y)=(\theta_{i},\theta_{i})(x,r_{i}(y)).
\]
In particular, $_{i}r$ and $r_{i}$ leave $\mathcal{I}$ stable and
hence induce linear maps on $\mathbf{f}(\epsilon)$ (denoted by the
same notation).
\begin{lem}[{cf. \cite[Lemma 1.2.15]{Lb}}]
 Let $x\in\mathbf{f}(\epsilon)_{\beta}$ where $\beta\neq0$.
\begin{enumerate}
\item If $r_{i}(x)=0$ for all $i\in I$, then $x=0$.
\item If $_{i}r(x)=0$ for all $i\in I$, then $x=0$.
\end{enumerate}
\end{lem}

\begin{prop}[{cf. \cite[Proposition 3.1.6]{Lb}}]
 For $x\in\mathbf{f}(\epsilon)$ and $i\in I$, we have in $\mathcal{U}(\epsilon)$
\begin{align*}
x^{+}f_{i}-f_{i}x^{+} & =\frac{r_{i}(x)^{+}\cdot k_{i}-k_{i}^{-1}\cdot{}_{i}r(x)^{+}}{q-q^{-1}},\\
x^{-}e_{i}-e_{i}x^{-} & =\frac{r_{i}(x)^{-}\cdot k_{i}^{-1}-k_{i}\cdot{}_{i}r(x)^{-}}{q-q^{-1}}.
\end{align*}
\end{prop}

For each $i\in I$, let $\mathbf{f}(\epsilon)[i]$ be the subalgebra
of $\mathbf{f}(\epsilon)$ generated by the following elements
\[
f(i,j;m)=\sum_{r+s=m}(-1)^{r}\mathbf{q}(\alpha_{i},\alpha_{j})^{r}\mathbf{q}(\alpha_{i},\alpha_{i})^{\frac{r(m-1)}{2}}\begin{bmatrix}m\\
r
\end{bmatrix}_{\mathbf{q}(\alpha_{i},\alpha_{i})^{1/2}}\theta_{i}^{r}\theta_{j}\theta_{i}^{s},\quad j\in I\setminus\{i\},\,m\in\mathbb{Z}_{\geq0}.
\]
Here the term $\mathbf{q}(\alpha_{i},\alpha_{i})^{\frac{r(m-1)}{2}}\begin{bmatrix}m\\
r
\end{bmatrix}_{\mathbf{q}(\alpha_{i},\alpha_{i})^{1/2}}$ is a polynomial in $\mathbf{q}(\alpha_{i},\alpha_{i})$. Similarly,
we define $^{\sigma}\mathbf{f}(\epsilon)[i]$ to be the subalgebra
generated by $\sigma(f(i,j;m))$ for $j\in I\setminus\{i\}$ and $m\in\mathbb{Z}_{\ge0}$,
so that $^{\sigma}\mathbf{f}(\epsilon)[i]=\sigma(\mathbf{f}(\epsilon)[i])$.

Note that if $p(i)=0$, then $f(i,j;m)=0$ for $m\geq2$, where $f(i,j;2)=0$
is exactly the even Serre relation (see \cite[Proposition 7.1.5]{Lb},
where our $f(i,j;m)$ is $[m]_{q_{i}}!$ times $f_{i,j;1,m;-1}$ there).
On the other hand, when $p(i)=1$ we still have $f(i,j;m)=0$ as $f_{i}^{2}=0$
and $\mathbf{q}(\alpha_{i},\alpha_{i})=-1$. With these in mind, the
following statements can be proved as in \cite[Chapter 38]{Lb}.
\begin{lem}[{cf. \cite[Lemma 38.1.3]{Lb}}]
\label{lem:T_i-on-Lusztig-f} There is a unique algebra isomorphism
$g:\mathbf{f}(\epsilon)[i]\rightarrow{}^{\sigma}\mathbf{f}(\epsilon)[i]$
such that $T_{i}(x^{+})=g(x)^{+}$ for all $x\in\mathbf{f}(\epsilon)[i]$.
\end{lem}

\begin{prop}[{cf. \cite[Proposition 38.1.6]{Lb}}]
 For each $i\in I$, we have
\begin{align*}
\mathbf{f}(\epsilon)[i] & =\left\{ x\in\mathbf{f}(\epsilon)\,|\,T_{i}(x^{+})\in\mathcal{U}^{+}(\epsilon s_{i})\right\} =\left\{ x\in\mathbf{f}(\epsilon)\,|\,{}_{i}r(x)=0\right\} ,\\
^{\sigma}\mathbf{f}(\epsilon)[i] & =\left\{ x\in\mathbf{f}(\epsilon)\,|\,T_{i}^{-1}(x^{+})\in\mathcal{U}^{+}(\epsilon s_{i})\right\} =\left\{ x\in\mathbf{f}(\epsilon)\,|\,r_{i}(x)=0\right\} .
\end{align*}
\end{prop}

\begin{prop}[{cf. \cite[Proposition 38.2.1]{Lb}}]
\label{prop:braid-isometry} For any $x,y\in\mathbf{f}(\epsilon)[i]$,
we have $(g(x),g(y))=(x,y)$.
\end{prop}

\subsection{\label{subsec:Drinfeld-pres-GQG}Drinfeld-type presentation of $\mathcal{U}(\epsilon)$}

The \textit{quantum affine superalgebra} of type $A$ associated with
$\epsilon$ is a Hopf superalgebra $U(\epsilon)$ generated by $E_{i},F_{i},K_{i}^{\pm1}$
($i\in I$) and central $C^{\pm1/2}=\left(\prod_{i\in I}K_{i}\right)^{\pm1/2}$
with $\mathbb{Z}/2\mathbb{Z}$-grading $p(i),p(i),0$ and $0$ respectively,
subject to the relations
\begin{align*}
 & K_{i}K_{i}^{-1}=K_{i}^{-1}K_{i}=1,\quad K_{i}K_{j}=K_{j}K_{i},\\
 & K_{i}E_{j}K_{i}^{-1}=q^{(\alpha_{i},\alpha_{j})}E_{j},\quad K_{i}F_{j}K_{i}^{-1}=q^{-(\alpha_{i},\alpha_{j})}F_{j},\\
 & E_{i}F_{j}-(-1)^{p(i)p(j)}F_{j}E_{i}=\delta_{ij}(-1)^{\epsilon_{i}}\frac{K_{i}-K_{i}^{-1}}{q-q^{-1}},\\
 & [E_{i},E_{j}]_{q^{-(\alpha_{i},\alpha_{j})}}=[F_{i},F_{j}]_{q^{-(\alpha_{i},\alpha_{j})}}=0\quad\quad\text{if }j\neq i\ensuremath{\pm1},
\end{align*}
and analogues of (\ref{eq:GQG-DJ-evenSerre}--\ref{eq:GQG-DJ-deg5Serre})
(see \cite[Proposition 6.7.1]{Y}). Here $[x,y]_{u}=xy-(-1)^{p(x)p(y)}uyx$
denotes the supercommutator only in the relations of $U(\epsilon)$
and another superalgebra $U^{\mathrm{Dr}}(\epsilon)$ defined as follows:
\begin{itemize}
\item generators: $X_{i,k}^{\pm}$, $K_{i}^{\pm1}$, $H_{i,r}$ for $i\in\mathring{I}$,
$k\in\mathbb{Z}$, $r\in\mathbb{Z}\setminus\{0\}$ and $C^{\pm1/2}=\left(\prod_{i\in I}K_{i}\right)^{\pm1/2}$,
\item $\mathbb{Z}/2\mathbb{Z}$-grading: $p(X_{i,k}^{\pm})=p(i)$, $p(K_{i}^{\pm1})=p(H_{i,s})=p(C^{1/2})=0$, 
\item relations:
\begin{align*}
K_{i}K_{i}^{-1}= & K_{i}^{-1}K_{i}=1,\,[K_{i},K_{j}]=[K_{i},H_{j,r}]=0,\,C^{\pm1/2}\text{\text{ is central}},\\
K_{i}X_{j,k}^{\pm}K_{i}^{-1} & =q^{(\alpha_{i},\pm\alpha_{j})}X_{j,k}^{\pm},\\{}
[H_{i,r},H_{j,s}] & =\delta_{r,-s}\frac{q^{r(\alpha_{i},\alpha_{j})}-q^{-r(\alpha_{i},\alpha_{j})}}{r(q-q^{-1})}\cdot\frac{C^{r}-C^{-r}}{q-q^{-1}}\\{}
[H_{i,r},X_{j,k}^{\pm}] & =\pm\frac{q^{r(\alpha_{i},\alpha_{j})}-q^{-r(\alpha_{i},\alpha_{j})}}{r(q-q^{-1})}C^{\pm\left|r\right|/2}X_{j,r+k}^{\pm},\\{}
[X_{i,k}^{+},X_{j,l}^{-}] & =(-1)^{\epsilon_{i}}\delta_{ij}\frac{C^{(k-l)/2}\phi_{i,k+l}^{+}-C^{(l-k)/2}\phi_{i,k+l}^{-}}{q-q^{-1}},\\
X_{i,k+1}^{\pm}X_{j,l}^{\pm}-(-1)^{p(i)p(j)}q^{\pm(\alpha_{i},\alpha_{j})} & X_{j,l}^{\pm}X_{i,k+1}^{\pm}=(-1)^{p(i)p(j)}q^{\pm(\alpha_{i},\alpha_{j})}X_{i,k}^{\pm}X_{j,l+1}^{\pm}-X_{j,l+1}^{\pm}X_{i,k}^{\pm},\\{}
[X_{i,k}^{\pm},X_{j,l}^{\pm}] & =0\quad\quad\text{if }(\alpha_{i},\alpha_{j})=0,\\{}
[X_{i,k_{1}}^{\pm},[X_{i,k_{2}}^{\pm},X_{j,l}^{\pm}]_{q^{-(\alpha_{i},\alpha_{j})}}]_{q^{-(\alpha_{i},\alpha_{i}+\alpha_{j})}} & +[X_{i,k_{2}}^{\pm},[X_{i,k_{1}}^{\pm},X_{j,l}^{\pm}]_{q^{-(\alpha_{i},\alpha_{j})}}]_{q^{-(\alpha_{i},\alpha_{i}+\alpha_{j})}}=0\\
 & \hspace{1em}\quad\quad\quad\quad\quad\quad\quad\quad\quad\text{if }p(i)=0\text{ and }j=i\pm1,\\{}
[[[X_{i-1,l}^{\pm},X_{i,k_{1}}^{\pm}]_{q^{\epsilon_{i}}},X_{i+1,m}^{\pm}]_{q^{\epsilon_{i+1}}},X_{i,k_{2}}^{\pm}] & +[[[X_{i-1,l}^{\pm},X_{i,k_{2}}^{\pm}]_{q^{\epsilon_{i}}},X_{i+1,m}^{\pm}]_{q^{\epsilon_{i+1}}},X_{i,k_{1}}^{\pm}]=0\quad\text{if }p(i)=1,
\end{align*}
where $\phi_{i,k}^{\pm}$ ($i\in\mathring{I}$, $k\in\mathbb{Z}$)
is defined by the equation
\begin{align*}
\sum_{k\in\mathbb{Z}}\phi_{i,k}^{\pm}z^{k} & =K_{i}^{\pm1}\exp\left(\pm(q-q^{-1})\sum_{r>0}H_{i,\pm r}z^{\pm r}\right).
\end{align*}
\end{itemize}
\begin{rem}
For the presentation of $U(\epsilon)$ we are following \cite[Definition 2.5]{KL},
which recovers the one in \cite{Y} by the rescaling $E_{i}\mapsto(-1)^{\epsilon_{i}}E_{i}$
with the other generators fixed. The formulas in this article including
the presentation of $U^{\mathrm{Dr}}(\epsilon)$ are adapted to our
convention. For instance, the relations in \cite[Section 8]{Y} are
recovered by $X_{i,k}^{+}\mapsto(-1)^{\epsilon_{i}}X_{i,k}^{+}$ (other
generators fixed). On the other hand, compared with \cite[Definition 3.1]{Z1}
where $\epsilon$ is chosen to be the standard one $\epsilon_{M|N}$,
our generator $H_{i,r}$ of $U^{\mathrm{Dr}}(\epsilon)$ is identified
with $(-1)^{\epsilon_{i}}h_{i,r}$ therein and the others are just
the same.
\end{rem}

The algebra $U^{\mathrm{Dr}}(\epsilon)$ is called the \textit{Drinfeld
presentation} of $U(\epsilon)$ and used in the study of its finite-dimensional
representations \cite{Z1}. More precisely, a surjective algebra map
$B:U^{\mathrm{Dr}}(\epsilon)\longrightarrow U(\epsilon)$ was constructed
in \cite[Section 8]{Y} along the lines of \cite{B2}. It is recently
shown \cite{LYZ} to be an isomorphism whenever
\begin{equation}
p(i)p(i+1)=0\quad\text{for all }i=1,2,\dots,n-1,\label{eq:LYZ-cond}
\end{equation}
equivalently $\epsilon$ does not contain a subsequence $(010)$ or
$(101)$ (regarding $\epsilon_{n+1}=\epsilon_{1}$). To develop the
representation theory for generalized quantum groups in a parallel
manner, we first establish the following Drinfeld presentation for
$\mathcal{U}(\epsilon)$ and related it to $U^{\mathrm{Dr}}(\epsilon)$. 
\begin{thm}
\label{thm:GQG-Drinfeld}Let $\mathcal{U}^{\mathrm{Dr}}(\epsilon)$
be the $\Bbbk$-algebra defined by the following presentation:
\begin{itemize}
\item generators: $x_{i,k}^{\pm}$, $k_{i}^{\pm1}$, $h_{i,r}$ for $i\in\mathring{I}$,
$k\in\mathbb{Z}$, $r\in\mathbb{Z}\setminus\{0\}$ and $c^{\pm1/2}=\left(\prod_{i\in I}k_{i}\right)^{\pm1/2}$,
\item relations:
\begin{align}
k_{i}k_{i}^{-1}= & k_{i}^{-1}k_{i}=1,\,[k_{i},k_{j}]=[k_{i},h_{j,r}]=0,\,c^{\pm1/2}\text{\text{ is central}},\label{eq:GQG-Drinfeld-hh}\\
k_{i}x_{j,k}^{\pm}k_{i}^{-1} & =\mathbf{q}(\alpha_{i},\pm\alpha_{j})x_{j,k}^{\pm},\\{}
[h_{i,r},h_{j,s}] & =\delta_{r,-s}\frac{\mathbf{q}(\alpha_{i},\alpha_{j})^{r}-\mathbf{q}(\alpha_{i},\alpha_{j})^{-r}}{r(q-q^{-1})}\cdot\frac{c^{r}-c^{-r}}{q-q^{-1}}\label{eq:GQG-Drinfeld-q-Heis}\\{}
[h_{i,r},x_{j,k}^{\pm}] & =\pm\frac{\mathbf{q}(\alpha_{i},\alpha_{j})^{r}-\mathbf{q}(\alpha_{i},\alpha_{j})^{-r}}{r(q-q^{-1})}c^{\pm\left|r\right|/2}x_{j,r+k}^{\pm},\label{eq:GQG-Drinfeld-hx}\\{}
[x_{i,k}^{+},x_{j,l}^{-}] & =\delta_{ij}\frac{c^{(k-l)/2}\psi_{i,k+l}^{+}-c^{(l-k)/2}\psi_{i,k+l}^{-}}{q-q^{-1}},\label{eq:GQG-Drinfeld-x+x-}\\
x_{i,k+1}^{\pm}x_{j,l}^{\pm}-\mathbf{q}(\alpha_{i},\alpha_{j})^{\pm1} & x_{j,l}^{\pm}x_{i,k+1}^{\pm}=\mathbf{q}(\alpha_{i,}\alpha_{j})^{\pm1}x_{i,k}^{\pm}x_{j,l+1}^{\pm}-x_{j,l+1}^{\pm}x_{i,k}^{\pm},\label{eq:GQG-Drinfeld-q-locality}\\{}
[x_{i,k}^{\pm},x_{j,l}^{\pm}]_{\mathbf{q}(\alpha_{i},\alpha_{j})} & =0\quad\quad\text{if }(\alpha_{i},\alpha_{j})=0,\label{eq:GQG-Drinfeld-quadSerre}\\{}
[x_{i,k_{1}}^{\pm},[x_{i,k_{2}}^{\pm},x_{j,l}^{\pm}]_{\mathbf{q}(\alpha_{i},\alpha_{j})^{-1}}]_{\mathbf{q}(\alpha_{i},\alpha_{i}+\alpha_{j})^{-1}} & +[x_{i,k_{2}}^{\pm},[x_{i,k_{1}}^{\pm},x_{j,l}^{\pm}]_{\mathbf{q}(\alpha_{i},\alpha_{j})^{-1}}]_{\mathbf{q}(\alpha_{i},\alpha_{i}+\alpha_{j})^{-1}}=0\label{eq:GQG-Drinfeld-eSerre}\\
 & \hspace{1em}\quad\quad\quad\quad\quad\quad\quad\quad\quad\text{if }p(i)=0\text{ and }j=i\pm1,\\{}
[[[x_{i-1,l}^{\pm},x_{i,k_{1}}^{\pm}]_{q_{i}},x_{i+1,m}^{\pm}]_{q_{i+1}},x_{i,k_{2}}^{\pm}] & +[[[x_{i-1,l}^{\pm},x_{i,k_{2}}^{\pm}]_{q_{i}},x_{i+1,m}^{\pm}]_{q_{i+1}},x_{i,k_{1}}^{\pm}]=0\quad\text{if }p(i)=1,\label{eq:GQG-Drinfeld-oSerre}
\end{align}
 where $\psi_{i,k}^{\pm}$ is given by the following equation
\begin{align}
\sum_{k\in\mathbb{Z}}\psi_{i,\pm k}^{\pm}z^{\pm k} & =k_{i}^{\pm1}\exp\left(\pm(q-q^{-1})\sum_{r>0}h_{i,\pm r}z^{\pm r}\right).\label{eq:GQG-psi-halfVO}
\end{align}
\end{itemize}
Adjoining a central element $c^{\pm1/2}$ to $\mathcal{U}(\epsilon)$,
there exists a surjective algebra map $\mathcal{B}:\mathcal{U}^{\mathrm{Dr}}(\epsilon)\rightarrow\mathcal{U}(\epsilon)$
which is also an isomorphism whenever $B:U^{\mathrm{Dr}}(\epsilon)\rightarrow U(\epsilon)$
is so.
\end{thm}

We will call the generators in this presentation (and also in $U^{\mathrm{Dr}}(\epsilon)$)
the \textit{loop generators}. The construction of the surjective algebra
map $\mathcal{B}$ is completely parallel with \cite[Section 8]{Y}
and summarized in Appendix \ref{sec:GQG-Drinfeld-constr}, focusing
on presenting precise definitions and formulas for $\mathcal{U}(\epsilon)$.
\begin{rem}
\label{rem:logical-independence}We remark that almost every results
below do not rely on the injectivity of $\tau$. Specifically, it
will be used in triangular decomposition with respect to loop generators
(\ref{eq:GQGDr-tri-decomp}) and in the classification of finite-dimensional
simple $\mathcal{U}(\epsilon)$-modules in Section \ref{subsec:classification-irrep}.
For the other parts of this article, whenever loop generators appear
in the statement, it is enough to identify them with their images
under $\mathcal{B}$ so that they satisfy the above relations in $\mathcal{U}(\epsilon)$.
\end{rem}

Let us recall a comparison result between $\mathcal{U}(\epsilon)$
and $U(\epsilon)$ from \cite[Section 2.3]{KL}. Introduce the commutative
$\Bbbk$-bialgebra $\Sigma$ generated by $\sigma_{i}$ for $i\in\mathbb{I}_{1}$
with $\sigma_{i}^{2}=1$, $\Delta(\sigma_{i})=\sigma_{i}\otimes\sigma_{i}$,
which acts on $\mathcal{U}(\epsilon)$ by
\[
\sigma_{i}e_{j}=(-1)^{(\delta_{i},\alpha_{j})}e_{j},\quad\sigma_{i}f_{j}=(-1)^{-(\delta_{i},\alpha_{j})}f_{j},\quad\sigma_{i}k_{j}=k_{j},
\]
 and in the same way on $U(\epsilon)$. This makes $\mathcal{U}(\epsilon)$
and $U(\epsilon)$ $\Sigma$-module algebras and so defines semidirect
products 
\[
\mathcal{U}(\epsilon)[\sigma]=\mathcal{U}(\epsilon)\rtimes\Sigma,\quad U(\epsilon)[\sigma]=U(\epsilon)\rtimes\Sigma.
\]
We also let $\sigma_{i}=1$ for $i\in\mathbb{I}_{0}$, serving as
a placeholder. 

Given $\epsilon$, there exists a unique sequence $1\leq i_{1}<i_{2}<\cdots<i_{l}\leq n$
such that 
\[
\epsilon_{i_{k}-1}\neq\epsilon_{i_{k}}=\epsilon_{i_{k+1}}=\cdots=\epsilon_{i_{k+1}-1}\neq\epsilon_{i_{k+1}}\quad\text{for }1\leq k\leq l,
\]
where we understand the subscripts modulo $n$ and we set $i_{l+1}\coloneqq i_{1}$.
Put $\sigma_{\leq i}=\sigma_{1}\sigma_{2}\cdots\sigma_{j}.$ To each
of generators $e_{i}$, $f_{i}$ and $k_{i}$ ($i\in I$) of $\mathcal{U}(\epsilon)$,
we assign elements $\sigma(e_{i})$, $\sigma(f_{i})$ and $\sigma(k_{i})$
of $\Sigma$ as follows. First we put $\sigma(k_{i}^{\pm1})=\sigma_{i}\sigma_{i+1}$.
The others are given in the case-by-case manner:
\begin{enumerate}
\item For $i\in I_{\mathrm{even}}$ with $(\epsilon_{i},\epsilon_{i+1})=(0,0)$,
we set $\sigma(e_{i})=\sigma_{i}$ and $\sigma(f_{i})=\sigma_{i+1}$.
\item For $i\in I_{\mathrm{odd}}$ with $(\epsilon_{i},\epsilon_{i+1})=(0,1)$,
\[
\sigma(e_{i})=\sigma_{\leq i},\quad\sigma(f_{i})=\sigma_{\leq i}\sigma_{i}\sigma_{i+1}.
\]
\item For $i\in I_{\mathrm{even}}$ with $(\epsilon_{i},\epsilon_{i+1})=(1,1)$,
we have $i\in\{i_{t},i_{t}+1,\dots,i_{t+1}-2\}\subset\mathbb{I}$
for some $t$. When $t<l$, we set
\[
\sigma(e_{i})=(-\sigma_{i}\sigma_{i+1})^{i-i_{t}+1},\quad\sigma(f_{i})=(-\sigma_{i}\sigma_{i+1})^{i-i_{t}}
\]
and when $t=l$,
\[
\sigma(e_{i})=\begin{cases}
(-\sigma_{i}\sigma_{i+1})^{i-i_{l}+1}\\
(-\sigma_{i}\sigma_{i+1})^{i-(i_{l}-n)+1},
\end{cases},\quad\sigma(f_{i})=\begin{cases}
(-\sigma_{i}\sigma_{i+1})^{i-i_{l}} & \text{if }i_{l}\leq i\leq n,\\
(-\sigma_{i}\sigma_{i+1})^{i-(i_{l}-n)} & \text{if }1\leq i\leq i_{1}-2.
\end{cases}
\]
\item For $i\in I_{\mathrm{odd}}$ with $(\epsilon_{i},\epsilon_{i+1})=(1,0)$,
we have $i=i_{t+1}-1$ for some $t$ and set
\[
\sigma(e_{i})=\begin{cases}
\sigma_{\leq i}(-\sigma_{i}\sigma_{i+1})^{i_{t+1}-i_{t}}\\
\sigma_{\leq i}(-\sigma_{i}\sigma_{i+1})^{i_{1}-(i_{l}-n)}
\end{cases},\quad\sigma(f_{i})=\begin{cases}
\sigma_{\leq i}(\sigma_{i}\sigma_{i+1})^{i_{t+1}-i_{t}+1} & \text{if }t<l,\\
\sigma_{\leq i}(\sigma_{i}\sigma_{i+1})^{i_{1}-(i_{l}-n)+1} & \text{if }t=l.
\end{cases}
\]
\end{enumerate}
\begin{example}
To $\epsilon=(11001101)$ corresponds a sequence $i_{1}=3<5<7<8=i_{4}$,
and then $\tau(x_{i})$ ($x=e,f$) is given as in the table below
(recall $\sigma_{i}=1$ for $i\in\mathbb{I}_{0}$).
\begin{center}
\begin{tabular}{|c|c|c|c|c|c|c|c|c|}
\hline 
$i\in I$ & $1$ & $2$ & $3$ & $4$ & $5$ & $6$ & $7$ & $8=0$\tabularnewline
\hline 
\hline 
$\sigma(e_{i})$ & $1$ & $-\sigma_{1}\sigma_{3}$ & $\sigma_{3}$ & $\sigma_{\leq4}$ & $-\sigma_{5}\sigma_{6}$ & $\sigma_{\leq6}$ & $\sigma_{\leq7}$ & $-\sigma_{8}\sigma_{1}$\tabularnewline
\hline 
$\sigma(f_{i})$ & $-\sigma_{1}\sigma_{2}$ & $\sigma_{\leq2}$ & $\sigma_{4}$ & $\sigma_{\leq3}\sigma_{5}$ & $1$ & $\sigma_{\leq5}\sigma_{7}$ & $\sigma_{\leq6}\sigma_{8}$ & $1$\tabularnewline
\hline 
\end{tabular} 
\par\end{center}

\end{example}

Now we let $\tau(x_{i})=X_{i}\sigma(x_{i})$ for $i\in I$ and $x=e,f,k^{\pm1}$,
$X=E,F,K^{\pm1}$ respectively and $\tau(c^{1/2})=C^{1/2}$, $\tau(\sigma_{i})=\sigma_{i}$.
This defines a well-defined algebra isomorphism $\tau:\mathcal{U}(\epsilon)[\sigma]\longrightarrow U(\epsilon)[\sigma]$,
whose inverse is also given as $\tau^{-1}(X_{i})=x_{i}\sigma(x_{i})$
\cite[Theorem 2.7]{KL}.
\begin{rem}
\label{rem:tau-pres-q-bracket}The map $\tau$ is devised to preserve
$q$-brackets on $U(\epsilon)$ and $\mathcal{U}(\epsilon)$. Namely,
for homogeneous $x,y\in\mathcal{U}^{+}(\epsilon)$ of $Q$-degree
$\beta,\gamma$, if we write $\tau(x)=X\sigma,\,\tau(Y)=Y\sigma^{\prime}$
for some $X,Y\in U^{+}(\epsilon)$ and monomials $\sigma,\sigma^{\prime}$
in $\sigma_{i}$'s, then we have 
\[
\tau(xy-\mathbf{q}(\beta,\gamma)^{\pm1}yx)=s\left(XY-(-1)^{p(\beta)p(\gamma)}q^{\pm(\beta,\gamma)}YX\right)\sigma\sigma^{\prime},
\]
where $s\in\{\pm1\}$ is determined from $\sigma Y=sY\sigma$. Since
the relations (\ref{eq:GQG-DJ-quadSerre}-\ref{eq:GQG-DJ-deg5Serre})
in $\mathcal{U}^{\pm}(\epsilon)$ and the corresponding ones in $U^{\pm}(\epsilon)$
are given in terms of $q$-brackets, they match well under $\tau$.
The other relations can be checked by case-by-case manner.
\end{rem}

The remark also allows us to compare through $\tau$ the image of
loop generators under $\mathcal{B}$ and $B$, using their realizations
as iterated $q$-brackets. Although the latter will be explained in
Section \ref{subsec:Pairing-imag-rvector}, for the sake of coherence
we assume the results there (which are independent of the remaining
contents of this section, see Remark \ref{rem:logical-independence})
and continue our discussion. 

We define the $q$-bracket for $\mathcal{U}(\epsilon)$ (resp. $U(\epsilon)$)
as in (\ref{eq:q-bracket-def}) (resp. (\ref{eq:q-bracket-def-qasa}))
below. According to (\ref{eq:ind-formula-xir}) and Lemma \ref{lem:formula-x-_i,1}
(resp. Remark \ref{rem:q-bracket-loop-gen-qasa}), the image of loop
generators $x_{i,k}^{\pm}$ and $\psi_{i,r}^{\pm}$ for $\mathcal{U}(\epsilon)$
under $\mathcal{B}$ and of $X_{i,k}^{\pm}$ and $\phi_{i,r}^{\pm}$
for $U(\epsilon)$ under $B$ can be described as iterated $q$-brackets
of Chevalley generators (see Remark \ref{rem:compare-rootv-Drinfeld}
to translate affine root vectors into loop generators). Therefore,
$\tau$ maps $x_{i,k}^{\pm}$ and $\psi_{i,r}^{\pm}$ to $X_{i,k}^{\pm}$
and $\phi_{i,r}^{\pm}$ respectively up to non-zero scalar multiples
and $\sigma_{i}$-factors. Since we have precise formulas, it is possible
to compare them exactly in specific cases, and we record here one
of them for a later use.
\begin{prop}
\label{prop:compare-Cartan-current}For $\epsilon=\epsilon_{M|N}$,
we have $\tau(\psi_{i,r}^{\pm})=(-1)^{rn+r(M+i+1)\epsilon_{i}}\phi_{i,r}^{\pm}\sigma_{i}\sigma_{i+1}$.
\end{prop}

Now assume $B$ is an isomorphism and consider the map $\tau^{\mathrm{Dr}}=B^{-1}\circ\tau\circ\mathcal{B}:\mathcal{U}^{\mathrm{Dr}}(\epsilon)\longrightarrow U^{\mathrm{Dr}}(\epsilon)$.
By the observation above, there exists $c_{i,k}^{\pm},d_{i,r}^{\pm}\in\Bbbk^{\times}$
such that 
\[
\tau^{\mathrm{Dr}}:x_{i,k}^{\pm}\longmapsto c_{i,k}^{\pm}X_{i,k}^{\pm}\sigma(k\delta\pm\alpha_{i}),\quad\psi_{i,r}^{\pm}\longmapsto d_{i,r}^{\pm}\phi_{i,r}^{\pm}\sigma(r\delta)\:(r\neq0)
\]
where for $\beta=\sum_{i\in I}c_{i}\alpha_{i}\in Q_{+}$ (resp. $\beta\in Q_{-}$),
we define $\sigma(\beta)=\prod\sigma(e_{i})^{c_{i}}$ (resp. $\sigma(\beta)=\prod\sigma(f_{i})^{d_{i}}$).
Since $\tau^{\mathrm{Dr}}$ is a well-defined algebra homomorphism,
we know it respects the defining relations. But this also implies
that if we simply define its inverse $U^{\mathrm{Dr}}(\epsilon)\longrightarrow\mathcal{U}^{\mathrm{Dr}}(\epsilon)$
\[
X_{i,k}^{\pm}\longmapsto(c_{i,k}^{\pm})^{-1}x_{i,k}^{\pm}\sigma(k\delta\pm\alpha_{i})^{-1},\quad\phi_{i,r}^{\pm}\longmapsto(d_{i,r}^{\pm})^{-1}\psi_{i,r}^{\pm}\sigma(r\delta)^{-1}
\]
it is a well-defined algebra homomorphism. Hence, we conclude $\mathcal{B}$
is an isomorphism as well.

Thanks to the comparison isomorphism, we may lift various results
on quantum affine superalgebras to our generalized quantum groups.
Let us record here two of them which will be used in the later sections
on representation theory of $\mathcal{U}(\epsilon)$. First, we obtain
another triangular decomposition of $\mathcal{U}(\epsilon)$ which
is more suitable for the study of finite-dimensional $\mathcal{U}(\epsilon)$-modules.
Let $\mathcal{U}^{\mathrm{Dr},\pm}(\epsilon)$ (resp. $\mathcal{U}^{\mathrm{Dr},0}(\epsilon)$)
be the subalgebra of $\mathcal{U}(\epsilon)$ generated by $x_{i,k}^{\pm}$
for $i\in\mathring{I}$ and $k\in\mathbb{Z}$ (resp. $\psi_{i,\pm r}^{\pm}$
for $i\in\mathring{I}$ and $r\geq0$).
\begin{cor}
Assume $\mathcal{B}$ is an isomorphism. Then we have an isomorphism
of vector spaces
\begin{equation}
\mathcal{U}(\epsilon)\cong\mathcal{U}^{\mathrm{Dr},-}(\epsilon)\otimes\mathcal{U}^{\mathrm{Dr},0}(\epsilon)\otimes\mathcal{U}^{\mathrm{Dr},+}(\epsilon)\label{eq:GQGDr-tri-decomp}
\end{equation}
given by the multiplication.
\end{cor}

\begin{proof}
The claim is equivalent to that the multiplication also gives rise
to a $\Bbbk$-vector space isomorphism
\[
\mathcal{U}(\epsilon)[\sigma]\cong\mathcal{U}^{\mathrm{Dr},-}(\epsilon)[\sigma]\otimes_{\Sigma}\mathcal{U}^{\mathrm{Dr},0}(\epsilon)[\sigma]\otimes_{\Sigma}\mathcal{U}^{\mathrm{Dr},+}(\epsilon)[\sigma],
\]
and the latter follows from the fact that $\tau$ restricts to isomorphisms
between $\mathcal{U}^{\mathrm{Dr},*}(\epsilon)[\sigma]$ and $U^{\mathrm{Dr},*}(\epsilon)[\sigma]$
($*=+,-,0$) and the corresponding triangular decomposition of $U^{\mathrm{Dr}}(\epsilon)$
\cite[Theorem 3.3]{Z1}.
\end{proof}
As another application, we obtain the evaluation homomorphism for
$\mathcal{U}(\epsilon)$ which will allow us to explicitly describe
certain $\mathcal{U}(\epsilon)$-modules. Consider the subalgebra
of $U(\epsilon)$ generated by $E_{i}$, $F_{i}$ and $K_{i}^{\pm1}$
for $\mathring{I}$, and let $\mathring{U}(\epsilon)$ denote its
extension by adding $K_{\delta_{i}}^{\pm1}$ ($i\in\mathbb{I}$) subject
to relations
\[
K_{i}=K_{\delta_{i}}K_{\delta_{i+1}}^{-1},\quad K_{\delta_{i}}E_{j}K_{\delta_{i}}^{-1}=q^{(\delta_{i},\alpha_{j})}E_{j},\quad K_{\delta_{i}}F_{j}K_{\delta_{i}}^{-1}=q^{-(\delta_{i},\alpha_{j})}F_{j}.
\]
The algebra $\mathring{U}(\epsilon)$ is the quantum group associated
with the Lie superalgebra $\mathfrak{gl}_{M|N}$ and $\epsilon$.
Similarly we introduce $\mathring{\mathcal{U}}(\epsilon)$ by introducing
$k_{\delta_{i}}^{\pm1}$ with
\[
k_{i}=k_{\delta_{i}}k_{\delta_{i+1}}^{-1},\quad k_{\delta_{i}}e_{j}k_{\delta_{i}}^{-1}=\mathbf{q}(\delta_{i},\alpha_{j})e_{j},\quad k_{\delta_{i}}f_{j}k_{\delta_{i}}^{-1}=\mathbf{q}(\delta_{i},\alpha_{j})^{-1}f_{j}.
\]
For the standard sequence $\epsilon=\epsilon_{M|N}$ (whose associated
sequence in the definition of $\tau$ is $1=i_{1}<i_{2}=M+1$), $\tau:\mathcal{U}(\epsilon)[\sigma]\longrightarrow U(\epsilon)[\sigma]$
is given by (recall $\sigma_{i}=1$ for $i\leq M$)
\[
\tau(k_{i})=K_{i}\sigma_{i}\sigma_{i+1},\quad\tau(e_{i})=\begin{cases}
E_{i}\\
E_{M}\\
E_{i}(-\sigma_{i}\sigma_{i+1})^{i-M}\\
E_{0}\sigma_{\leq n}(-\sigma_{n})^{N}
\end{cases},\quad\tau(f_{i})=\begin{cases}
F_{i} & \text{if }1\leq i<M\\
F_{M}\sigma_{M+1} & \text{if }i=M\\
F_{i}(-\sigma_{i}\sigma_{i+1})^{i-M-1} & \text{if }M\leq i<n\\
F_{0}\sigma_{\leq n}\sigma_{n}^{N-1} & \text{if }i=n=0.
\end{cases}
\]
It first restricts to the subalgebras above and then extends to an
isomorphism $\mathring{\tau}:\mathring{\mathcal{U}}(\epsilon)[\sigma]\longrightarrow\mathring{U}(\epsilon)[\sigma]$
by $\mathring{\tau}(k_{\delta_{i}})=K_{\delta_{i}}\sigma_{i}.$ 
\begin{prop}
\label{prop:eval-hom}For $\epsilon=\epsilon_{M|N}$, there exists
an algebra homomorphism $\mathrm{ev}_{a}:\mathcal{U}(\epsilon)\longrightarrow\mathring{\mathcal{U}}(\epsilon)$
defined by
\begin{align*}
 & \mathrm{ev}_{a}(x_{j})=x_{j}\quad\text{for }x=e,f,k\text{ and }j\in\mathring{I},\\
 & \mathrm{ev}_{a}(k_{0})=k_{1}^{-1}\cdots k_{n-1}^{-1}=k_{\delta_{n}}k_{\delta_{1}}^{-1},\\
 & \mathrm{ev}_{a}(e_{0})=a(-q)^{N-M}[[\cdots[f_{1},f_{2}]_{q_{2}},\dots]_{q_{n-2}},f_{n-1}]_{q_{n-1}}k_{\delta_{1}}k_{\delta_{n}},\\
 & \mathrm{ev}_{a}(f_{0})=a^{-1}k_{\delta_{1}}^{-1}k_{\delta_{n}}^{-1}[[\cdots[e_{1},e_{2}]_{q_{2}},\cdots]_{q_{n-2}},e_{n-1}]_{q_{n-1}}.
\end{align*}
\end{prop}

\begin{proof}
According to \cite[Proposition 5.6]{Z1}, there exists an algebra
homomorphism $\mathrm{ev}_{a}^{Z}:U(\epsilon)\longrightarrow\mathring{U}(\epsilon)$
defined for each $a\in\Bbbk^{\times}$ , and we set 
\[
\mathrm{ev}_{a}=\mathring{\tau}^{-1}\circ\mathrm{ev}_{a}^{Z}\circ\tau:\mathcal{U}(\epsilon)[\sigma]\longrightarrow\mathring{\mathcal{U}}(\epsilon)[\sigma]
\]
 where we let $\mathrm{ev}_{a}^{Z}(\sigma_{i})=\sigma_{i}$. The formulas
in the statement can be read from those for $\mathrm{ev}_{a}^{Z}$,
from which we also observe that the image of $\mathcal{U}(\epsilon)\subset\mathcal{U}(\epsilon)[\sigma]$
does not involve $\Sigma$.
\end{proof}

\section{A multiplicative formula for the universal $R$-matrix\label{sec:mult-formula-R}}

\subsection{Hopf pairing and universal $R$-matrix}

There exists a Hopf pairing on $\mathcal{U}(\epsilon)$ (\textit{cf}.
\cite[Proposition 6.2.1]{Y}), namely a bilinear form between $\mathcal{U}^{\geq0}(\epsilon)=\mathcal{U}^{+}(\epsilon)\mathcal{U}^{0}(\epsilon)$
and $\mathcal{U}^{\leq0}(\epsilon)=\mathcal{U}^{0}(\epsilon)\mathcal{U}^{-}(\epsilon)$
\[
(\,\cdot\,,\,\cdot\,):\mathcal{U}^{\geq0}(\epsilon)\otimes\mathcal{U}^{\leq0}(\epsilon)\longrightarrow\Bbbk
\]
satisfying
\begin{align}
 & (x,y_{1}y_{2})=(\Delta(x),y_{1}\otimes y_{2}),\quad(x_{1}x_{2},y)=(x_{2}\otimes x_{1},\Delta(y)),\label{eq:Hopf-pair-comult}\\
 & (e_{i},f_{j})=\frac{\delta_{ij}}{q^{-1}-q},\quad(k_{i},k_{j})=\mathbf{q}(\alpha_{i},\alpha_{j})^{-1},\quad(k_{i}^{\pm1},f_{j})=(e_{i},k_{j}^{\pm1})=0,\\
 & (xk_{i},y)=(x,y)=(x,yk_{i})\quad^{\forall}x\in\mathcal{U}^{+},\,{}^{\forall}y\in\mathcal{U}^{-}.
\end{align}
where we let $(x_{1}\otimes x_{2},y_{1}\otimes y_{2})=(x_{1},y_{1})(x_{2},y_{2})$.
It is related to the bilinear form on $\mathbf{f}(\epsilon)$ by the
following simple observation.
\begin{lem}
\label{lem:Hopf-pair-bil-form}For $x,y\in\mathbf{f}(\epsilon)$,
we have $(x,y)=(x^{+},y^{-})$.
\end{lem}

\begin{proof}
Suppose we have $y_{1},y_{2}\in\mathbf{f}(\epsilon)$ such that $(x,y_{1})=(x_{1}^{+},y_{1}^{-})$
and $(x,y_{2})=(x^{+},y_{2}^{-})$ for any $x\in\mathbf{f}(\epsilon)$.
We shall show that $(x,y_{1}y_{2})=(x^{+},y_{1}^{-}y_{2}^{-})$ for
all $x\in\mathbf{f}(\epsilon)$. Indeed, if we write $r(x)=\sum x_{(1)}\otimes x_{(2)}$
for homogeneous $x_{(1)},x_{(2)}$, then we have
\[
\Delta(x^{+})=\sum x_{(1)}^{+}k_{\mathrm{cl}(\deg x_{(2)})}\otimes x_{(2)}^{+}
\]
where $k_{\beta}\coloneqq\prod k_{i}^{c_{i}}$ for $\beta=\sum c_{i}\alpha_{i}\in\mathring{Q}$.
With the assumption we argue that
\begin{align*}
(x,y_{1}y_{2})=(r(x),y_{1}\otimes y_{2}) & =\sum(x_{(1)},y_{1})(x_{(2)},y_{2})\\
 & =\sum(x_{(1)}^{+},y_{1}^{-})(x_{(2)}^{+},y_{2}^{-})\\
 & =\sum(x_{(1)}^{+}k_{\mathrm{cl}(\deg x_{(2)})},y_{1}^{-})(x_{(2)}^{+},y_{2}^{-})\\
 & =\sum(x_{(1)}^{+}k_{\mathrm{cl}(\deg x_{(2)})}\otimes x_{(2)}^{+},y_{1}^{-}\otimes y_{2}^{-})\\
 & =(\Delta(x^{+}),y_{1}^{-}\otimes y_{2}^{-})=(x^{+},y_{1}^{-}y_{2}^{-})
\end{align*}
as $(ak_{\beta},b)=(a,b)$ for $a\in\mathcal{U}^{+}$, $b\in\mathcal{U}^{-}$
and $\beta\in\mathring{Q}$. It remains to recall that $(\theta_{i},\theta_{j})=\frac{\delta_{ij}}{q_{i}^{-1}-q_{i}}=(e_{i},f_{j})$.
\end{proof}
In particular, the Hopf pairing is non-degenerate if $M\neq N$, which
is always assumed throughout this article. In such cases, the universal
$R$-matrix of $\mathcal{U}(\epsilon)$ can be constructed as follows
(see \cite[Section 3.2]{KL} for details, where a different comultiplication
is used). For each $\beta\in Q^{+}$, take a basis $\{v_{\beta}^{i}\}$
of $\mathcal{U}^{+}(\epsilon)_{\beta}$ and the dual basis $\{v_{\beta,i}\}$
of $\mathcal{U}^{-}(\epsilon)_{-\beta}$ with respect to the Hopf
pairing. Then the universal $R$-matrix is given by
\[
\mathcal{R}=\left(\sum_{\beta\in Q^{+}}\sum_{i}v_{\beta}^{i}\otimes v_{\beta,i}\right)\overline{\Pi}_{\mathbf{q}},
\]
where the summation is defined in a suitably completed tensor product
$\mathcal{U}^{+}(\epsilon)\widehat{\otimes}\mathcal{U}^{-}(\epsilon)$
and $\overline{\Pi}_{\mathbf{q}}$ is an operator given by
\[
\overline{\Pi}_{\mathbf{q}}(x\otimes y)=\mathbf{q}(\mu,\nu)^{-1}x\otimes y\quad\text{if }k_{\beta}x=\mathbf{q}(\beta,\mu)x,\,k_{\beta}y=\mathbf{q}(\beta,\nu)y\text{ for }\beta\in\mathring{Q}.
\]

The goal of this section is to prove the following Khoroshkin--Tolstoy-type
multiplicative formula for the universal $R$-matrix \cite{KT} for
$\mathcal{U}(\epsilon)$.
\begin{thm}
\label{thm:mult-formula-R} If $\epsilon$ is such that $p(i)p(i+1)=0$
for all $i=1,\dots,n-2$, then the universal $R$-matrix $\mathcal{R}$
for $\mathcal{U}(\epsilon)$ has the following decomposition:
\begin{align*}
\mathcal{R} & =\mathcal{R}^{+}\mathcal{R}^{0}\mathcal{R}^{-}\overline{\Pi}_{\mathbf{q}},\\
\mathcal{R}^{\pm} & =\prod_{\beta\in\Phi_{+}(\pm\infty)}\exp_{\beta}\left((q^{-1}-q)e_{\beta}\otimes f_{\beta}\right),\\
\mathcal{R}^{0} & =\exp\left(-\sum_{r>0}\sum_{i,j\in\mathring{I}}\frac{r(q-q^{-1})^{2}}{q_{i}^{r}-q_{i}^{-r}}\widetilde{C}_{ji}^{r}h_{i,r}\otimes h_{j,-r}\right)
\end{align*}
where $\widetilde{C}^{r}=(\widetilde{C}_{ij}^{r})_{i,j}$ is the inverse
matrix of 
\[
C^{r}=\left(\frac{\mathbf{q}(\alpha_{i},\alpha_{j})^{r}-\mathbf{q}(\alpha_{i},\alpha_{j})^{-r}}{q_{i}^{r}-q_{i}^{-r}}\right)_{i,j\in\mathring{I}}
\]
which is invertible whenever $M\neq N$.
\end{thm}

The unexplained notations, which will not play any significant role
in this article, can be found in Section \ref{subsec:aff-PBW} and
\ref{subsec:Pairing-PBW}. The proof consists of an explicit construction
of affine root vectors using the braid action and then a PBW basis
of $\mathcal{U}^{+}(\epsilon)$ (Section \ref{subsec:aff-PBW}), and
a computation of Hopf pairing between PBW vectors (Section \ref{subsec:Pairing-imag-rvector}
and \ref{subsec:Pairing-PBW}) which makes use of a comultiplication
formula for affine root vectors (Section \ref{subsec:comult-Dr})
and Levendorskii-Soibelman formula (Theorem \ref{thm:LS-formula}).
We will mostly follow the argument in \cite{D1}, with essentially
the only exception of Theorem \ref{thm:pairing-imag-rootv} where
one needs a new idea of relating the Hopf pairing to the bilinear
form on the algebra $\mathbf{f}(\epsilon)$.

\subsection{Affine root vectors, PBW basis and Levendorskii--Soibelman formula\label{subsec:aff-PBW}}

The set of positive roots $\Phi^{+}$($\subset Q$) of the affine
Lie superalgebra $\widehat{\mathfrak{sl}}_{M|N}$ can be identified
with the one of $\widehat{\mathfrak{sl}}_{M+N}$ once we forget the
parity (determined by $\epsilon$). Namely, 
\begin{align*}
\Phi^{+} & =\Phi_{\mathrm{re}}^{+}\sqcup\Phi_{\mathrm{im}}^{+},\\
\Phi_{\mathrm{re}}^{+} & =\left\{ r\delta+\beta\,|\,r\in\mathbb{Z}_{\geq0},\,\beta\in\mathring{\Phi}^{+}\right\} \cup\left\{ s\delta-\gamma\,|\,s\in\mathbb{Z}_{>0},\,\gamma\in\mathring{\Phi}^{+}\right\} ,\\
\Phi_{\mathrm{im}}^{+} & =\left\{ s\delta\,|\,s\in\mathbb{Z}_{>0}\right\} ,
\end{align*}
where $\mathring{\Phi}^{+}\subset\mathring{Q}$ is the set of positive
roots $\{\delta_{i}-\delta_{j}\,|\,1\leq i<j\leq n\}$ of $\mathfrak{sl}_{M|N}$
(or $\mathfrak{sl}_{M+N}$) and $\mathring{Q}$ is embedded into $Q$
by $\alpha_{i}\mapsto\alpha_{i}$ ($i\in\mathring{I}$). As above,
we assume that the affine Weyl group $W$ of the affine Lie algebra
$\widehat{\mathfrak{sl}}_{n}$ acts on $Q$ as in the non-super case
(\ref{eq:affW-action}). Recall that $W$ is a semidirect product
$\mathfrak{S}_{n}\ltimes\mathring{Q}$ of the Weyl group $\mathfrak{S}_{n}$
of $\mathfrak{sl}_{M+N}$ and the root lattice $\mathring{Q}=\bigoplus_{i=1}^{n-1}\mathbb{Z}\varpi_{i}$
of $\mathfrak{sl}_{n}$ (see \cite[Chapter 6]{Kac} for more details).
For $\alpha\in\mathring{Q}$, we denote by $t_{\alpha}$ the element
of $W$ corresponding to $(1,\alpha)\in\mathfrak{S}_{n}\ltimes\mathring{Q}$. 

The extended affine Weyl group $\widetilde{W}$ is a semidirect product
of the affine Weyl group $W$ and the cyclic group $T=\left\langle \mu\right\rangle $
of order $n$, where $\mu$ is the automorphism of the Dynkin diagram
of type $A_{n-1}^{(1)}$ determined by $\alpha_{i}\mapsto\alpha_{i+1}$
for all $i\in I$. Equivalently, we have $\mu s_{i}\mu^{-1}=s_{i+1}$
in $\widetilde{W}$. We say an expression 
\[
w=\mu^{k}s_{i_{1}}s_{i_{2}}\cdots s_{i_{\ell}}\in\widetilde{W}
\]
is reduced if $s_{i_{1}}\cdots s_{i_{\ell}}\in W$ is so.

The braid action can be extended to the one associated with $\widetilde{W}$
by introducing an isomorphism representing $\mu$. Indeed, the action
of $W$ on $\epsilon$ can be extended to $\widetilde{W}$ by setting
$(\epsilon\mu)_{i+1}=\epsilon_{i}$ and then we define
\begin{align*}
Z:\mathcal{U}(\epsilon) & \longrightarrow\mathcal{U}(\epsilon\mu)\\
e_{i},f_{i},k_{i} & \longmapsto e_{i+1},f_{i+1},k_{i+1}.
\end{align*}
Immediately we have $ZT_{i}Z^{-1}=T_{i+1}$, and hence for $w=\mu^{k}w^{\prime}\in\widetilde{W}$
we have a well-defined isomorphism $T_{w}=Z^{k}T_{w^{\prime}}$. 

On the other hand, $\widetilde{W}$ is also realized as $\mathring{W}\rtimes\mathring{P}_{A_{n-1}}$,
where the subgroup $W$ occurs as $\mathring{W}\rtimes\mathring{Q}$.
Let $\varpi_{i}\in\mathring{P}_{A_{n-1}}$ be the $i$-th fundamental
weight ($i=1,\dots,n-1$) and consider the element $t_{\varpi_{i}}=(1,\varpi_{i})$
of $\widetilde{W}=\mathring{W}\rtimes\mathring{P}_{A_{n-1}}$. We
may write it uniquely as $t_{\varpi_{i}}=\tau_{i}w_{i}^{\prime}$
for $\tau_{i}\in T,\,w_{i}^{\prime}\in W$, so that
\begin{align*}
t_{2\rho} & =t_{\varpi_{1}}t_{\varpi_{2}}\cdots t_{\varpi_{n-1}}t_{\varpi_{1}}t_{\varpi_{2}}\cdots t_{\varpi_{n-1}}\\
 & =(\tau_{1}w_{1}^{\prime})(\tau_{2}w_{2}^{\prime})\cdots(\tau_{n-1}w_{n-1}^{\prime})(\tau_{1}w_{1}^{\prime})(\tau_{2}w_{2}^{\prime})\cdots(\tau_{n-1}w_{n-1}^{\prime})
\end{align*}
for $2\rho=2\sum\varpi_{i}\in\mathring{Q}$. We can move $\tau_{i}$'s
to the right to rewrite it as
\[
t_{2\rho}=w_{1}w_{2}\cdots w_{n-1}w_{n}w_{n+1}\cdots w_{2n-2}\left(\prod_{i=1}^{n-1}\tau_{i}\right)^{2}.
\]
Since $t_{2\rho}\in W$, the last factor $(\prod\tau_{i})^{2}$ is
actually trivial. Then it is well-known that $\ell(t_{2\rho})=\sum_{i=1}^{2n-2}\ell(w_{i})$,
so that if we take a reduced expression $\underline{w}_{i}$ of $w_{i}$'s,
then the expression
\[
t_{2\rho}=\underline{w}_{1}\underline{w}_{2}\cdots\underline{w}_{n-1}\underline{w}_{n}\underline{w}_{n+1}\cdots\underline{w}_{2n-2}
\]
 is also reduced (see \textit{e.g.} \cite{Bo}).
\begin{rem}
\label{rem:T_varpi_i-inv}We use the notation $T_{\lambda}=T_{t_{\lambda}}$
for $\lambda\in\mathring{P}_{A_{n-1}}$. Then for $i\in\mathring{I}$,
$T_{\varpi_{i}}$ (and so any $T_{\lambda}$) is an automorphism of
$\mathcal{U}(\epsilon)$, namely does not change $\epsilon$. 
\end{rem}

From now on, we fix a reduced expression of $w_{j}^{\prime}$ given
by
\begin{equation}
w_{j}^{\prime}=\mathbf{s}_{(n-j,n-1)}\mathbf{s}_{(n-j-1,n-2)}\cdots\mathbf{s}_{(1,j)},\quad\mathbf{s}_{(a,b)}=s_{a}s_{a+1}\cdots s_{b}.\label{eq:red-exp-translation}
\end{equation}
Accordingly we obtain a reduced expression of $t_{2\rho}$ which we
write
\[
t_{2\rho}=s_{i_{1}}s_{i_{2}}\cdots s_{i_{N}},
\]
and then we let $i_{k}\coloneqq i_{k(\mathrm{mod}N)}$ for $k\in\mathbb{Z}$.
The set of positive real roots $\Phi_{\mathrm{re}}^{+}$ are then
parametrized by
\[
\Phi_{\mathrm{re}}^{+}=\{\beta_{k}\}_{k\in\mathbb{Z}},\quad\beta_{k}=\begin{cases}
s_{i_{1}}s_{i_{2}}\cdots s_{i_{k-1}}(\alpha_{i_{k}}) & \text{if }k>0\\
s_{i_{0}}s_{i_{-1}}\cdots s_{i_{k+1}}(\alpha_{i_{k}}) & \text{if }k\leq0.
\end{cases}
\]
One can further check that
\begin{align*}
\Phi^{+}(\infty)\coloneqq & \left\{ \beta_{k}\right\} _{k\geq0}=\left\{ s\delta-\gamma\,|\,s\in\mathbb{Z}_{>0},\,\gamma\in\mathring{\Phi}^{+}\right\} ,\\
\Phi^{+}(-\infty)\coloneqq & \left\{ \beta_{k}\right\} _{k<0}=\left\{ r\delta+\beta\,|\,r\in\mathbb{Z}_{\geq0},\,\beta\in\mathring{\Phi}^{+}\right\} .
\end{align*}
In particular, the subsets $\Phi^{+}(\pm\infty)$ of $\Phi_{\mathrm{re}}^{+}$
are independent of the choice of the reduce expression of $t_{2\rho}$. 

We define \textit{real root vectors} $e_{\beta}\in\mathcal{U}(\epsilon)_{\beta}$
for $\beta\in\Phi_{\mathrm{re}}^{+}$ by 
\[
e_{\beta_{k}}=\begin{cases}
T_{i_{1}}T_{i_{2}}\cdots T_{i_{k-1}}(e_{i_{k}}) & \text{if }k>0\\
T_{i_{0}}^{-1}T_{i_{-1}}^{-1}\cdots T_{i_{k+1}}^{-1}(e_{i_{k}}) & \text{if }k\leq0,
\end{cases}
\]
where $e_{i_{k}}$ is taken from $\mathcal{U}(\epsilon s_{i_{1}}\cdots s_{i_{k-1}})$
if $k>0$ and from $\mathcal{U}(\epsilon s_{i_{0}}\cdots s_{i_{k+1}})$
if $k\le0$ to ensure $e_{\beta}\in\mathcal{U}(\epsilon)$. Thanks
to Corollary \ref{prop:Lus-Lemma-40.1.2}, we have $e_{\beta_{k}}\in\mathcal{U}^{+}(\epsilon)$.
The root vectors for negative real roots are then defined by $f_{\beta_{k}}=\Omega(e_{\beta_{k}})\in\mathcal{U}^{-}(\epsilon)_{-\beta_{k}}$. 

Next, recall that the imaginary roots $r\delta$ occur with multiplicity
$n-1$. For $i\in\mathring{I}$ and $r>0$, we define a vector of
weight $r\delta$
\[
\widetilde{e}_{(r\delta,i)}=e_{r\delta-\alpha_{i}}e_{i}-\mathbf{q}(\alpha_{i},\alpha_{i})^{-1}e_{i}e_{r\delta-\alpha_{i}}.
\]
which obviously belongs to $\mathcal{U}^{+}(\epsilon)$. As \textit{imaginary
root vectors} we choose a renormalization $e_{(r\delta,i)}\in\mathcal{U}^{+}(\epsilon)_{r\delta}$
determined from the following equation
\[
\exp\left((q-q^{-1})\sum_{r>0}e_{(r\delta,i)}u^{r}\right)=1+(q-q^{-1})\sum_{r>0}\widetilde{e}_{(r\delta,i)}u^{r}.
\]
Finally we let $\widetilde{f}_{(r\delta,i)}=\Omega(\widetilde{e}_{(r\delta,i)})$
and $f_{(r\delta,i)}=\Omega(e_{(r\delta,i)})$ for $r>0$.
\begin{rem}
\label{rem:compare-rootv-Drinfeld}Let $o(i)=(-1)^{i}$. From our
choice of the reduced expression of $t_{2\rho},$ the loop generators
are given in terms of affine root vectors as (\textit{cf.} \cite[Lemma 1.5]{BCP})
\begin{align*}
x_{i,k}^{+} & =\begin{cases}
o(i)^{k}e_{k\delta+\alpha_{i}} & \text{if }k\geq0\\
-o(i)^{k}f_{-k\delta-\alpha_{i}}k_{i}^{-1}c^{k} & \text{if }k<0,
\end{cases}\quad x_{i,k}^{-}=\begin{cases}
-o(i)^{k}c^{-k}k_{i}e_{k\delta-\alpha_{i}} & \text{if }k>0\\
o(i)^{k}f_{-k\delta+\alpha_{i}} & \text{if }k\leq0,
\end{cases}\\
h_{i,r} & =\begin{cases}
o(i)^{r}c^{-r/2}e_{(r\delta,i)} & \text{if }r>0\\
o(i)^{r}c^{r/2}f_{(-r\delta,i)} & \text{if }r<0,
\end{cases}\quad\begin{cases}
\psi_{i,r}^{+}=o(i)^{r}(q-q^{-1})c^{-r/2}k_{i}\widetilde{e}_{(r\delta,i)}\\
\psi_{i,-r}^{-}=-o(i)^{r}(q-q^{-1})c^{r/2}k_{i}^{-1}\widetilde{f}_{(r\delta,i)}
\end{cases}\text{for }r>0,\quad\psi_{i,0}^{\pm}=k_{i}^{\pm1}.
\end{align*}
\end{rem}

The choice of a reduced expression of $t_{2\rho}$ gives rise to a
total ordering $\prec$ on the set of positive roots $\Phi^{+}$
\[
\underbrace{\beta_{0}\prec\beta_{-1}\prec\beta_{-2}\prec\cdots}_{\text{real roots }\beta_{k},\,k\leq0}\prec\underbrace{\cdots\prec2\delta\prec\delta}_{\text{imaginary roots}}\prec\underbrace{\cdots\prec\beta_{3}\prec\beta_{2}\prec\beta_{1}}_{\text{real roots }\beta_{k},\,k>0}
\]
which is convex in the sense that
\[
\beta=\beta_{1}+\beta_{2}\text{ for }\beta,\beta_{1},\beta_{2}\in\Phi^{+}\Longrightarrow\beta_{1}\prec\beta\prec\beta_{2}\text{ or }\beta_{2}\prec\beta\prec\beta_{1}.
\]
Let $\widetilde{\Phi}_{\mathrm{im}}^{+}=\left\{ (r\delta,i)\,|\,r\in\mathbb{Z}_{>0},\,i\in\mathring{I}\right\} $
be the set of positive imaginary roots counted with multiplicity.
We refine the ordering to the one on $\widetilde{\Phi}^{+}=\Phi_{\mathrm{re}}^{+}\cup\widetilde{\Phi}_{\mathrm{im}}^{+}$
by declaring
\[
\beta_{0}\prec\beta_{-1}\prec\cdots\prec\cdots\prec(2\delta,n-1)\prec(\delta,1)\prec\cdots\prec(\delta,n-2)\prec(\delta,n-1)\prec\cdots\prec\beta_{2}\prec\beta_{1}
\]
following the convention of \cite{D1}.

Now we set
\[
\mathrm{Par}_{\epsilon}=\left\{ \overrightarrow{\gamma}=(\gamma_{1}\preceq\cdots\preceq\gamma_{r})\in\bigcup_{r\geq0}\left(\widetilde{\Phi}_{+}\right)^{r}\,\vert\,{}^{\forall}i,\,\gamma_{i}\neq\gamma_{i+1}\text{ if }p(\gamma_{i})=1\right\} 
\]
where we assume $p((r\delta,i))=p(r\delta)=0$. The extra condition
is a reflection of the fact that any odd element $x$ of a Lie superalgebra
$\mathfrak{g}$ satisfying $[x,x]=0$ (which is automatic in $\mathfrak{sl}_{M|N}$
or $\widehat{\mathfrak{sl}}_{M|N}$) squares to zero in the universal
enveloping algebra. This is still justified in our quantum case by
the following easy observation.
\begin{lem}
\label{lem:orrvector-sqzero}If $p(\beta_{k})=1$, then $e_{\beta_{k}}^{2}=0=f_{\beta_{k}}^{2}$.
\end{lem}

\begin{proof}
Since $e_{\beta_{k}}^{2}=T_{i_{1}}\cdots T_{i_{k-1}}(e_{i_{k}}^{2})$
for $k>0$ (and similarly for $k\leq0$), it suffices to observe that
the parity is fixed under $T_{i}^{\pm1}$, that is
\[
p_{\epsilon}(\beta)=p_{\epsilon s_{i}}(s_{i}\beta)\quad^{\forall}\beta\in Q.
\]
\end{proof}
For $\overrightarrow{\gamma}\in\mathrm{Par}_{\epsilon}$, we define
a PBW vector (associated with the choice of $\epsilon$ and a reduced
expression of $t_{2\rho}$)
\[
e(\overrightarrow{\gamma})=e_{\gamma_{1}}e_{\gamma_{2}}\cdots e_{\gamma_{r}},\quad e(-\overrightarrow{\gamma})=e_{\gamma_{r}}e_{\gamma_{r-1}}\cdots e_{\gamma_{1}}
\]
and also $f(\overrightarrow{\gamma})=\Omega(e(-\overrightarrow{\gamma}))$.
We shall prove that they form a basis of $\mathcal{U}^{+}(\epsilon)$. 
\begin{thm}
\label{thm:PBW-basis}The sets $\left\{ e(\overrightarrow{\gamma})\,|\,\overrightarrow{\gamma}\in\mathrm{Par}_{\epsilon}\right\} $
and $\left\{ e(-\overrightarrow{\gamma})\,|\,\overrightarrow{\gamma}\in\mathrm{Par}_{\epsilon}\right\} $
are bases of $\mathcal{U}^{+}(\epsilon)$.
\end{thm}

To prove this, we may follow the argument of \cite[Section 1]{L}
using the classical limit $q\rightarrow1$ of the quantum affine superalgebra
$U(\epsilon)$. For this, we first consider affine PBW vectors for
$U(\epsilon)$ as above. A braid action $T_{i}:U(\epsilon)\rightarrow U(\epsilon s_{i})$
is given in \cite[Proposition 8.2.1]{Y}
\begin{align*}
T_{i}(K_{j}) & =\begin{cases}
K_{i}^{-1}\\
K_{i-1}K_{i}\\
K_{i+1}K_{i}\\
K_{j}
\end{cases},\quad T_{i}(E_{j})=\begin{cases}
-F_{i}K_{i}\\
-(-1)^{(\epsilon_{i-1}+\epsilon_{i+1})(\epsilon_{i}+\epsilon_{i+1})}\left\llbracket E_{i},E_{i-1}\right\rrbracket \\
-\left\llbracket E_{i},E_{i+1}\right\rrbracket \\
E_{j}
\end{cases},\\
T_{i}(F_{j}) & =\begin{cases}
-K_{i}^{-1}E_{i} & \text{if }j=i\\
-(-1)^{(\epsilon_{i-1}+\epsilon_{i+1})(\epsilon_{i}+\epsilon_{i+1})}\left\llbracket F_{i-1},F_{i}\right\rrbracket  & \text{if }j=i-1\\
-\left\llbracket F_{i+1},F_{i}\right\rrbracket  & \text{if }j=i+1\\
F_{j} & \text{otherwise}
\end{cases}
\end{align*}
where the $q$-bracket $\left\llbracket x,y\right\rrbracket $ is
defined for homogeneous $x,y\in U^{+}(\epsilon)$ by 
\begin{equation}
\left\llbracket x,y\right\rrbracket =xy-(-1)^{p(x)p(y)}q^{(\deg x,\deg y)}yx.\label{eq:q-bracket-def-qasa}
\end{equation}
Then the real root vectors $E_{\beta}\in U^{+}(\epsilon)$ for $\beta\in\Phi_{\mathrm{re}}^{+}$
can be defined with the braid action in the same manner. We take imaginary
root vectors $\widetilde{E}_{(r\delta,i)}$ (before renormalization)
\[
\widetilde{E}_{(r\delta,i)}=\left\llbracket E_{r\delta-\alpha_{i}},E_{i}\right\rrbracket =E_{r\delta-\alpha_{i}}E_{i}-(-1)^{p(i)}q^{-(\alpha_{i},\alpha_{i})}E_{i}E_{r\delta-\alpha_{i}}.
\]
We also remark that the definition of the $q$-bracket differs from
\cite[Section 8]{Y}, so our $\widetilde{E}_{(r\delta,i)}$ equals
to $(-1)^{\epsilon_{i}+\epsilon_{i+1}}C^{r/2}\overline{\psi}_{i,r}$
therein.
\begin{prop}
\label{prop:rootv-q-brac}The root vectors $E_{\beta}$ and $\widetilde{E}_{(r\delta,i)}$
can be expressed as iterated $q$-brackets of $E_{0},E_{1},\dots,E_{n-1}$
up to a non-zero scalar multiple from $\pm q^{\mathbb{Z}}$.
\end{prop}

\begin{proof}
Since imaginary root vectors $\widetilde{E}_{(r\delta,i)}$ are already
defined as the $q$-bracket of real root vectors, it is enough to
consider the latters. Indeed, we shall prove the claim for $T_{w}(E_{i})$,
for any $w\in W$ and $i\in I$ such that $w(\alpha_{i})>0$, by induction
on the length of $w$ as in the proof of \cite[Proposition 1.8]{L}.

The base case $\ell(w)=1$ is the definition of $T_{i}$. Assume the
claim for every $w^{\prime}\in W$ with $\ell(w^{\prime})<\ell(w)$
and let $i\in I$ be such that $w(\alpha_{i})>0$. Take $j\neq i$
with $w(\alpha_{j})<0$ and $w^{\prime}$ the minimal length element
in the coset $w\left\langle s_{i},s_{j}\right\rangle $ of the subgroup
generated by $s_{i},s_{j}\in W$. Then we divide the proof into the
following cases:

\textit{Case 1. }$\left|i-j\right|>1$ and $w=w^{\prime}s_{j}$. We
have $T_{w}(E_{i})=T_{w^{\prime}}T_{j}(E_{i})=T_{w^{\prime}}(E_{i})$
and the induction works immediately.

\textit{Case 2. }$\left|i-j\right|=1$ and $w=w^{\prime}s_{i}s_{j}$.
By a direct computation one can verify $T_{i}T_{j}(E_{i})=E_{j}$
(\textit{cf.} Proposition \ref{prop:Tw-on-ef}) so that $T_{w}(E_{i})=T_{w^{\prime}}(E_{i})$.

\textit{Case 3. }$\left|i-j\right|=1$ and $w=w^{\prime}s_{j}$. For
notational convenience we may assume $j=i-1$, and we have $T_{w}(E_{i})=T_{w^{\prime}}(\left\llbracket E_{i-1},E_{i}\right\rrbracket )$
up to a scalar multiple. Assume $T_{w}(E_{i})\in U(\epsilon)$, so
that the $q$-bracket $\left\llbracket E_{i-1},E_{i}\right\rrbracket $
in the righthand side belongs to $U(\epsilon w^{\prime})$ and then
\begin{align*}
\left\llbracket E_{i-1},E_{i}\right\rrbracket  & =E_{i-1}E_{i}-(-1)^{(\epsilon_{i}^{\prime}+\epsilon_{i+1}^{\prime})(\epsilon_{i-1}^{\prime}+\epsilon_{i}^{\prime})}q^{-(-1)^{\epsilon_{i}^{\prime}}}E_{i}E_{i-1}\\
 & =E_{i-1}E_{i}-(-1)^{(\epsilon_{w^{\prime}(i)}+\epsilon_{w^{\prime}(i+1)})(\epsilon_{w^{\prime}(i-1)}+\epsilon_{w^{\prime}(i)})}q^{-(-1)^{\epsilon_{w^{\prime}(i)}}}E_{i}E_{i-1}
\end{align*}
since $(\epsilon w^{\prime})_{k}=\epsilon_{w^{\prime}(k)}$. Now that
$\epsilon_{w^{\prime}(i)}+\epsilon_{w^{\prime}(i+1)}=p(w^{\prime}(\alpha_{i}))$,
$\epsilon_{w^{\prime}(i-1)}+\epsilon_{w^{\prime}(i)}=p(w^{\prime}(\alpha_{i-1}))$
and $(-1)^{\epsilon_{w^{\prime}(i)}}=-(w^{\prime}(\alpha_{i}),w^{\prime}(\alpha_{i-1}))$,
we have
\begin{align*}
T_{w^{\prime}}(\left\llbracket E_{i-1},E_{i}\right\rrbracket ) & =T_{w^{\prime}}(E_{i-1})T_{w^{\prime}}(E_{i})-(-1)^{(\epsilon_{w^{\prime}(i)}+\epsilon_{w^{\prime}(i+1)})(\epsilon_{w^{\prime}(i-1)}+\epsilon_{w^{\prime}(i)})}q^{-(-1)^{\epsilon_{w^{\prime}(i)}}}T_{w^{\prime}}(E_{i})T_{w^{\prime}}(E_{i-1})\\
 & =T_{w^{\prime}}(E_{i-1})T_{w^{\prime}}(E_{i})-(-1)^{p(w^{\prime}(\alpha_{i}))p(w^{\prime}(\alpha_{i-1}))}q^{(w^{\prime}(\alpha_{i}),w^{\prime}(\alpha_{i-1}))}T_{w^{\prime}}(E_{i})T_{w^{\prime}}(E_{i-1})\\
 & =\left\llbracket T_{w^{\prime}}(E_{i-1}),T_{w^{\prime}}(E_{i})\right\rrbracket 
\end{align*}
\[
\]
and the induction hypothesis applies.
\end{proof}
\begin{rem}
The same holds for the root vectors of $\mathcal{U}(\epsilon)$ if
we define the $q$-bracket as 
\begin{equation}
\left\llbracket x,y\right\rrbracket =xy-\mathbf{q}(\left|x\right|,\left|y\right|)yx\label{eq:q-bracket-def}
\end{equation}
for homogeneous $x,y\in\mathcal{U}^{+}(\epsilon)$ (\textit{cf}. Remark
\ref{rem:tau-pres-q-bracket}). Later we will obtain a more precise
description for several root vectors, see Section \ref{subsec:Pairing-imag-rvector}
below. We also note that by Lemma \ref{lem:W-pres-q}, the $q$-bracket
is preserved under $T_{i}$: $T_{i}\left\llbracket x,y\right\rrbracket =\left\llbracket T_{i}x,T_{i}y\right\rrbracket $.
\[
\]
\end{rem}

From the proposition we see that under the classical limit $q\rightarrow1$,
the root vectors of $U^{+}(\epsilon)$ become the usual root vectors
of the affine Lie superalgebra $\widehat{\mathfrak{sl}}_{M|N}=\mathfrak{sl}_{M|N}\otimes\mathbb{C}[t^{\pm1}]\oplus\mathbb{C}C$
\begin{align*}
E_{r\delta\pm\alpha} & \longmapsto\overline{E}_{\pm\alpha}\otimes t^{r}\quad\text{for }r\delta\pm\alpha\in\Phi_{\mathrm{re}}^{+},\\
\widetilde{E}_{(r\delta,i)} & \longmapsto h_{i}\otimes t^{r}
\end{align*}
up to signs, where $\overline{E}_{\pm\alpha}$ is a root vector of
$\mathfrak{sl}_{M|N}$ (with respect to the simple system associated
with $\epsilon$, see \cite[Section 1.3]{CW}) and $h_{i}=(-1)^{\epsilon_{i}}E_{ii}-(-1)^{\epsilon_{i+1}}E_{i+1,i+1}\in\mathrm{End}(\mathbb{C}^{M|N})$
with respect to the matrix presentation of $\mathrm{End}(\mathbb{C}^{M|N})$
associated to $\epsilon$. Therefore, the set
\[
\left\{ E(\overrightarrow{\gamma})=E_{\gamma_{1}}E_{\gamma_{2}}\cdots E_{\gamma_{r}}\,|\,\overrightarrow{\gamma}\in\mathrm{Par}_{\epsilon}\right\} \subset U^{+}(\epsilon)
\]
 induces a PBW basis of the positive part $U^{+}(\widehat{\mathfrak{sl}}_{M|N})$
of the universal enveloping algebra of $\widehat{\mathfrak{sl}}_{M|N}$
by taking $q\rightarrow1$, which implies its linear independence
(and also that it is a PBW basis for $U^{+}(\epsilon)$).

Thanks to Remark \ref{rem:tau-pres-q-bracket}, a PBW vector $e(\overrightarrow{\gamma})$
is mapped to $E(\overrightarrow{\gamma})$ up to $\Sigma$-factors
and hence linearly independent as well. Since the character of $\mathcal{U}^{+}(\epsilon)$
with respect to the $Q$-grading coincides with that of $U^{+}(\epsilon)$
and hence of $U^{+}(\widehat{\mathfrak{sl}}_{M|N})$ which matches
the one expected by counting the PBW vectors, this proves Theorem
\ref{thm:PBW-basis}. 

The proof for $\{e(-\overrightarrow{\gamma}))\}$ is exactly the same,
and then by the triangular decomposition (\ref{eq:GQG-tri-decomp}),
we obtain a PBW basis $\{e(-\overrightarrow{\gamma})k_{\alpha}f(\overrightarrow{\gamma}^{\prime})\}_{\overrightarrow{\gamma},\overrightarrow{\gamma}^{\prime}\in\mathrm{Par}_{\epsilon},\,\alpha\in\mathring{Q}}$
for the whole algebra $\mathcal{U}(\epsilon)$, which we will assume
henceforth. The following Levendorskii--Soibelman formula for $\mathcal{U}^{+}(\epsilon)$
can then be proved using the PBW basis as in \cite[Proposition 7]{B1}. 
\begin{thm}
\label{thm:LS-formula}For $\alpha,\beta\in\widetilde{\Phi}^{+}$
such that $\beta\succ\alpha$, we have
\[
e_{\beta}e_{\alpha}-\mathbf{q}(\beta,\alpha)e_{\alpha}e_{\beta}=\sum_{\alpha\prec\gamma_{1}\preceq\cdots\preceq\gamma_{r}\prec\beta}c_{\overrightarrow{\gamma}}e(\overrightarrow{\gamma}).\quad(c_{\overrightarrow{\gamma}}\in\Bbbk)
\]
\end{thm}

\subsection{A comultiplication formula for affine root vectors\label{subsec:comult-Dr}}

To use the PBW basis for the computation of the universal $R$-matrix,
we have to compute the Hopf pairing between PBW vectors $e(-\overrightarrow{\gamma})$
and $f(\overrightarrow{\gamma})$. To prepare it, in this section
we shall obtain a comultiplication formula for affine root vectors
in \cite[Section 4-7]{D1}, which is an interesting problem in itself.

For $r\in\mathbb{Z}$, we let
\[
\Phi^{+}(r)=\begin{cases}
\{\beta_{s}\,|\,1\leq s<r\} & \text{if }r>0\\
\{\beta_{s}\,|\,r<s\leq0\} & \text{if }r\leq0
\end{cases}
\]
so that $\bigcup_{r>0}\Phi^{+}(r)=\Phi^{+}(\infty)$ and $\bigcup_{r\leq0}\Phi^{+}(r)=\Phi^{+}(-\infty)$.
Accordingly, we define the subspaces $\mathcal{U}_{r}^{+}(\epsilon)$
spanned by the PBW vectors comprised of root vectors from $\Phi^{+}(r)$:
\[
\mathcal{U}_{r}^{+}(\epsilon)=\mathrm{span}\left\{ e(\overrightarrow{\gamma})\,|\,\overrightarrow{\gamma}\in\mathcal{P},\,^{\forall}\gamma_{i}\in\Phi^{+}(r)\right\} ,\quad\mathcal{U}_{r}^{\geq0}(\epsilon)=\mathcal{U}_{r}^{+}(\epsilon)\cdot\mathcal{U}^{0}(\epsilon).
\]
By the LS formula, $\mathcal{U}_{r}^{+}(\epsilon)$ (and so $\mathcal{U}_{r}^{\geq0}(\epsilon)$)
is a subalgebra of $\mathcal{U}$.

\subsubsection{Real root vectors}

As real root vectors are obtained using the braid symmetry $T_{i}^{\pm1}$,
their comultiplication $\Delta(e_{\beta_{k}})$ can be computed by
invoking a relation between $\Delta$ and $T_{i}^{\pm1}$. For each
$i\in I$, consider the universal $R$-matrix for the rank 1 subalgebra
generated by $e_{i},\,f_{i},\,k_{i}^{\pm1}$:
\[
\mathcal{U}(\epsilon)\widehat{\otimes}\mathcal{U}(\epsilon)\ni\mathcal{R}_{i,\epsilon}=\begin{cases}
\sum_{k\geq0}(q_{i}^{-1}-q_{i})^{k}\frac{q_{i}^{-k(k-1)/2}}{[k]_{q_{i}}!}e_{i}^{k}\otimes f_{i}^{k} & \text{if }\epsilon_{i}=\epsilon_{i+1},\\
1+(q^{-1}-q)e_{i}\otimes f_{i} & \text{if }\epsilon_{i}\neq\epsilon_{i+1},
\end{cases}
\]
which is invertible. The following result, known in the even case
of $\epsilon_{i}=\epsilon_{i+1}$ \cite[Proposition 37.3.2]{Lb},
can be proved in the remaining odd case by a simple calculation.
\begin{lem}
For $i\in I$ and $x\in\mathcal{U}(\epsilon)$, we have
\[
\Delta(T_{i}(x))=\mathcal{R}_{i,\epsilon s_{i}}^{-1}\cdot(T_{i}\otimes T_{i})\Delta(x)\cdot\mathcal{R}_{i,\epsilon s_{i}}\in\mathcal{U}(\epsilon s_{i})\otimes\mathcal{U}(\epsilon s_{i})
\]
\end{lem}

Since the formula is formally the same for either even or odd $i\in I$,
the remaining argument in \cite[Section 5]{D1} works here as well.
Namely, if we let for $k>0$
\[
R(k)=\begin{cases}
\mathcal{R}_{i_{1},\epsilon} & \text{if }k=1\\
\left((T_{i_{1}}\otimes T_{i_{1}})\cdots(T_{i_{k-1}}\otimes T_{i_{k-1}})\mathcal{R}_{i_{k},\epsilon s_{i_{1}}\cdots s_{i_{k-1}}}\right)R(k-1) & \text{if }k>1
\end{cases}
\]
 then we have 
\[
\Delta(e_{\beta_{k}})=\Delta(T_{i_{1}}\cdots T_{i_{k-1}}(e_{i_{k}}))=R(k)^{-1}\left((T_{i_{1}}\otimes T_{i_{1}})\cdots(T_{i_{k-1}}\otimes T_{i_{k-1}})\Delta(e_{i_{k}})\right)R(k).
\]
By definition and the ordering of positive real roots, if we write
$R(k)=\sum_{u}l_{u}\otimes l_{u}^{\prime}$ and $R(k)^{-1}=\sum_{u}m_{u}\otimes m_{u}^{\prime}$,
then
\[
l_{u}k_{\mathrm{wt}(l_{u})}\otimes l_{u}^{\prime}k_{\mathrm{wt}(l_{u})}^{-1},\,m_{u}k_{\mathrm{wt}(m_{u})}\otimes m_{u}^{\prime}k_{\mathrm{wt}(m_{u})}^{-1}\in\mathcal{U}_{k+1}^{+}(\epsilon)\otimes\mathcal{U}_{k+1}^{-}(\epsilon).
\]
One can repeat this for $k\leq0$, from which we deduce an approximation
of the comultiplication of real root vectors (the one for $f_{\beta_{k}}$
follows by applying $\Omega$).
\begin{prop}
For $k\in\mathbb{Z}$, we have
\begin{align*}
\Delta(e_{\beta_{k}})-(e_{\beta_{k}}\otimes1+k_{\beta_{k}}\otimes e_{\beta_{k}}) & \in\begin{cases}
\mathcal{U}_{k}^{\geq0}(\epsilon)\otimes\mathcal{U}^{+}(\epsilon) & \text{if }k>0\\
\mathcal{U}^{\geq0}(\epsilon)\otimes\mathcal{U}_{k}^{+}(\epsilon) & \text{if }k\leq0,
\end{cases}\\
\Delta(f_{\beta_{k}})-(f_{\beta_{k}}\otimes k_{\beta_{k}}^{-1}+1\otimes f_{\beta_{k}}) & \in\begin{cases}
\mathcal{U}^{-}(\epsilon)\otimes\mathcal{U}_{k}^{\leq0}(\epsilon) & \text{if }k>0\\
\mathcal{U}_{k}^{-}(\epsilon)\otimes\mathcal{U}^{\leq0}(\epsilon) & \text{if }k\leq0.
\end{cases}
\end{align*}
\end{prop}

Together with LS formula, this is enough to compute the Hopf pairing
of PBW vectors consisting of real root vectors. Note that in the formula
below, if $p(\beta)=1$, then $m_{\beta}$ is at most 1 so $(m_{\beta})_{\beta}=1$.
Moreover, we have $(e_{\beta_{k}},e_{\beta_{k}})=\frac{1}{q_{i_{k}}^{-1}-q_{i_{k}}}$
by Proposition \ref{prop:braid-isometry}. 
\begin{cor}[{cf. \cite[Theorem 4.4]{D1}}]
Let $\overrightarrow{\gamma},\overrightarrow{\gamma}^{\prime}\in\mathcal{P}$
be such that $\gamma_{i},\gamma_{j}^{\prime}\in\Phi^{+}(\infty)$
for all $i,j$. Then we have
\[
\left(e(-\overrightarrow{\gamma}),e(-\overrightarrow{\gamma}^{\prime})\right)=\delta_{\overrightarrow{\gamma},\overrightarrow{\gamma}^{\prime}}\prod_{\beta\in\Phi^{+}(\infty)}\frac{(m_{\beta})_{\beta}}{(q_{\beta}^{-1}-q_{\beta})^{m_{\beta}}}
\]
where $m_{\beta}$ is the multiplicity of $\beta$ in $\overrightarrow{\gamma}$,
$q_{\beta}=q_{i_{k}}$ when $\beta=\beta_{k}$ and $(n)_{\beta}=1+\mathbf{q}(\beta,\beta)+\cdots+\mathbf{q}(\beta,\beta)^{n-1}$.
\end{cor}

\subsubsection{Imaginary root vectors}

As above, we take
\[
\mathcal{U}_{\mathrm{im}}^{+}(\epsilon)=\mathrm{span}\left\{ e(\overrightarrow{\gamma})\,|\,\overrightarrow{\gamma}\in\mathcal{P},\,^{\forall}\gamma_{i}\in\widetilde{\Phi}_{\mathrm{im}}^{+}\right\} ,\quad\mathcal{U}_{\mathrm{im}}^{\geq0}(\epsilon)=\mathcal{U}_{\mathrm{im}}^{+}(\epsilon)\cdot\mathcal{U}^{0}(\epsilon).
\]
We also let $\Phi^{+}(\pm\widetilde{\infty})=\Phi^{+}(\pm\infty)\cup\Phi_{\mathrm{im}}^{+}$
and define $\mathcal{U}_{\pm\widetilde{\infty}}^{+}(\epsilon),\mathcal{U}_{\pm\widetilde{\infty}}^{\geq0}(\epsilon)\subset\mathcal{U}(\epsilon)$
accordingly. Again, they are all subalgebras of $\mathcal{U}(\epsilon)$.
\begin{lem}[{cf. \cite[Lemma 5.1]{D1}}]
 If $k>0$, then $T_{2\rho}(e_{\beta_{k}})=e_{\beta_{k+N}}$ and
there exists $m>0$ such that $T_{2\rho}^{-m^{\prime}}(e_{\beta_{k}})\in\mathcal{U}^{\leq0}(\epsilon)$
for all $m^{\prime}\geq m$.

Similarly when $k\leq0$, $T_{2\rho}^{-1}(e_{\beta_{k}})=e_{\beta_{k-N}}$
and there exists $m>0$ such that $T_{2\rho}^{m^{\prime}}(e_{\beta_{k}})\in\mathcal{U}^{\leq0}(\epsilon)$
for all $m^{\prime}\geq m$.
\end{lem}

This is obvious from the definition of real root vectors, and the
following statement can be shown by this lemma with the formulas for
real root vectors above.
\begin{prop}[{cf. \cite[Proposition 5.2]{D1}}]
 If $x\in\mathcal{U}_{\mathrm{im}}^{+}(\epsilon)$, then $\Delta(x)\in\mathcal{U}_{\widetilde{\infty}}^{\geq0}(\epsilon)\otimes\mathcal{U}_{-\widetilde{\infty}}^{+}(\epsilon)$.
\end{prop}

\begin{lem}[{cf. \cite[Lemma 6.1]{D1}}]
\begin{enumerate}
\item If $k>0$, then there exists $i\in\mathring{I}$ such that $e_{\beta_{k}}e_{i}\neq\mathbf{q}(\alpha_{i},\beta_{k})e_{i}e_{\beta_{k}}$.
\item If $k\leq0$, then there exists $i\in\mathring{I}$ such that $e_{\beta_{k}}e_{\delta-\alpha_{i}}\neq\mathbf{q}(\alpha_{i},\beta_{k})e_{\delta-\alpha_{i}}e_{\beta_{k}}$.
\item Given $m>0$ and $a_{i}\in\Bbbk$ for $i\in\mathring{I}$, if either
\[
\left[\sum a_{i}e_{(m\delta,i)},e_{j}\right]=0\:^{\forall}j\in\mathring{I}\quad\text{or}\quad\left[\sum a_{i}e_{(m\delta,i)},e_{\delta-\alpha_{j}}\right]=0\:^{\forall}j\in\mathring{I}
\]
holds, then $a_{i}=0$ for all $i\in\mathring{I}$.
\item Given $m>0$ and $x\in\mathcal{U}_{\mathrm{im}}^{+}(\epsilon)_{m\delta}$,
if either
\[
\left[x,e_{j}\right]=0\,^{\forall}j\in\mathring{I}\quad\text{or}\quad\left[x,e_{\delta-\alpha_{j}}\right]=0\,^{\forall}j\in\mathring{I}
\]
holds, then $x=0$.
\end{enumerate}
\end{lem}

\begin{proof}
We follow the proof of \cite[Lemma 6.1]{D1}.
\begin{enumerate}
\item If the claim does not hold for $\beta_{k}=r\delta-\alpha$ $(\alpha\in\mathring{\Phi}^{+})$,
then we have $e_{\beta_{k}}e_{\alpha}=\mathbf{q}(\alpha,\beta_{k})e_{\alpha}e_{\beta_{k}}$.
This is equivalent to $\left\llbracket E_{\beta_{k}},E_{\alpha}\right\rrbracket =0$
by applying $\tau^{-1}$, which is absurd as it implies after taking
the classical limit
\[
0\neq\left[\overline{E}_{-\alpha},\overline{E}_{\alpha}\right]\otimes t^{r}=[\overline{E}_{-\alpha}\otimes t^{r},\overline{E}_{\alpha}\otimes1]=0.
\]
\item As $k\leq0$, $\beta_{k}=m\delta+\beta$ for a $\beta\in\mathring{\Phi}^{+}$
and $T_{\rho}^{-1}(e_{\beta_{k}})=e_{m^{\prime}\delta+\beta}$ for
an $m^{\prime}>m$. On the other hand, $T_{\rho}^{-1}(e_{\delta-\alpha_{i}})=T_{\varpi_{i}}^{-1}(e_{\delta-\alpha_{i}})$
by Remark \ref{rem:T_varpi_i-inv} and then as $e_{\delta-\alpha_{i}}=T_{\varpi_{i}}T_{i}^{-1}(e_{i}),$
we have $T_{\rho}^{-1}(e_{\delta-\alpha_{i}})=T_{i}^{-1}(e_{i})=-k_{i}^{-1}f_{i}$.
Thus, the statement is not true whenever $[e_{m^{\prime}\delta+\beta},f_{i}]=-k_{i}T_{\rho}^{-1}\left\llbracket e_{\delta-\alpha_{i}},e_{\beta_{k}}\right\rrbracket =0$
for all $i\in\mathring{I}$, which cannot happen as in (1).
\item Using the commutation relation \ref{eq:GQG-Drinfeld-hx} (together
with Remark \ref{rem:compare-rootv-Drinfeld}), the assertion follows
from the invertibility of the matrix $(\frac{\mathbf{q}(\alpha_{i},\alpha_{j})^{m}-\mathbf{q}(\alpha_{i},\alpha_{j})^{-m}}{q-q^{-1}})_{i,j\in\mathring{I}}$
which is true whenever $M\neq N$.
\item This is completely parallel with the non-super case, using the Levendorskii--Soibelman
formula.
\end{enumerate}
\end{proof}
\begin{prop}[{cf. \cite[Corollary 6.6]{D1}}]
 If $x\in\mathcal{U}_{\widetilde{\infty}}^{+}(\epsilon)_{\beta}$
is such that $xe_{i}=\mathbf{q}(\beta,\alpha_{i})e_{i}x$ for all
$i\in\mathring{I}$, then $x=0$.

Similarly for $y\in\mathcal{U}_{-\widetilde{\infty}}^{+}(\epsilon)_{\beta}$
satisfying $ye_{\delta-\alpha_{i}}=\mathbf{q}(\beta,\alpha_{i})e_{\delta-\alpha_{i}}y$
for all $i\in\mathring{I}$, we have $y=0$.
\end{prop}

\begin{proof}
The proof in \cite{D1} relies on the lemma and several properties
of the ordering $\prec$ \cite[Lemma 6.2 - 6.4]{D1} and hence can
be reproduced in our case word for word.
\end{proof}
By the comultiplication formula for imaginary root vectors, we mean
the following refined statement which can be again proved as in \cite[Section 7]{D1}.
\begin{thm}[{cf. \cite[Proposition 7.1]{D1}}]
\label{thm:comult-imag-rootv} For $\alpha=(r\delta,i)\in\widetilde{\Phi}_{\mathrm{im}}^{+}$,
we have
\[
\Delta(e_{\alpha})-(e_{\alpha}\otimes1+c^{r}\otimes e_{\alpha})\in\mathcal{U}_{\infty}^{\geq0}(\epsilon)\otimes\mathcal{U}_{-\infty}^{+}(\epsilon).
\]
\end{thm}

This theorem also implies the following special case of orthogonality.
\begin{lem}[{cf. \cite[Proposition 8.1]{D1}}]
\label{lem:Dam-Prop8.1} For $\alpha\in\widetilde{\Phi}^{+}$ and
$\overrightarrow{\gamma}=(\gamma_{1},\dots,\gamma_{r})\in\mathrm{Par}_{\epsilon}$,
if 
\[
\text{either }(e_{\alpha},f(-\overrightarrow{\gamma}))\neq0\text{ or }(e(-\overrightarrow{\gamma}),f_{\alpha})\neq0,
\]
then we have $r=1$.
\end{lem}

\subsection{\label{subsec:Pairing-imag-rvector}Pairing of imaginary root vectors}

For each $j\in\mathring{I}$, consider the projection
\[
\pi_{j}:\mathcal{U}(\epsilon)\otimes\mathcal{U}(\epsilon)\longrightarrow\mathcal{U}(\epsilon)\otimes\Bbbk e_{j}
\]
on the second factor along the complement of $\Bbbk e_{j}$ with respect
to the PBW basis of $\mathcal{U}(\epsilon)$. Thanks to the convexity
of the ordering $\prec$ and the comultiplication formula for $e_{(r\delta,i)}\in\mathcal{U}_{\mathrm{im}}^{+}(\epsilon)$,
we have
\[
\pi_{j}(\Delta(e_{(r\delta,i)}))=c_{r}^{ij}e_{r\delta-\alpha_{j}}k_{j}\otimes e_{j}
\]
for some scalar $c_{m}^{ij}$, which is related to the pairing of
imaginary root vectors as follows.
\begin{lem}[{cf. \cite[Lemma 9.4]{D1}}]
\label{lem:pair-imag-cijr}For $r,s>0$ and $i,j\in\mathring{I}$,
we have 
\[
(e_{(r\delta,i)},f_{(s\delta,j)})=-\delta_{r,s}c_{r}^{ij}\frac{\mathbf{q}(\alpha_{j},\alpha_{j})}{(q_{j}^{-1}-q_{j})(q_{j+1}^{-1}-q_{j+1})}.
\]
\end{lem}

\begin{proof}
First, from the $Q$-grading we have $(e_{(r\delta,i)},f_{(s\delta,j)})=0$
whenever $r\neq s$. For $r=s$, one can proceed as in the proof of
\cite[Lemma 9.4]{D1}. Indeed, as $f_{(r\delta,j)}=\widetilde{f}_{(r\delta,j)}$
modulo the subalgebra generated by $f_{(s\delta,j)}$ for $s<r$,
by Lemma \ref{lem:Dam-Prop8.1} we have
\[
(e_{(r\delta,i)},f_{(r\delta,j)})=(e_{(r\delta,i)},\widetilde{f}_{(r\delta,j)})=(e_{(r\delta,i)},f_{j}f_{r\delta-\alpha_{j}}-\mathbf{q}(\alpha_{j},\alpha_{j})f_{r\delta-\alpha_{j}}f_{j}).
\]
Since $f_{j}f_{r\delta-\alpha_{j}}=f(\alpha_{j},r\delta-\alpha_{j})$,
again by Lemma \ref{lem:Dam-Prop8.1},
\begin{align*}
(e_{(r\delta,i)},f_{j}f_{r\delta-\alpha_{j}}-\mathbf{q}(\alpha_{j},\alpha_{j})f_{r\delta-\alpha_{j}}f_{j}) & =(e_{(r\delta,i)},-\mathbf{q}(\alpha_{j},\alpha_{j})f_{r\delta-\alpha_{j}}f_{j})\\
 & =-\mathbf{q}(\alpha_{j},\alpha_{j})(\pi_{j}(\Delta(e_{(r\delta,i)})),f_{r\delta-\alpha_{j}}\otimes f_{j})\\
 & =-\mathbf{q}(\alpha_{j},\alpha_{j})c_{r}^{ij}(e_{r\delta-\alpha_{j}},f_{r\delta-\alpha_{j}})(e_{j},f_{j})\\
 & =-c_{r}^{ij}\frac{\mathbf{q}(\alpha_{j},\alpha_{j})}{(q_{j}^{-1}-q_{j})(q_{j+1}^{-1}-q_{j+1})}
\end{align*}
as $(e_{r\delta-\alpha_{j}},f_{r\delta-\alpha_{j}})=\frac{1}{q_{j+1}^{-1}-q_{j+1}}$
by Lemma \ref{lem:Hopf-pair-bil-form} and \ref{prop:braid-isometry}.
\end{proof}
To compute $c_{r}^{ij}$ we have the following lemmas from \cite[Section 9]{D1}
coming from the commutation relation (\ref{eq:GQG-Drinfeld-hx}) and
hence the same argument applies.
\begin{lem}[{cf. \cite[Lemma 9.5]{D1}}]
 For $r>0$ and $i,j\in\mathring{I}$ such that $p(j)=0$, we have
\[
c_{r}^{ij}=(o(i)o(j))^{r}\frac{\mathbf{q}(\alpha_{i},\alpha_{j})^{r}-\mathbf{q}(\alpha_{i},\alpha_{j})^{-r}}{\mathbf{q}(\alpha_{j},\alpha_{j})^{r}-\mathbf{q}(\alpha_{j},\alpha_{j})^{-r}}c_{r}^{jj}.
\]
Moreover, $c_{r}^{ij}=0$ unless $\left|i-j\right|\leq1$.
\end{lem}

\begin{lem}[{cf. \cite[Lemma 9.7]{D1}}]
 For $i\in\mathring{I}$, $c_{1}^{ii}=(1-\mathbf{q}(\alpha_{i},\alpha_{i})^{-2})$.
\end{lem}

\begin{lem}[{\cite[Lemma 9.6]{D1}}]
 For $r>0$ and $i,j\in\mathring{I}$, we have
\[
r(\mathbf{q}(\alpha_{i},\alpha_{j})-\mathbf{q}(\alpha_{i},\alpha_{j})^{-1})c_{r}^{ij}=(o(i)o(j))^{r-1}(\mathbf{q}(\alpha_{i},\alpha_{j})^{r}-\mathbf{q}(\alpha_{i},\alpha_{j})^{-r})c_{1}^{ij}.
\]
\end{lem}

To sum up, we have 
\[
c_{r}^{ij}=(o(i)o(j))^{r}\frac{\mathbf{q}(\alpha_{i},\alpha_{j})^{r}-\mathbf{q}(\alpha_{i},\alpha_{j})^{-r}}{r\mathbf{q}(\alpha_{j},\alpha_{j})}
\]
for any $r>0$ and $i,j\in\mathring{I}$ with $p(j)=0$, which allows
us to compute $(e_{(r\delta,i)},f_{(r\delta,j)})$ in such cases.
For the remaining case of $p(j)=1$, we may use the following symmetry
which is enough to complete the computation for those $\epsilon$
in the statement of the main theorem.
\begin{prop}
\label{prop:symm-pair-hh}For $r>0$ and $i,j\in\mathring{I}$, we
have $(e_{(r\delta,i)},f_{(r\delta,j)})=(e_{(r\delta,j)},f_{(r\delta,i)})$.
\end{prop}

This is not obvious as $(x,y)=(\Omega(y),\Omega(x))$ is far from
true in general. The idea is to invoke the symmetricity of the bilinear
form on $\mathbf{f}(\epsilon)$ that is related to the Hopf pairing
under
\[
(x,y)=(x^{+},y^{-})\quad\text{for }x,y\in\mathbf{f}(\epsilon)
\]
(Lemma \ref{lem:Hopf-pair-bil-form}). Let $\omega$ be the $\Bbbk$-algebra
involution on $\mathcal{U}(\epsilon)$ defined by $\omega(e_{i})=f_{i}$
and $\omega(k_{i})=k_{i}^{-1}$, so that for $y\in\mathbf{f}(\epsilon)$
we have $y^{-}=\omega(y^{+})$. In general $\omega$ is different
from $\Omega$, but in this specific case we can obtain a nice relation
from a description as an iterated $q$-commutator. Namely, for $x,y\in\mathcal{U}^{+}(\epsilon)$,
we have
\begin{align*}
\omega(\left\llbracket x,y\right\rrbracket ) & =\left\llbracket \omega(x),\omega(y)\right\rrbracket ,\\
\Omega(\left\llbracket x,y\right\rrbracket ) & =\Omega(y)\Omega(x)-\mathbf{q}(\left|x\right|,\left|y\right|)^{-1}\Omega(x)\Omega(y)=-\mathbf{q}(\left|x\right|,\left|y\right|)^{-1}\left\llbracket \Omega(x),\Omega(y)\right\rrbracket 
\end{align*}
and so if there is a nice relation between $\omega(x)$, $\omega(y)$
and $\Omega(x)$, $\Omega(y)$ respectively, then we may obtain a
similar one between $\omega(\left\llbracket x,y\right\rrbracket )$
and $\Omega(\left\llbracket x,y\right\rrbracket )$. For this reason
we will work with $\widetilde{e}_{(r\delta,i)}$ that is defined as
a $q$-commutator, rather than $e_{(r\delta,i)}$. Indeed, as in the
proof of Lemma \ref{lem:pair-imag-cijr}, it is equivalent to verify
$(\widetilde{e}_{(r\delta,i)},\widetilde{f}_{(r\delta,j)})=(\widetilde{f}_{(r\delta,j)},\widetilde{e}_{(r\delta,i)})$. 

By definition,

\[
\widetilde{e}_{(r\delta,i)}=e_{r\delta-\alpha_{i}}e_{i}-\mathbf{q}(-\alpha_{i},\alpha_{i})e_{i}e_{r\delta-\alpha_{i}}=\left\llbracket e_{r\delta-\alpha_{i}},e_{i}\right\rrbracket .
\]
To represent $e_{r\delta-\alpha_{i}}$ as a $q$-commutator, we use
the relation (\ref{eq:GQG-Drinfeld-hx}), which reads in terms of
affine root vectors as (see Remark \ref{rem:compare-rootv-Drinfeld})
\[
[\widetilde{e}_{(\delta,j)},e_{r\delta-\alpha_{i}}]=o(j)o(i)\frac{\mathbf{q}(\alpha_{j},\alpha_{i})-\mathbf{q}(\alpha_{j},\alpha_{i})^{-1}}{q-q^{-1}}e_{(r+1)\delta-\alpha_{i}}
\]
for $r>0$. Taking $j=i\pm1$, we have
\begin{equation}
e_{(r+1)\delta-\alpha_{i}}=e_{r\delta-\alpha_{i}}\widetilde{e}_{(\delta,j)}-\widetilde{e}_{(\delta,j)}e_{r\delta-\alpha_{i}}=\left\llbracket e_{r\delta-\alpha_{i}},\widetilde{e}_{(\delta,j)}\right\rrbracket \label{eq:ind-formula-xir}
\end{equation}
as $\left|\widetilde{e}_{(\delta,i)}\right|=0$. Since $\widetilde{e}_{(\delta,j)}$
is again a $q$-commutator of $e_{\delta-\alpha_{j}}$ and $e_{j}$,
it remains to realize $e_{\delta-\alpha_{j}}$ for each $j\in\mathring{I}$. 
\begin{lem}
\label{lem:formula-x-_i,1}For $j\in\mathring{I}$, we have 
\[
e_{\delta-\alpha_{j}}=\left\llbracket \left\llbracket \cdots\left\llbracket \left\llbracket \left\llbracket \left\llbracket \cdots\left\llbracket e_{0},e_{n-1}\right\rrbracket ,\cdots\right\rrbracket ,e_{j+1}\right\rrbracket ,e_{1}\right\rrbracket ,e_{2}\right\rrbracket ,\cdots\right\rrbracket ,e_{j-1}\right\rrbracket .
\]
\end{lem}

\begin{proof}
This is essentially proved in \cite[Lemma 4.6]{JKP} which we now
recall. The real root vector $e_{\delta-\alpha_{j}}$ is obtained
as $T_{\varpi_{j}}T_{j}^{-1}(e_{j})$, and recall that we are using
the following reduced expression for $t_{\varpi_{j}}$ (\ref{eq:red-exp-translation}):
\[
t_{\varpi_{j}}=\mu^{j}\mathbf{s}_{(n-j,n-1)}\mathbf{s}_{(n-j-1,n-2)}\cdots\mathbf{s}_{(1,j)},\quad\mathbf{s}_{(a,b)}=s_{a}s_{a+1}\cdots s_{b}.
\]
First, we have
\[
T_{\mathbf{s}_{(1,j)}}T_{j}^{-1}(e_{j})=T_{1}T_{2}\cdots T_{j-1}(e_{j})=\left\llbracket \left\llbracket \cdots\left\llbracket e_{1},e_{2}\right\rrbracket ,\dots\right\rrbracket ,e_{j}\right\rrbracket .
\]
We apply $T_{\mathbf{s}(1+r,j+r)}$ to this one by one. Since we have
\[
T_{\mathbf{s}_{(1+r,j+r)}}(e_{k})=\begin{cases}
e_{k} & \text{if }k<r\\
\left\llbracket e_{r+1},e_{r}\right\rrbracket  & \text{if }k=r\\
e_{k+1} & \text{if }k>r
\end{cases}
\]
from Proposition \ref{prop:Tw-on-ef}, it is easy to see that
\[
T_{\mathbf{s}_{(n-j,n-1)}}\cdots T_{\mathbf{s}_{(2,j+1)}}T_{\mathbf{s}_{(1,j)}}T_{j}^{-1}(e_{j})=\left\llbracket \left\llbracket \cdots\left\llbracket \left\llbracket \left\llbracket \left\llbracket \cdots\left\llbracket e_{n-j},e_{n-j-1}\right\rrbracket ,\cdots\right\rrbracket ,e_{1}\right\rrbracket ,e_{n-j+1}\right\rrbracket ,e_{n-j+2}\right\rrbracket ,\cdots\right\rrbracket ,e_{n-1}\right\rrbracket .
\]
Since $\mu$ acts on $\mathcal{U}(\epsilon)$ by $Z:e_{i}\mapsto e_{i+1}$
($i\in I$), the statement follows.
\end{proof}
\begin{rem}
\label{rem:q-bracket-loop-gen-qasa}In the same manner we obtain a
$q$-bracket realization of affine root vectors $E_{r\delta-\alpha_{i}}$
of $U(\epsilon)$. For example, when $\epsilon=\epsilon_{M|N}$, we
have
\begin{align*}
E_{\delta-\alpha_{i}} & =(-1)^{n+\delta(i<M)}\left\llbracket \left\llbracket \cdots\left\llbracket \left\llbracket \left\llbracket \left\llbracket \cdots\left\llbracket E_{0},E_{n-1}\right\rrbracket ,\cdots\right\rrbracket ,E_{j+1}\right\rrbracket ,E_{1}\right\rrbracket ,E_{2}\right\rrbracket ,\cdots\right\rrbracket ,E_{j-1}\right\rrbracket ,\\
E_{(r+1)\delta-\alpha_{i}} & =(-1)^{\epsilon_{i-1}}\left\llbracket E_{r\delta-\alpha_{i}},\left\llbracket E_{\delta-\alpha_{i-1}},E_{i-1}\right\rrbracket \right\rrbracket =(-1)^{\epsilon_{i}}\left\llbracket E_{r\delta-\alpha_{i}},\left\llbracket E_{\delta-\alpha_{i+1}},E_{i+1}\right\rrbracket \right\rrbracket ,\\
\phi_{i,r}^{+} & =(-1)^{ir+\epsilon_{i}+\epsilon_{i+1}}(q-q^{-1})C^{-r/2}K_{i}\left\llbracket E_{r\delta-\alpha_{i}},E_{i}\right\rrbracket 
\end{align*}
where for a statement $\mathrm{P}$ we set $\delta(\mathrm{P})=1$
if $\mathrm{P}$ is true and $\delta(\mathrm{P})=0$ otherwise. 
\end{rem}

\begin{proof}[Proof of Proposition \ref{prop:symm-pair-hh}]
From the lemma, as $\Omega\left\llbracket x,y\right\rrbracket =-\mathbf{q}(x,y)^{-1}\left\llbracket \Omega x,\Omega y\right\rrbracket $,
we have
\begin{align*}
\Omega(e_{\delta-\alpha_{j}}) & =(-1)^{n-2}\left(q_{j-1}^{-1}q_{j-2}^{-1}\cdots q_{2}^{-1}q_{n}^{-1}\cdots q_{j+2}^{-1}q_{1}^{-1}\right)^{-1}\left\llbracket \left\llbracket \cdots\left\llbracket \left\llbracket \left\llbracket \left\llbracket \cdots\left\llbracket f_{0},f_{n-1}\right\rrbracket ,\cdots\right\rrbracket ,f_{j+1}\right\rrbracket ,f_{1}\right\rrbracket ,f_{2}\right\rrbracket ,\cdots\right\rrbracket ,f_{j-1}\right\rrbracket \\
 & =(-1)^{n-2}\left(\prod_{k\in\mathbb{I}\setminus\{j,j+1\}}q_{k}\right)\omega(e_{\delta-\alpha_{j}})
\end{align*}
and so 
\begin{align*}
\Omega(\widetilde{e}_{(\delta,j)}) & =-\mathbf{q}(-\alpha_{j},\alpha_{j})^{-1}\left\llbracket \Omega(e_{\delta-\alpha_{j}}),\Omega(e_{j})\right\rrbracket \\
 & =-q_{j}q_{j+1}\cdot(-1)^{n-2}\left(\prod_{k\in\mathbb{I}\setminus\{j,j+1\}}q_{k}\right)\left\llbracket \omega(e_{\delta-\alpha_{j}}),\omega(e_{j})\right\rrbracket \\
 & =(-1)^{n-1}\left(\prod_{k\in\mathbb{I}}q_{k}\right)\omega(\widetilde{e}_{(\delta,j)}).
\end{align*}
for any $j\in\mathring{I}$. Therefore, we obtain inductively
\begin{align*}
\Omega(e_{r\delta-\alpha_{i}}) & =(-1)^{r(n-2)}q_{i}^{-1}q_{i+1}^{-1}\left(\prod_{k\in\mathbb{I}}q_{k}\right)^{r}\omega(e_{r\delta-\alpha_{i}}),\\
\Omega(\widetilde{e}_{(r\delta,i)}) & =-(-1)^{r(n-2)}\left(\prod_{k\in\mathbb{I}}q_{k}\right)^{r}\omega(\widetilde{e}_{(r\delta,i)}).
\end{align*}

Now take $\widetilde{\theta}_{(r\delta,i)}\in\mathbf{f}(\epsilon)$
such that $\widetilde{\theta}_{(r\delta,i)}^{+}=\widetilde{e}_{(r\delta,i)}$.
From the last computation, we proceed as
\begin{align*}
(\widetilde{e}_{(r\delta,i)},\widetilde{f}_{(r\delta,j)}) & =(\widetilde{e}_{(r\delta,i)},\Omega(\widetilde{e}_{(r\delta,j)}))\\
 & =-(-1)^{r(n-2)}\left(\prod_{k\in\mathbb{I}}q_{k}\right)^{r}(\widetilde{e}_{(r\delta,i)},\omega(\widetilde{e}_{(r\delta,j)}))\\
 & =-(-1)^{r(n-2)}\left(\prod_{k\in\mathbb{I}}q_{k}\right)^{r}(\widetilde{\theta}_{(r\delta,i)}^{+},\widetilde{\theta}_{(r\delta,j)}^{-})\\
 & =-(-1)^{r(n-2)}\left(\prod_{k\in\mathbb{I}}q_{k}\right)^{r}(\widetilde{\theta}_{(r\delta,i)},\widetilde{\theta}_{(r\delta,j)}).
\end{align*}
Since the last one is symmetric in $i$ and $j,$ it is also equal
to $(\widetilde{e}_{(r\delta,j)},\widetilde{f}_{(r\delta,i)})$. 
\end{proof}
Summarizing all the computations in this section, we obtain the following
formula for the pairing of imaginary root vectors.
\begin{thm}
\label{thm:pairing-imag-rootv}If $\epsilon$ is such that $p(i)p(i+1)=0$
for all $1\leq i\leq n-2$, then for $r,s>0$ and $i,j\in\mathring{I}$
we have
\[
(e_{(r\delta,i)},f_{(s\delta,j)})=\begin{cases}
-\delta_{r,s}(o(i)o(j))^{r}\frac{\mathbf{q}(\alpha_{i},\alpha_{j})^{r}-\mathbf{q}(\alpha_{i},\alpha_{j})^{-r}}{r(q_{j}^{-1}-q_{j})(q_{j+1}^{-1}-q_{j+1})} & \text{if }p(j)=0\\
-\delta_{r,s}(o(i)o(j))^{r}\frac{\mathbf{q}(\alpha_{i},\alpha_{j})^{r}-\mathbf{q}(\alpha_{i},\alpha_{j})^{-r}}{r(q_{i}^{-1}-q_{i})(q_{i+1}^{-1}-q_{i+1})} & \text{if }p(j)=1.
\end{cases}
\]
\end{thm}

\begin{rem}
\label{rem:hh-q-Heis} Comparing this with the relation (\ref{eq:GQG-Drinfeld-hh}),
we have
\[
[h_{i,r},h_{j,s}]=\delta_{r,-s}(h_{i,r},h_{j,s})(c^{-2r}-1),
\]
which can be interpreted as a $q$-deformed Heisenberg relation.
\end{rem}

\subsection{\label{subsec:Pairing-PBW}Pairing of PBW vectors}

Now it remains to compute the pairing between general PBW vectors
to obtain the multiplicative formula for the universal $R$-matrix.
Following the argument in the beginning of \cite[Section 10]{D1}
(with Remark \ref{rem:hh-q-Heis}) one can find another basis $\{\overline{e}_{(r\delta,i)}\}_{i\in\mathring{I}}$
for the subspace spanned by $\{e_{(r\delta,i)}\}_{i\in\mathring{I}}$
such that $(\overline{e}_{(r\delta,i)},\Omega\overline{e}_{(r\delta,j)})=0$
whenever $i\neq j$. With this orthogonality, once again thanks to
the formally same comultiplication formula and the root system with
convex order, the analysis of \cite[Section 10]{D1} using LS formula
applies here verbatim and allows us to deduce the following results.
\begin{lem}[{cf. \cite[Lemma 10.3]{D1}}]
 For $\alpha\in\widetilde{\Phi}^{+}$ and $\overrightarrow{\gamma}\in\mathrm{Par}_{\epsilon}$,
if either $(\overline{e}_{\alpha},\overline{f}(-\overrightarrow{\gamma}))$
or $(\overline{e}(-\overrightarrow{\gamma}),\overline{f}_{\alpha})$
is non-zero, then $\overrightarrow{\gamma}=(\alpha)$.
\end{lem}

\begin{prop}[{cf. \cite[Proposition 10.5, Lemma 10.6]{D1}}]
 For $\overrightarrow{\gamma},\overrightarrow{\gamma}^{\prime}\in\mathcal{P}$,
we have 
\[
(\overline{e}(-\overrightarrow{\gamma}),\overline{f}(-\overrightarrow{\gamma}^{\prime}))=\delta_{\overrightarrow{\gamma},\overrightarrow{\gamma}^{\prime}}\prod_{\beta\in\widetilde{\Phi}^{+}}(\overline{e}_{\beta}^{m_{\beta}},\overline{f}_{\beta}^{m_{\beta}})=\delta_{\overrightarrow{\gamma},\overrightarrow{\gamma}^{\prime}}\prod_{\beta\in\Phi_{\mathrm{re}}^{+}}\frac{(m_{\beta})_{\beta}}{(q_{\beta}^{-1}-q_{\beta})^{m_{\beta}}}\prod_{\beta\in\widetilde{\Phi}_{\mathrm{im}}^{+}}m_{\beta}!(\overline{e}_{\beta},\overline{f}_{\beta}).
\]
\end{prop}

To sum up, we obtain a multiplicative formula
\begin{align*}
\mathcal{R}= & \left(\sum_{\overrightarrow{\gamma}\in\mathcal{P}}\frac{\overline{e}(-\overrightarrow{\gamma})\otimes\overline{f}(-\overrightarrow{\gamma})}{(\overline{e}(-\overrightarrow{\gamma}),\overline{f}(-\overrightarrow{\gamma}))}\right)\overline{\Pi}_{\mathbf{q}}=\mathcal{R}^{+}\mathcal{R}^{0}\mathcal{R}^{-}\overline{\Pi}_{\mathbf{q}},\quad\text{ where}\\
\mathcal{R}^{\pm} & =\prod_{\beta\in\Phi^{+}(\pm\infty)}\mathrm{exp}_{\beta}((q_{\beta}^{-1}-q_{\beta})e_{\beta}\otimes f_{\beta}),\quad\quad\exp_{\beta}(X)=\sum_{k\geq0}\frac{X^{k}}{(k)_{\beta}!},\\
\mathcal{R}^{0} & =\prod_{\beta\in\widetilde{\Phi}_{\mathrm{im}}^{+}}\mathrm{exp}\left(\frac{\overline{e}_{\beta}\otimes\overline{f}_{\beta}}{(\overline{e}_{\beta},\overline{f}_{\beta})}\right)=\prod_{r>0}\mathrm{exp}\left(\sum_{i\in\mathring{I}}\frac{\overline{e}_{(r\delta,i)}\otimes\overline{f}_{(r\delta,i)}}{(\overline{e}_{(r\delta,i)},\overline{f}_{(r\delta,i)})}\right)
\end{align*}
where again for $\beta\in\Phi_{\mathrm{re}}^{+}$ such that $p(\beta)=1$,
$e_{\beta}^{2}=0$ and we understand $\mathrm{exp}_{\beta}((q_{\beta}^{-1}-q_{\beta})e_{\beta}\otimes f_{\beta})=1+(q_{\beta}^{-1}-q_{\beta})e_{\beta}\otimes f_{\beta}$.
Now if we let
\[
B^{r}=\left((h_{i,r},h_{j,-r})\right)_{i,j\in\mathring{I}},
\]
and $\widetilde{B}^{r}$ be its inverse, then the last term can be
rewritten as
\[
\sum_{i\in\mathring{I}}\frac{\overline{e}_{(r\delta,i)}\otimes\overline{f}_{(r\delta,i)}}{(\overline{e}_{(r\delta,i)},\overline{f}_{(r\delta,i)})}=\sum_{i,j\in\mathring{I}}\widetilde{B}_{ji}^{r}h_{i,r}\otimes h_{j,-r}.
\]
The matrix $C^{r}$ in the statement of Theorem \ref{thm:mult-formula-R}
is a normalization of $B^{r}$:
\[
B^{r}=\mathrm{diag}\left(-\frac{q_{i}^{r}-q_{i}^{-r}}{r(q^{-1}-q)^{2}}\right)\cdot C^{r}.
\]
Hence $\widetilde{B}_{ji}^{r}=-\frac{r(q-q^{-1})^{2}}{q_{i}^{r}-q_{i}^{-r}}\widetilde{C}_{ji}^{r}$
and we obtain the desired formula

\[
\mathcal{R}^{0}=\exp\left(-\sum_{i,j,r}\frac{r(q-q^{-1})^{2}}{q_{i}^{r}-q_{i}^{-r}}\widetilde{C}_{ji}^{r}h_{i,r}\otimes h_{j,-r}\right).
\]

\subsection{Deformed Cartan matrices for $\mathfrak{sl}_{M|N}$\label{subsec:def-Cartan}}

In this section, we assume that $q$ and $\widetilde{q}$ are two
algebraically independent variables. 
\begin{defn}
Given a $(01)$-sequence $\epsilon=(\epsilon_{i})_{1\leq i\leq n}$
(possibly $M=N$), let
\[
q_{i}=\begin{cases}
q & \text{if }\epsilon_{i}=0\\
\widetilde{q} & \text{if }\epsilon_{i}=1.
\end{cases}
\]
We define a \textit{symmetrized $(q,\widetilde{q})$-deformed Cartan
matrix $B(q,\widetilde{q})$} associated with $\mathfrak{sl}_{M|N}$
and $\epsilon$ by 
\[
B(q,\widetilde{q})=(B_{ij}(q,\widetilde{q}))_{i,j\in\mathring{I}},\quad B_{ii}=\frac{q_{i}^{-1}q_{i+1}^{-1}-q_{i}q_{i+1}}{(q_{i}^{-1}-q_{i})(q_{i+1}^{-1}-q_{i+1})},\quad B_{i-1,i}=B_{i,i-1}=\frac{1}{q_{i}-q_{i}^{-1}}
\]
and the other entries are zero. A \textit{$(q,\widetilde{q})$-deformed
Cartan matrix} $C(q,\widetilde{q})$ is its normalization
\[
C(q,\widetilde{q})=D(q,\widetilde{q})^{-1}B(q,\widetilde{q}),\quad D_{ij}(q,\widetilde{q})=\frac{\delta_{ij}}{q_{i}^{-1}-q_{i}}.
\]
\end{defn}

For instance, the $(q,\widetilde{q})$-deformed Cartan matrices for
$\epsilon=(010),(101),(0011)$ are
\[
\begin{pmatrix}\frac{q\widetilde{q}-q^{-1}\widetilde{q}^{-1}}{\widetilde{q}-\widetilde{q}^{-1}} & -\frac{q-q^{-1}}{\widetilde{q}-\widetilde{q}^{-1}}\\
-1 & \frac{q\widetilde{q}-q^{-1}\widetilde{q}^{-1}}{q-q^{-1}}
\end{pmatrix},\quad\begin{pmatrix}\frac{q\widetilde{q}-q^{-1}\widetilde{q}^{-1}}{q-q^{-1}} & -\frac{\widetilde{q}-\widetilde{q}^{-1}}{q-q^{-1}}\\
-1 & \frac{q\widetilde{q}-q^{-1}\widetilde{q}^{-1}}{\widetilde{q}-\widetilde{q}^{-1}}
\end{pmatrix},\quad\begin{pmatrix}q+q^{-1} & -1\\
-1 & \frac{q\widetilde{q}-q^{-1}\widetilde{q}^{-1}}{\widetilde{q}-\widetilde{q}^{-1}} & -\frac{q-q^{-1}}{\widetilde{q}-\widetilde{q}^{-1}}\\
 & -1 & \widetilde{q}+\widetilde{q}^{-1}
\end{pmatrix}
\]
respectively (we will always suppress $\epsilon$ from the notation
$C(q,\widetilde{q})$). Observe that the matrix $\widetilde{C}^{r}$
in the multiplicative formula for the universal $R$-matrix is indeed
the inverse of the specialization $C(q^{r},(-q^{-1})^{r})$. Moreover,
$C(q^{r},(-q^{-1})^{r})$ is \textit{not} equal to the substitution
$q=q^{r}$ of $C(q,-q^{-1})$, hence our introduction of the second
parameter $\widetilde{q}$ is genuine.

The symmetrized deformed Cartan matrix $B(q,\widetilde{q})$ is extrapolated
from our computation of $(h_{i,1},h_{j,-1})$, where the cases $p(i)=p(j)=1$
excluded in Theorem \ref{thm:pairing-imag-rootv} are handled by considering
$(\widetilde{\theta}_{(\delta,i)},\widetilde{\theta}_{(\delta,j)})$
in $\mathbf{f}(\epsilon)$ instead. By inspection, it coincides with
the deformed Cartan matrix \cite{FJMV} defined by the contraction
of root currents (up to the overall factor $t_{1}t_{2}t_{3}$) defined
on the tensor product 
\[
\mathcal{F}_{c_{1}}(u_{1})\otimes\mathcal{F}_{c_{2}}(u_{2})\otimes\cdots\otimes\mathcal{F}_{c_{n}}(u_{n})
\]
of $n$ Fock modules of the quantum toroidal algebra $\mathcal{E}$
associated with $\mathfrak{gl}_{1}$ (see \cite[Section 2,3]{FJMV}
for definitions), where $c_{i}=3-\epsilon_{i}$ and the quantum parameters
$\mathrm{q}_{1},\mathrm{q}_{2},\mathrm{q}_{3}$ therein correspond
to $q^{-2}\widetilde{q}^{-2},\widetilde{q}^{2},q^{2}$ respectively
in our setting. Their normalization \cite[Section 3.3]{FJMV} is slightly
but not essentially different from ours which is chosen to have the
formulas for $q$-characters below as similar as possible to the original
construction \cite{FR2}.

In the non-super case $N=0$, this picture recovers the connection
between the representation theory of the quantum affine algebra $U_{q}(\widehat{\mathfrak{sl}}_{M})$
and the $q$-deformed $W$-algebra $\mathcal{W}_{q}(\mathfrak{sl}_{M})$
\cite[Section 8]{FR2}. That is, the $(q,t)$-deformed Cartan matrix
of type $A$ \cite{FR1} is exactly the one obtained by subtitution
$\mathrm{q}_{3}=q^{2}t^{-2}$. The relation then follows from their
description by the screening operators defined from the root currents
above at the `classical limit' $t=1$. Note that in this non-super
case $B(q,\widetilde{q})$ and $C(q,\widetilde{q})$ do not involve
$\widetilde{q}$, we may make an auxiliary choice $\mathrm{q_{2}=q^{-2}}$
to reinterpret the limit as a specialization $\mathrm{q}_{1}=1$. 

Now we understand our specialization $\widetilde{q}=-q^{-1}$ in the
general super cases in the same vein, as it corresponds to $\mathrm{q}_{1}=1$.
We will see that the matrix $C(q,-q^{-1})$ indeed governs the theory
of $q$-characters (hence representation theory) for $\mathcal{U}(\epsilon)$
in Section \ref{sec:FR-FM-qchar}, while such a connection for quantum
affine superalgebras $U(\epsilon)$ under $\mathrm{q}_{1}=1$ is claimed
in their subsequent works, see \cite[Section 5]{FJM1}.

Indeed, there is another specialization $\widetilde{q}=q^{-1}$ which
is expected to appear in the study of the quantum affine superalgebra
$U(\epsilon)$ itself from the comparison result with $\mathcal{U}(\epsilon)$.
In fact, it is $C(q,q^{-1})$ that recovers the Cartan matrix $C_{ij}=(-1)^{\epsilon_{i}}(\alpha_{i},\alpha_{j})$
of $\mathfrak{sl}_{M|N}$ associated with $\epsilon$ by taking $q=1$.
Nevertheless, it turns out that $U(\epsilon)$ possesses the same
$q$-character theory as that of $\mathcal{U}(\epsilon)$ (Remark
\ref{rem:compare-l-wt-GQG-qasa}) which may be understood from the
fact that both $\widetilde{q}=\pm q^{-1}$ give rise to the same specialization
$\mathrm{q}_{1}=1$ in the toroidal realization above. Our approach
through $\mathcal{U}(\epsilon)$ (hence the specialization $\widetilde{q}=-q^{-1}$)
thus reveals the $2$-parameter, rather than super, nature of the
representation theory of the quantum affine superalgebra at the Grothendieck
ring level, which is hidden in $U(\epsilon)$ as $C(q^{r},q^{-r})$
can be written as $F(q^{r})$ for some single variable $F(t)\in\mathbb{C}[t^{\pm1}]$.

For later use, we record formulas for the deformed Cartan matrices
for $\epsilon=\epsilon_{M|N}$ which can be verified by a straightforward
calculation.
\begin{prop}
\label{prop:inv-qCartan}Assume $\epsilon=\epsilon_{M|N}$ with $M,N>0$. 
\begin{enumerate}
\item The determinant of $C(q,\widetilde{q})$ is given by
\[
d_{M,N}=\frac{q^{M}\widetilde{q}^{N}-q^{-M}\widetilde{q}^{-N}}{\widetilde{q}-\widetilde{q}^{-1}}.
\]
In particular, $C(q,\pm q^{-1})$ is invertible if and only if $M\neq N$.
\item Suppose $M\neq N$. The $(i,j)$-th entry of the matrix $d_{M,N}C(q,\widetilde{q})^{-1}$
is given by
\begin{align*}
 & \frac{(q^{k}-q^{-k})(q^{M-K}\widetilde{q}^{N}-q^{K-M}\widetilde{q}^{-N})}{(q-q^{-1})(\widetilde{q}-\widetilde{q}^{-1})} & i,j\leq M\\
 & \frac{(q^{j}-q^{-j})(\widetilde{q}^{M+N-i}-\widetilde{q}^{i-M-N})}{(q-q^{-1})(\widetilde{q}-\widetilde{q}^{-1})} & j\leq M<i\\
 & \frac{(q^{i}-q^{-i})(\widetilde{q}^{M+N-j}-\widetilde{q}^{j-M-N})}{(\widetilde{q}-\widetilde{q}^{-1})^{2}} & i\leq M<j\\
 & \frac{(q^{M}\widetilde{q}^{k-M}-q^{-M}\widetilde{q}^{M-k})(\widetilde{q}^{M+N-K}-\widetilde{q}^{K-M-N})}{(\widetilde{q}-\widetilde{q}^{-1})^{2}} & M<i,j
\end{align*}
where $k=\min(i,j)$ and $K=\max(i,j)$.
\end{enumerate}
\end{prop}

As a final remark, recall that the symmetric matrix $B(q,-q^{-1})$
can be realized in terms of the bilinear from on $\mathbf{f}(\epsilon)$.
One may hope to obtain $B(q,\widetilde{q})$ directly from $\mathbf{f}(\epsilon)$
which seems to yield a $2$-parameter version, unlike the full quantum
group $\mathcal{U}(\epsilon)$. Indeed, the construction of $\mathbf{f}(\epsilon)$
in Section \ref{subsec:Lusztig-f-bil-form} can be performed over
the base field $\mathbb{Q}(q,\widetilde{q})$. It turns out that the
2-parameter version of the relations (\ref{eq:GQG-Drinfeld-quadSerre}
-- \ref{eq:GQG-Drinfeld-oSerre}) except $\theta_{i}^{2}=0$ ($p(i)=1$),
defined formally by the same formula, still lie in the radical of
the bilinear form on $^{\prime}\mathbf{f}(\epsilon)$. While it is
not clear at this stage if the braid action $T_{i}$ can be defined
(in the sense of Lemma \ref{lem:T_i-on-Lusztig-f}) for this $\mathbb{Q}(q,\widetilde{q})$-algebra,
we may introduce $\widetilde{\theta}_{(\delta,i)}^{+}$ in terms of
iterated $q$-brackets as in Section \ref{subsec:Pairing-imag-rvector}. 
\begin{conjecture}
For any given $\epsilon$ (possibly $M=N$), we have $B_{ij}(q,\widetilde{q})=(\widetilde{\theta}_{(\delta,i)},\widetilde{\theta}_{(\delta,j)})$
up to a scalar multiple that does not depend on $i,j$.
\end{conjecture}

\section{The $q$-character map for $\mathcal{U}(\epsilon)$\label{sec:RT-GQG}}

From now on, we extend the base field $\Bbbk=\mathbb{Q}(q)$ to its
algebraic closure $\mathbf{k}$ inside $\bigcup_{t\in\mathbb{Z}_{>0}}\mathbb{C}(q^{1/t})$.
We also assume $\epsilon=\epsilon_{M|N}$ in the remaining of this
article.

\subsection{Finite-dimensional $\mathcal{U}(\epsilon)$-modules and their $\ell$-weights\label{subsec:classification-irrep}}
\begin{prop}
Every finite-dimensional simple $\mathcal{U}(\epsilon)$-modules can
be obtained from a module on which $c^{1/2}$ acts trivially by twisting
by an automorphism of $\mathcal{U}(\epsilon)$.
\end{prop}

\begin{proof}
Thanks to the relation (\ref{eq:GQG-Drinfeld-hh}), the proof in \cite[Proposition 3.2]{MY}
for modules of quantum affine algebras with finite-dimensional weight
spaces works verbatim.
\end{proof}
Henceforth we will only concern such $\mathcal{U}(\epsilon)$-modules
as in the non-super case, which also allows us to replace $\mathcal{U}(\epsilon)$
by its quotient by $c-1$ (denoted by the same symbol). Accordingly
we will assume they have the $\mathring{P}$-weight space decomposition:

\[
V=\bigoplus_{\lambda\in\mathring{P}}V_{\lambda},\quad V_{\lambda}=\{v\in V\,|\,k_{i}v=\mathbf{q}(\alpha_{i},\lambda)v\:^{\forall}i\in I\}.
\]
Note that while $\mathcal{U}(\epsilon)$ is an analogue of $U_{q}(\widehat{\mathfrak{sl}}_{M|N})$,
we are concerning $\mathfrak{gl}$-weighted modules.

Let $D^{\pm}=\mathbb{C}\cdot v_{\pm}$ be the $\mathcal{U}(\epsilon)$-module
where $e_{i},f_{i}$ act trivially and the weight of $v_{\pm}$ is
$\pm(\delta_{1}+\cdots+\delta_{M}-\delta_{M+1}-\cdots-\delta_{M+N})$.
In particular, $k_{i}\cdot v_{\pm}=(-1)^{\delta_{i,M}}v_{\pm}$. Then
the tensor product by $D^{+}$ has the inverse $D^{-}\otimes-$. The
1-dimensional module $D^{+}$ is an analogue of the determinant representation.
However, unlike the non-super case, it is not enough to consider only
the polynomial modules (whose weights are contained in $\bigoplus\mathbb{Z}_{\geq0}\delta_{i}$).
For instance, the dual of a polynomial module is not necessarily a
polynomial module even after taking tensor products with $D^{\pm}$.
This is one of the reasons why it is more natural to consider $\mathfrak{gl}$-weighted
modules (\textit{cf.} the last paragraph of Section \ref{subsec:univ-coeff-R}).

The elements $\psi_{i,\pm r}^{\pm}$ ($i\in\mathring{I},\,r\in\mathbb{Z}$)
acting on $\mathcal{U}(\epsilon)$-modules commute with each others,
which allows us to consider a refined decomposition into the simultaneous
generalized eigenspaces 
\begin{align*}
V & =\bigoplus_{\Psi=(\Psi_{i}^{\pm}(z))_{i\in\mathring{I}}}V_{\Psi},\quad\Psi_{i}^{\pm}(z)=\sum_{r\geq0}\Psi_{i,\pm r}^{\pm}z^{\pm r}\in\Bbbk\left\llbracket z^{\pm1}\right\rrbracket ,\\
V_{\Psi} & =\left\{ v\in V\,|\,^{\exists}p>0\text{ such that }^{\forall}i\in\mathring{I},\,^{\forall}r\in\mathbb{Z}_{\geq0},\,(\psi_{i,\pm r}^{\pm}-\Psi_{i,\pm r}^{\pm})^{p}v=0\right\} .
\end{align*}
We say $\Psi=(\Psi_{i}^{\pm}(z))$ is an $\ell$-weight of $V$ if
$V_{\Psi}\neq0$. Note that by definition, for $v\in V_{\Psi}$ we
have $k_{i}^{\pm}v=\Psi_{i,0}^{\pm}v$ and so $\Psi_{i,0}^{\pm}$
is of the form $q_{i}^{\lambda_{i}}\widetilde{q}_{i+1}^{\lambda_{i+1}}$
for some $\lambda_{i}\in\mathbb{Z}$, in which case $V_{\Psi}\subset V_{\lambda}$
for $\lambda=\sum\lambda_{i}\delta_{i}\in\mathring{P}$.

More generally, an \textit{$\ell$-weight} refers to any $\mathring{I}$-tuple
of sequences of scalars
\[
\Psi=(\Psi_{i}^{\pm}(z))_{i\in\mathring{I},r\geq0},\quad\Psi_{i}^{\pm}(z)=\sum_{r\geq0}\Psi_{i,\pm r}^{\pm}z^{\pm r}
\]
such that $\Psi_{i,0}^{\pm}\in\pm q^{\mathbb{Z}}$ and $\Psi_{i,0}^{+}\Psi_{i,0}^{-}=1$
for all $i$. We define the multiplication of $\ell$-weights $\Psi$,
$\Psi^{\prime}$ as $(\Psi\Psi^{\prime})_{i}^{\pm}(z)=\Psi_{i}^{\pm}(z)\Psi_{i}^{\prime\pm}(z)$. 
\begin{rem}
\label{rem:compare-l-wt-GQG-qasa}For modules over the quantum affine
superalgebra $U(\epsilon)$ we can define the notion of $\ell$-weights
in the same manner with respect to $\phi_{i,r}^{\pm}\in U(\epsilon)$.
These notions are compatible under the isomorphism $\tau:\mathcal{U}(\epsilon)[\sigma]\rightarrow U(\epsilon)[\sigma]$.
Indeed, any $\mathring{P}$-graded $\mathcal{U}(\epsilon)$-module
$V$ naturally carries the extended action of $\mathcal{U}(\epsilon)[\sigma]$
by 
\[
\sigma_{i}\restriction_{V_{\lambda}}=(-1)^{\epsilon_{i}(\delta_{i},\lambda)}
\]
and so we may consider a $U(\epsilon)$-module $V^{\tau}$ obtained
from $V$ through $\tau$, and vice versa. Then the $\ell$-weight
space decompositions for $V$ and $V^{\tau}$ coincide, in the sense
that the $\ell$-weight space $V_{\Psi}$ of the $\mathcal{U}(\epsilon)$-module
$V$ is also the generalized eigenspace of $\sum\phi_{i,\pm r}^{\pm}z^{\pm r}$
with respect to the eigenvalue
\[
(-1)^{(\epsilon_{i}\delta_{i}+\epsilon_{i+1}\delta_{i+1},\lambda)}\left(\Psi_{i,0}^{\pm}+\sum_{r>0}\Psi_{i,\pm r}^{\pm}z^{\pm r}(-1)^{rn+r(M+i+1)\epsilon_{i}}\right)
\]
by Proposition \ref{prop:compare-Cartan-current}. Therefore, to find
all $\ell$-weights (with multiplicity) of a $\mathcal{U}(\epsilon)$-module
is equivalent to do it for $V^{\tau}$, which justifies our approach
to the representation theory of quantum affine superalgebras through
$\mathcal{U}(\epsilon)$.
\end{rem}

\begin{defn}
A $\mathcal{U}(\epsilon)$-module $V$ is called a highest $\ell$-weight
module of the highest $\ell$-weight $\Psi=(\Psi_{i,\pm r}^{\pm})$
if it is generated by a \textit{highest $\ell$-weight vector} $v\in V_{\Psi}$,
that is
\begin{enumerate}
\item $x_{i,k}^{+}v=0$ for all $i\in\mathring{I}$ and $k\in\mathbb{Z}$,
\item $\psi_{i,\pm r}^{\pm}v=\Psi_{i,\pm r}^{\pm}v$ for all $i\in\mathring{I},\,r\in\mathbb{Z}_{\geq0}$.
\end{enumerate}
\end{defn}

Suppose we are given an $\ell$-weight $\Psi$. Thanks to the triangular
decomposition (\ref{eq:GQGDr-tri-decomp}) with respect to loop generators,
we may construct the simple highest $\ell$-weight module $L(\Psi)$
of highest weight $\Psi$ as the unique simple quotient of the universal
highest $\ell$-weight module. Since the $U(\epsilon)$-module $L(\Psi)^{\tau}$
is also a simple highest $\ell$-weight module, the condition for
$L(\Psi)$ to be finite-dimensional can be read from the highest $\ell$-weight
classification of finite-dimensional simple $U(\epsilon)$-modules
given in \cite{Z1}.
\begin{thm}[{\cite[Proposition 4.15]{Z1}}]
 Up to the parity shift, the finite-dimensional simple $U(\epsilon)$-modules
are exactly the simple highest $\ell$-weight modules generated by
$v$ such that $X_{i,k}^{+}v=0$ for all $i\in\mathring{I}$ and $k\in\mathbb{Z}$
and the eigenvalue $\phi_{i}^{\pm}(z)v=\Phi_{i}^{\pm}(z)v$ is an
expansion of a rational function $f_{i}(z)\in\mathbf{k}(z)$ at $z^{\pm1}=0$
which is of the form 
\[
f_{i}(z)=\begin{cases}
\text{product of }\ensuremath{q_{i}\frac{1-q_{i}^{-1}az}{1-q_{i}az}}\text{ with }a\in\mathbf{k}^{\times} & \text{if }i\neq M,\\
\text{product of }\ensuremath{c\frac{1-c^{-1}az}{1-caz}}\text{ with }c,a\in\mathbf{k}^{\times} & \text{if }i=M.
\end{cases}
\]
\end{thm}

\begin{rem}
\label{rem:intble-fails}Compared with the classification for quantum
affine algebras \cite{CP}, the condition on $\Phi_{i}^{\pm}(z)$
when $p(i)=1$ is relaxed. Recall that the local nilpotency of $e_{i}$
and $f_{i}$ is essential in establishing various results for finite-dimensional
modules of quantum affine algebras. However in the super case, it
is automatically satisfied whenever $p(i)=1$ and so not as strong
as in the non-super case. Indeed, this classification result can be
understood in connection with the one for the simples in the affinization
of a parabolic BGG category $\mathcal{O}$ for the quantum affine
algebra $U_{q}(\widehat{\mathfrak{sl}}_{n})$ where the integrability
condition is relaxed \cite[Theorem 3.7]{MY}.
\end{rem}

\begin{defn}
\label{def:fund-l-wt}For $i\in\{1,\dots,M\}$ and $a\in\mathbf{k}^{\times}$,
let $Y_{i,a}$ denote the $\ell$-weight given by
\[
(Y_{i,a})_{k}^{\pm}(z)=\delta_{i,k}q\frac{1-q^{-1}az}{1-qaz}+(1-\delta_{i,k})
\]
and similarly for $i\in\{M,\dots,M+N-1\}$, $\widetilde{Y}_{i,a}$
the one given by
\[
(\widetilde{Y}_{i,a})_{k}^{\pm}(z)=\delta_{i,k}\widetilde{q}\frac{1-\widetilde{q}^{-1}az}{1-\widetilde{q}az}+(1-\delta_{i,k}).
\]
Note that $D\coloneqq Y_{M,a}\widetilde{Y}_{M,-a}$ is the $\ell$-weight
of the 1-dimensional module $D^{+}$ above and independent of $a$.

We also introduce \textit{simple $\ell$-roots} for $i\in\mathring{I}$
\begin{equation}
A_{i,a}=\begin{cases}
Y_{i,aq}Y_{i,aq^{-1}}Y_{i-1,a}^{-1}Y_{i+1,a}^{-1} & \text{if }i<M\\
\widetilde{Y}_{i,a\widetilde{q}}\widetilde{Y}_{i,a\widetilde{q}^{-1}}\widetilde{Y}_{i-1,a}^{-1}\widetilde{Y}_{i+1,a}^{-1} & \text{if }i>M\\
Y_{M,a}\widetilde{Y}_{M,-a}Y_{M-1,a}^{-1}\widetilde{Y}_{M+1,a}^{-1}=DY_{M-1,a}^{-1}\widetilde{Y}_{M+1,a}^{-1} & \text{if }i=M.
\end{cases}\label{eq:def-A-qCartan}
\end{equation}
In particular, $A_{M,a}$ represents the $\ell$-weight
\[
\left(1,\dots,1,q^{-1}\frac{1-qaz}{1-q^{-1}az},-1,\widetilde{q}^{-1}\frac{1-\widetilde{q}az}{1-\widetilde{q}^{-1}az},1,\dots,1\right)
\]
where $-1$ lies at the $M$-th component. The weight of $Y_{i,a}$
(resp. $\widetilde{Y}_{j,b}$) is defined to be $\delta_{1}+\cdots+\delta_{i}$
(resp. $-\delta_{j+1}-\cdots-\delta_{n}$) so that the weight of $A_{i,a}$
is $\alpha_{i}=\delta_{i}-\delta_{i+1}$ for all $i\in\mathring{I}$.
The lower index $a\in\mathbf{k}^{\times}$ in $Y$ and $A$ is called
the \textit{spectral parameter}. 
\end{defn}

As its name suggests, $A_{i,a}$ plays the role of simple roots in
the sense that any $\ell$-weight of a given (finite-dimensional)
highest $\ell$-weight module can be written as a monomial in $A_{i,a}^{-1}$'s
times the highest $\ell$-weight (Corollary \ref{cor:q-char-cone-prop}).
This important property in the non-super case was first shown in \cite[Theorem 4.1]{FM}
using the restriction map below (Section \ref{subsec:restriction-rk12}),
but later streamlined and naturally extended to the affinization of
the BGG category by establishing the following general result \cite{MY}. 
\begin{thm}[{cf. \cite[Proposition 3.9]{MY}}]
\label{thm:integral-l-wt} Suppose $\Psi$ and $\Psi^{\prime}$ are
$\ell$-weights of a simple $\mathcal{U}(\epsilon)$-module $V$ such
that 
\[
V_{\Psi^{\prime}}\cap\bigoplus_{r\in\mathbb{Z}}x_{i,r}^{\pm}V_{\Psi}\neq0
\]
for some $i\in I$. Then we have $\Psi^{\prime}=\Psi A_{i,a}^{\pm1}$
for some $a\in\mathbf{k}^{\times}$.
\end{thm}

\begin{proof}
While the proof is completely parallel, we reproduce it here for reader's
convenience. Let $(v_{k})$ be an ordered basis for $V_{\Psi}$ such
that
\[
(\psi_{i}^{\pm}(z)-\Psi_{i}^{\pm}(z))v_{k}=\sum_{k^{\prime}<k}\xi_{i}^{\pm,k,k^{\prime}}(z)v_{k^{\prime}}
\]
for some $\xi_{i}^{\pm,k,k^{\prime}}(z)\in z^{\pm1}\mathbf{k}\left\llbracket z^{\pm}\right\rrbracket $,
and $(w_{l})$ be the one for $V_{\Psi^{\prime}}$ such that
\[
(\psi_{i}^{\pm}(z)-\Psi_{i}^{\prime\pm}(z))w_{l}=\sum_{l^{\prime}>l}\zeta_{i}^{\pm,l,l^{\prime}}(z)w_{l^{\prime}}
\]
for some $\zeta_{i}^{\pm,l,l^{\prime}}(z)\in z^{\pm}\mathbf{k}\left\llbracket z^{\pm}\right\rrbracket $.
That is, the action of $\psi_{i,\pm r}^{\pm}$ is upper-triangular
(resp. lower-triangular) with respect to $(v_{k})$ (resp. $(w_{l})$).

Let $x_{i}^{+}(z)=\sum x_{i,r}^{+}z^{-r}$. The defining relation
(\ref{eq:GQG-Drinfeld-hx}) can be rewritten as
\[
(\mathbf{q}(\alpha_{i},\alpha_{j})-z_{1}z_{2})x_{i}^{+}(z_{1})\psi_{j}^{+}(z_{2})=(1-\mathbf{q}(\alpha_{i},\alpha_{j})z_{1}z_{2})\psi_{j}^{+}(z_{2})x_{i}^{+}(z_{1})
\]
which implies 
\begin{align*}
(\mathbf{q}(\alpha_{i},\alpha_{j})-z_{1}z_{2}) & x_{i}^{+}(z_{1})\left(\psi_{j}^{+}(z_{2})-\Psi_{j}^{+}(u)\right)v_{k}\\
 & =\left((1-\mathbf{q}(\alpha_{i},\alpha_{j})z_{1}z_{2})\psi_{j}^{+}(z_{2})-(\mathbf{q}(\alpha_{i},\alpha_{j})-z_{1}z_{2})\Psi_{j}^{+}(z_{2})\right)x_{i}^{+}(z_{1})v_{k}.
\end{align*}
Hence, if we write the $V_{\Psi^{\prime}}$-component of $x_{i}^{+}(z)v_{k}$
(with respect to the $\ell$-weight space decomposition) as $\sum_{l}\lambda_{k,l}(z)w_{l}$
for $\lambda_{k,l}(z)\in\mathbf{k}\left\llbracket z^{\pm1}\right\rrbracket $,
$\lambda_{k,l}(z)$ is subject to the following equation
\begin{align*}
 & (\mathbf{q}(\alpha_{i},\alpha_{j})-z_{1}z_{2})\sum_{k^{\prime}<k}\lambda_{k^{\prime},l}(z_{1})\xi_{j}^{+,k,k^{\prime}}(z_{2})\\
= & \left((1-\mathbf{q}(\alpha_{i},\alpha_{j})z_{1}z_{2})\Psi_{j}^{\prime+}(z_{2})-(\mathbf{q}(\alpha_{i},\alpha_{j})-z_{1}z_{2})\Psi_{j}^{+}(z_{2})\right)\lambda_{k,l}(z_{1})+(1-\mathbf{q}(\alpha_{i},\alpha_{j})z_{1}z_{2})\sum_{l^{\prime}=1}^{l-1}\zeta_{j}^{+,l^{\prime},l}(z_{1})\lambda_{k,l^{\prime}}(z_{2})
\end{align*}
obtained as the coefficients of $w_{l}$ of the $V_{\Psi^{\prime}}$-component
of the identity above.

Now assume $V_{\Psi^{\prime}}\cap\bigoplus_{r\in\mathbb{Z}}x_{i,r}^{\pm}V_{\Psi}\neq0$,
so that there is the smallest $K$ such that the $V_{\Psi^{\prime}}$-component
of $x_{i}^{+}(z)v_{K}$ is non-zero, and then the smallest $L$ with
$\lambda_{K,L}(z)\neq0$. For such $K$ and $L$ the equation reduces
to
\begin{align*}
0 & =\left((1-\mathbf{q}(\alpha_{i},\alpha_{j})z_{1}z_{2})\Psi_{j}^{\prime+}(z_{2})-(\mathbf{q}(\alpha_{i},\alpha_{j})-z_{1}z_{2})\Psi_{j}^{+}(z_{2})\right)\lambda_{K,L}(z_{1})\\
 & =\lambda_{K,L}(z_{1})\sum_{t=0}^{\infty}z_{2}^{t}\left((\Psi_{j,t}^{\prime+}-\mathbf{q}(\alpha_{i},\alpha_{j})\Psi_{j,t}^{+})-(\mathbf{q}(\alpha_{i},\alpha_{j})\Psi_{j,t-1}^{\prime+}-\Psi_{j,t-1}^{+})z_{1}\right)
\end{align*}
which imposes an infinite set of linear recurrences on the coefficients
of $\lambda_{K,L}(z_{1})\in\mathbf{k}[z_{1}^{\pm1}]$, for each $j\in J$.
Hence, it has a non-trivial solution if and only if there exists $a\in\mathbf{k}$
such that 
\[
\frac{\Psi_{j,t}^{\prime+}-\mathbf{q}(\alpha_{i},\alpha_{j})\Psi_{j,t}^{+}}{\mathbf{q}(\alpha_{i},\alpha_{j})\Psi_{j,t-1}^{\prime+}-\Psi_{j,t-1}^{+}}=a
\]
for all $t\geq0$ and $j\in\mathring{I}$, equivalently
\[
\Psi_{j}^{\prime+}(z)=\mathbf{q}(\alpha_{i},\alpha_{j})\frac{1-\mathbf{q}(\alpha_{i},\alpha_{j})^{-1}az}{1-\mathbf{q}(\alpha_{i},\alpha_{j})az}\Psi_{j}^{+}(z)=A_{i,a}(z)_{j}\Psi_{j}^{+}(z)
\]
for all $j$.
\end{proof}
\begin{cor}
\label{cor:q-char-cone-prop} For any $\ell$-weight $\Psi^{\prime}$
of a simple $\mathcal{U}(\epsilon)$-module $L(\Psi)$, $\Psi^{-1}\Psi^{\prime}$
is a monomial in $A_{i,a}^{-1}$ for $i\in\mathring{I}$ and $a\in\mathbf{k}^{\times}$.
\end{cor}

\subsection{The category $\mathcal{F}(\epsilon)$ and fundamental representations\label{subsec:fund-repn}}

To develop the theory of $q$-characters, we introduce the category
of $\mathcal{U}(\epsilon)$-modules with \textit{integral} $\ell$-weights.
\begin{defn}
The category $\mathcal{F}(\epsilon)$ consists of the $\mathcal{U}(\epsilon)$-modules
whose $\ell$-weights are monomials in $Y_{i,a}^{\pm1}$ and $\widetilde{Y}_{j,a}^{\pm1}$
for $i\leq M\leq j$, $a\in\mathbf{k}^{\times}$.
\end{defn}

Introduce a Laurent polynomial ring 
\begin{align*}
\mathscr{Y}(\epsilon) & =\mathbb{Z}[Y_{i,a}^{\pm1},\widetilde{Y}_{j,a}^{\pm1}]_{i\leq M\leq j,a\in\mathbf{k}^{\times}}/\left(Y_{M,a}\widetilde{Y}_{M,-a}-Y_{M,b}\widetilde{Y}_{M,-b},\,^{\forall}a,b\in\mathbf{k}^{\times}\right)\\
 & \cong\mathbb{Z}[Y_{i,a}^{\pm1},\widetilde{Y}_{j,a}^{\pm1},D^{\pm1}]_{i\leq M<j,a\in\mathbf{k}^{\times}}
\end{align*}
where $D$ represents $Y_{M,a}\widetilde{Y}_{-a}$. Let $K(\mathcal{F}(\epsilon))$
be the Grothendieck group of the abelian category $\mathcal{F}(\epsilon)$.
Then we have a well-defined map
\begin{align*}
\chi_{q}:K(\mathcal{F}(\epsilon)) & \longrightarrow\mathscr{Y}(\epsilon)\\{}
[V] & \longmapsto\sum\dim(V_{\Psi})\cdot\Psi
\end{align*}
which we call the \textit{$q$-character} map. For example, $\chi_{q}(D^{\pm})=D^{\pm1}$.
Since the $\ell$-weights of $L(\Psi)$ has lower weights than $\Psi$
and $Y_{i,a}$, $\widetilde{Y}_{j,a}$ ($i\leq M<j$) and $D$ are
algebraically independent, $\chi_{q}$ is injective. Finally, thanks
to the comultiplication formula (Theorem \ref{thm:comult-imag-rootv}),
it can be shown as in \cite[Proposition 1]{FR2} that $\mathcal{F}(\epsilon)$
is a monoidal category (and so $K(\mathcal{F}(\epsilon))$ has a multiplication)
and $\chi_{q}$ is a ring homomorphism.
\begin{prop}
\label{prop:q-char-mult}The $q$-character map is an injective ring
homomorphism. In particular, $K(\mathcal{F}(\epsilon))$ is commutative.
\end{prop}

\begin{rem}
To compare with the usual notion of $q$-character for quantum affine
superalgebra $U(\epsilon)$ (see for example \cite{Z5}), let $\mathrm{Y}_{i,a}$
denote the $\ell$-weight 
\[
(\mathrm{Y}_{i,a})_{k}=\delta_{i,k}q^{(-1)^{\epsilon_{i}}}\frac{1-q^{-(-1)^{\epsilon_{i}}}az}{1-q^{(-1)^{\epsilon_{i}}}az}+(1-\delta_{i,k})
\]
for $i=1,\dots,n-1$ and consider the $\ell$-weight space $V_{\mathrm{Y}_{i,a}}$of
a $U(\epsilon)$-module $V$. According to Remark \ref{rem:compare-l-wt-GQG-qasa},
it is also an $\ell$-weight space of the $\mathcal{U}(\epsilon)$-module
$V^{\tau^{-1}}$ of the $\ell$-weight
\[
\begin{cases}
Y_{i,(-1)^{n}} & \text{if }i\leq M\\
\widetilde{Y}_{i,(-1)^{n+M+i+1}} & \text{if }i>M.
\end{cases}
\]
Therefore, if we introduce the category $F(\epsilon)$ of $U(\epsilon)$-modules
whose $\ell$-weights are monomials in $\mathrm{Y}_{i,a}^{\pm1}$
($1\leq i\leq n-1$, $a\in\mathbf{k}^{\times}$) and define the $q$-character
$\chi_{q}(V)$ of each $V\in F(\epsilon)$ as above, then one may
easily recover $\chi_{q}(V^{\tau^{-1}})$ from $\chi_{q}(V)$ and
vice versa. In this sense, we understand that the theories of $q$-character
for the quantum affine superalgebra $U(\epsilon)$ and the generalized
quantum group $\mathcal{U}(\epsilon)$ are equivalent.
\end{rem}

From the classification of the finite-dimensional simple $U(\epsilon)$-modules,
the remark above and Corollary \ref{cor:q-char-cone-prop}, we know
that the simple $\mathcal{U}(\epsilon)$-modules in $\mathcal{F}(\epsilon)$
are exactly those $L(\Psi)$ such that $\Psi$ is a monomial in $Y_{i,a}$
and $\widetilde{Y}_{j,a}$ (\textit{i.e.} does not contain any negative
powers). Therefore, it is natural to consider the simple modules of
the form $L(Y_{i,a})$ and $L(\widetilde{Y}_{j,a})$ which belong
to $\mathcal{F}(\epsilon)$ again by Corollary \ref{cor:q-char-cone-prop}
and generate $\mathcal{F}(\epsilon)$ as a monoidal abelian category.
\begin{defn}
For $i\in\{1,\dots,M\}$ (resp. $i\in\{M,\dots,M+N-1\}$), the $i$-th
\textit{positive} (resp. \textit{negative}) \textit{fundamental representation}
with spectral parameter $a\in\mathbf{k}^{\times}$ refers to the simple
highest $\ell$-weight $\mathcal{U}(\epsilon)$-module $L(Y_{i,a})$
(resp. $L(\widetilde{Y}_{i,a})$).
\end{defn}

\begin{rem}
\label{rem:compare-fund-GQG-qasa}In \cite{Z3} a positive fundamental
representation $V_{s,a}^{+}$ ($1\leq s\leq M$) of the quantum affine
superalgebra $U(\epsilon)$ is defined by a fusion procedure from
the natural representations $V_{1,a}^{+}$ and a negative $V_{t,a}^{-}$
($1\leq t\leq N$) from the dual $V_{1,a}^{-}$ of $V_{1,a}^{+}$.
It turns out that $V_{s,a}^{+}$ (resp. $V_{t,a}^{-}$) corresponds
to $L(Y_{s,aq^{-2s-1}})$ (resp. $L(\widetilde{Y}_{M+N-t,-aq})$)
by inspecting its highest $\ell$-weight given at the end of \cite[Section 1]{Z3}.
The corresponding fusion procedure for $\mathcal{U}(\epsilon)$ is
given in Proposition \ref{prop:SES-fund-rep} below. 
\end{rem}

\begin{prop}
\label{prop:fund-gen-all-simple}Any simple $\mathcal{U}(\epsilon)$-module
in $\mathcal{F}(\epsilon)$ can be obtained as a subquotient of tensor
products of fundamental representations.
\end{prop}

\begin{proof}
As in the non-super case, it is a consequence of the following claim
which can be proved using the comultiplication formula for loop generators:
for $\mathcal{U}(\epsilon)$-modules $V_{i}$ with a highest $\ell$-weight
vector $v_{i}\in(V_{i})_{\Psi_{i}}$ ($i=1,2$), $v_{1}\otimes v_{2}\in V_{1}\otimes V_{2}$
is also a highest $\ell$-weight vector of $\ell$-weight $\Psi_{1}\Psi_{2}$. 
\end{proof}
To describe fundamental representations explicitly, let us recall
some basic facts on polynomial modules of the finite type subalgebra
of $\mathcal{U}(\epsilon)$ \cite[Section 3.1]{KL}, \textit{cf.}
\cite{BKK}. We denote by $\mathring{\mathcal{U}}(\epsilon)$ the
subalgebra of $\mathcal{U}(\epsilon)$ generated by $e_{i}$, $f_{i}$
and $k_{i}^{\pm1}$ for $i\in\mathring{I}$, which is an analogue
of $U_{q}(\mathfrak{sl}_{M|N})$. Let $V(\lambda)$ be the simple
highest weight $\mathring{\mathcal{U}}(\epsilon)$-module of highest
weight $\lambda=\sum\lambda_{i}\delta_{i}\in\mathring{P}$, which
is finite-dimensional if and only if $\lambda_{i}-\lambda_{i+1}\geq0$
for all $i\neq M$.

A partition $\lambda=(\lambda_{i})$ is called an $(M|N)$-hook partition
if $\lambda_{M+1}\leq N$. They parametrize simple polynomial $\mathring{\mathcal{U}}(\epsilon)$-modules,
by assigning to each $\lambda$ the simple module $V(\varpi_{\lambda})$
whose highest weight $\varpi_{\lambda}=\sum m_{i}\delta_{i}\in\mathring{P}$
is given by
\[
m_{i}=\begin{cases}
\lambda_{i} & \text{if }i\leq M\\
\mu_{i-M} & \text{if }i>M,
\end{cases}
\]
where $\mu$ is the conjugate of the partition $(\lambda_{M+1},\lambda_{M+2},\dots)$
obtained by removing the first $M$ rows of $\lambda$.

For $1\leq i\leq M$ and $a\in\mathbf{k}^{\times}$, the $\mathring{\mathcal{U}}(\epsilon)$-module
$V(1^{i})$ associated with the single-column partition $(1^{i})$
extends to a $\mathcal{U}(\epsilon)$-module $V(1^{i})_{a}$ explicitly
described as follows:
\begin{align*}
V(1^{i})= & \bigoplus_{\mathbf{m}\in\mathbb{Z}_{+}^{n}(\epsilon_{M|N}^{c})}\mathbf{k}\ket{\mathbf{m}},\quad\mathbb{Z}_{+}^{n}(\epsilon_{M|N}^{c})=\left\{ (m_{i})\in\mathbb{Z}_{\geq0}^{n}\,|\,m_{i}\in\{0,1\}\text{ if }1\le i\leq M\right\} ,\\
e_{r}\ket{\mathbf{m}} & =a^{\delta_{r,0}}[m_{r+1}]_{q_{r+1}}\ket{\mathbf{m}+\mathbf{e}_{r}-\mathbf{e}_{r+1}},\\
f_{r}\ket{\mathbf{m}} & =a^{-\delta_{r,0}}[m_{r}]_{q_{r}}\ket{\mathbf{m}-\mathbf{e}_{r}+\mathbf{e}_{r+1}}\\
k_{r}\ket{\mathbf{m}} & =q_{r}^{m_{r}}q_{r+1}^{-m_{r+1}}\ket{\mathbf{m}}.
\end{align*}
where we assume $\ket{\mathbf{m}}=0$ if $\mathbf{m}$ does not meet
the condition and the lower indices $r\in I$ should be read modulo
$n$. Similarly, formally the same formula defines a $\mathcal{U}(\epsilon)$-action
on
\[
V(i)_{a}=\bigoplus_{\mathbf{m}\in\mathbb{Z}_{+}^{n}(\epsilon_{M|N})}\mathbf{k}\ket{\mathbf{m}},\quad\mathbb{Z}_{+}^{n}(\epsilon_{M|N})=\left\{ (m_{i})\in\mathbb{Z}_{\geq0}^{n}\,|\,m_{i}\in\{0,1\}\text{ if }M<i\leq n\right\} 
\]
which extends the simple $\mathring{\mathcal{U}}(\epsilon)$-module
$V(i)$ associated with the single-row partition $(i)$. Let us also
put $V(-i)_{a}=V(i)_{a}^{*}$ and $V(-1^{i})_{a}\coloneqq V(1^{i})_{a}^{*}$,
where $V^{*}$ denotes the (left) dual of a finite-dimensional $\mathcal{U}(\epsilon)$-module
$V$.
\begin{rem}
\label{rem:compare-fund-KL}The $\mathcal{U}(\epsilon)$-module $V(i)_{a}$
is denoted by $\mathcal{W}_{i}(a)$ and also dubbed a fundamental
representation in \cite[Section 3.1]{KL}, as it corresponds to the
$i$-th fundamental representation of the quantum affine algebra $\mathcal{U}(\epsilon_{0|N})=U_{-q^{-1}}(\widehat{\mathfrak{sl}}_{N})$
under the super duality. On the other hand, $V(1^{i})_{a}$, which
really is the fundamental representation in this paper by the proposition
below and also corresponds to the one of $\mathcal{U}(\epsilon_{M|0})=U_{q}(\widehat{\mathfrak{sl}}_{M})$
under the duality, is the Kirillov--Reshetikhin module $\mathcal{W}^{1,i}(c)$
in \cite[Section 6.6]{KL}.
\end{rem}

The resulting modules $V(1^{i})_{a}$, $V(i)_{a}$, $V(-i)_{a}$ and
$V(-1^{i})_{a}$ are already simple over $\mathring{\mathcal{U}}(\epsilon)$
and we shall compute their highest $\ell$-weights. To this end we
recall the inductive formula for $e_{k\delta-\alpha_{i}}$ (\ref{eq:ind-formula-xir})
and Lemma \ref{lem:formula-x-_i,1} which allows us to compute directly
the action of 
\[
\psi_{i,r}^{+}=o(i)^{r}(q-q^{-1})k_{i}\widetilde{e}_{(r\delta,i)}=(-1)^{ir}(q-q^{-1})k_{i}(e_{r\delta-\alpha_{i}}e_{i}-q_{i}^{-1}q_{i+1}^{-1}e_{i}e_{r\delta-\alpha_{i}}),
\]
(Remark \ref{rem:compare-rootv-Drinfeld}) on their highest weight
vectors. Since it is straightforward, we only record the result and
leave the proof to the readers.
\begin{prop}
\label{prop:fund-rep-l-wt} As $\mathcal{U}(\epsilon)$-modules, we
have
\begin{align*}
V(1^{i})_{a} & =L(Y_{i,(-1)^{i+1+M}q^{N-M}a}),\\
V(-i)_{a} & =L(\widetilde{Y}_{n-i,(-1)^{-i+1+n}q^{2N-2M}a}),\\
V(i)_{a} & =L(Y_{1,(-1)^{M}q^{i-1+N-M}a}Y_{1,(-1)^{M}q^{i-3+N-M}a}\cdots Y_{1,(-1)^{M}q^{1-i+N-M}a}),\\
V(-1^{i})_{a} & =L(\widetilde{Y}_{n-1,(-1)^{i+1+n}q^{i-1+2N-2M}a}\widetilde{Y}_{n-1,(-1)^{i+1+n}q^{i-3+2N-2M}a}\cdots\widetilde{Y}_{n-1,(-1)^{i+1+n}q^{1-i+2N-2M}a}).
\end{align*}
\end{prop}

Comparing the definitions we obtain the following identification of
left dual modules from which one can also clarify the right dual of
fundamental representations (see Lemma \ref{lem:double-dual}).
\begin{cor}
Let $Q=q^{M}\widetilde{q}^{N}=(-1)^{N}q^{M-N}$. We have 
\begin{align*}
L(Y_{i,1})^{*} & \cong L(\widetilde{Y}_{n-1,q^{i-1}Q^{-1}}\widetilde{Y}_{n-1,q^{i-3}Q^{-1}}\dots\widetilde{Y}_{n-1,q^{1-i}Q^{-1}}),\\
L(\widetilde{Y}_{n-j,Q^{-1}}) & \cong L(Y_{1,\widetilde{q}^{j-1}Q^{-1}}Y_{1,\widetilde{q}^{j-3}Q^{-1}}\cdots Y_{1,\widetilde{q}^{1-j}Q^{-1}})^{*}.
\end{align*}
\end{cor}

Note that the dual of a fundamental representation is not necessarily
a fundamental representation, but rather a Kirillov--Reshetikhin-type
module. This can be understood in terms of the following \textit{fusion
construction} of fundamental representations (\textit{cf.} \cite[Section 1.3]{Z3}).
\begin{prop}
\label{prop:SES-fund-rep}There are exact sequences in $\mathcal{F}(\epsilon)$
\begin{align*}
0\longrightarrow L(Y_{i+1,1})\longrightarrow L(Y_{1,q^{i}}) & \otimes L(Y_{i,q^{-1}})\longrightarrow L(Y_{i,q^{-1}})\otimes L(Y_{1,q^{i}})\longrightarrow L(Y_{i+1,1})\longrightarrow0,\\
0\longrightarrow L(Y_{1,q^{i}}Y_{1,q^{i-2}}\cdots Y_{1,q^{-i}}) & \longrightarrow L(Y_{1,q^{i-2}}Y_{1,q^{i-4}}\cdots Y_{1,q^{-i}})\otimes L(Y_{1,q^{i}})\\
\longrightarrow L(Y_{1,q^{i}}) & \otimes L(Y_{1,q^{i-2}}Y_{1,q^{i-4}}\cdots Y_{1,q^{-i}})\longrightarrow L(Y_{1,q^{i}}Y_{1,q^{i-2}}\cdots Y_{1,q^{-i}})\longrightarrow0,\\
0\longrightarrow L(\widetilde{Y}_{n-i-1,1})\longrightarrow L(\widetilde{Y}_{n-i,\widetilde{q}}) & \otimes L(\widetilde{Y}_{n-1,\widetilde{q}^{-i}})\longrightarrow L(\widetilde{Y}_{n-1,\widetilde{q}^{-i}})\otimes L(\widetilde{Y}_{n-i,\widetilde{q}})\longrightarrow L(\widetilde{Y}_{n-i-1,1})\longrightarrow0.
\end{align*}
\end{prop}

\begin{proof}
The first two are analogues of \cite[Lemma B.1]{AK} and can be verified
by a straightforward calculation based on the explicit module structure
above (see also \cite[Appendix A.1]{KL}, noting that we are using
a different comultiplication). In fact, in each of the sequences the
middle maps are normalized $R$-matrices (see Section \ref{subsec:univ-coeff-R}
below) and the exact sequences can also be deduced from their spectral
decomposition \cite[Theorem 3.15]{KL}. The last one follows from
the second one by taking the left dual.
\end{proof}
\begin{cor}
The category $\mathcal{F}(\epsilon)$ is the smallest full subcategory
of $\mathcal{U}(\epsilon)$-modules which contains the natural representations
$L(Y_{1,a})$ for $a\in\mathbf{k}^{\times}$ and is closed under taking
subquotients, extensions, tensor products and duals.
\end{cor}

\subsection{A rank 2 case $\epsilon=(001)$\label{subsec:rank2-case}}

In this section we assume $\epsilon=(001)$. The purpose of this section
is to give examples on finite-dimensional simple $\mathcal{U}(001)$-modules
and their $q$-characters, which will be used later in describing
Frenkel--Mukhin-type algorithm by restriction to rank 1 and rank
2 subalgebras (Section \ref{subsec:restriction-rk12}).

We shall take the same approach as in Proposition \ref{prop:fund-rep-l-wt}.
Recall again from Lemma \ref{lem:formula-x-_i,1} that
\begin{align*}
e_{\delta-\alpha_{1}} & =\left\llbracket e_{0},e_{2}\right\rrbracket =e_{0}e_{2}+qe_{2}e_{0},\\
e_{\delta-\alpha_{2}} & =\left\llbracket e_{0},e_{1}\right\rrbracket =e_{0}e_{1}-q^{-1}e_{1}e_{0}.
\end{align*}
In this low rank case it is even possible to compute all the $\ell$-weights
(hence the $q$-character) of several simple modules, including fundamental
representations.
\begin{prop}
The $q$-characters of positive fundamental representations are given
by
\begin{align*}
\chi_{q}(L(Y_{1,1})) & =Y_{1,1}+Y_{1,q^{2}}^{-1}Y_{2,q}+\widetilde{Y}_{2,-q}^{-1}\\
 & =Y_{1,1}(1+A_{1,q}^{-1}+A_{1,q}^{-1}A_{2,q^{2}}^{-1}),\\
\chi_{q}(L(Y_{2,1})) & =Y_{2,1}+Y_{1,q}\widetilde{Y}_{2,-1}^{-1}+Y_{1,q^{3}}^{-1}Y_{2,q^{2}}\widetilde{Y}_{2,-1}^{-1}+\widetilde{Y}_{2,-q^{2}}^{-1}\widetilde{Y}_{2,-1}^{-1}\\
 & =Y_{2,1}(1+A_{2,q}^{-1}+A_{2,q}^{-1}A_{1,q^{2}}^{-1}+A_{2,q}^{-1}A_{1,q^{2}}^{-1}A_{2,q^{3}}^{-1}).
\end{align*}
\end{prop}

For $s\geq1$, the $\mathring{\mathcal{U}}(\epsilon)$-module $V(s,s)$
associated to a two-row partition $(s,s)$ has the following uniform
description: it is a 4-dimensional vector space spanned by $v_{12},\,v_{13},\,v_{23},\,v_{33}$
whose highest weight vector is $v_{12}$ with the weight $s\delta_{1}+s\delta_{2}$,
and the $\mathring{\mathcal{U}}(\epsilon)$-action is given by
\begin{align*}
 & f_{2}v_{12}=[s]v_{13},\quad f_{1}v_{13}=v_{23},\quad f_{2}v_{23}=v_{33},\\
 & e_{2}v_{13}=v_{12},\quad e_{1}v_{23}=v_{13},\quad e_{2}v_{33}=-[s+1]v_{23},
\end{align*}
and for any other combination of $e_{i}v_{kl}$ or $f_{i}v_{kl}$
vanishes. Pulling back by $\mathrm{ev}_{a}$ in Proposition \ref{prop:eval-hom}
for $a=-q^{1-s}$ we obtain a $\mathcal{U}(\epsilon)$-module whose
$e_{0},f_{0}$-action is given by
\begin{align*}
e_{0}v_{12} & =[s]v_{23},\quad f_{0}v_{23}=v_{12},\\
e_{0}v_{13} & =v_{33},\quad f_{0}v_{33}=-[s+1]v_{13}.
\end{align*}
We may compute its $q$-character in the same manner. It turns out
that its highest $\ell$-weight is $Y_{2,cq^{2-2s}}Y_{2,cq^{4-2s}}\cdots Y_{2,c}$
for some $c\in\mathbf{k}^{\times}$, and we obtain the following result.
\begin{prop}
For $s\geq1$, the $q$-character of $L(Y_{2,q^{2-2s}}Y_{2,q^{4-2s}}\cdots Y_{2,1})$
equals to
\begin{align*}
Y_{2,q^{2-2s}}\cdots Y_{2,q^{-2}}\cdot\chi_{q}(L(Y_{2,1})) & =Y_{2,q^{2-2s}}\cdots Y_{2,q^{-2}}(Y_{2,1}+Y_{1,q}\widetilde{Y}_{2,-1}^{-1}+Y_{1,q^{3}}^{-1}Y_{2,q^{2}}\widetilde{Y}_{2,-1}^{-1}+\widetilde{Y}_{2,-q^{2}}^{-1}\widetilde{Y}_{2,-1}^{-1}).\\
 & =Y_{2,q^{2-2s}}\cdots Y_{2,q^{-2}}Y_{2,1}(1+A_{2,q}^{-1}+A_{2,q}^{-1}A_{1,q^{2}}^{-1}+A_{2,q}^{-1}A_{1,q^{2}}^{-1}A_{2,q^{3}}^{-1}).
\end{align*}
\end{prop}

Recall that the negative fundamental representation $L(\widetilde{Y}_{2,1})$
can be obtained as the dual of $V(1)_{a}$ for some $a$. We may consider
more generally the dual of $V(1^{s})_{a}$ to obtain simple modules
with the highest $\ell$-weight of the form $\widetilde{Y}_{2,\widetilde{q}^{2-2s}}\widetilde{Y}_{2,\widetilde{q}^{4-2s}}\cdots\widetilde{Y}_{2,1}$,
whose $q$-character can be again computed from explicit descriptions
of the $\mathcal{U}(\epsilon)$-module structure.
\begin{prop}
The $q$-character of $L(\widetilde{Y}_{2,\widetilde{q}^{2-2s}}\widetilde{Y}_{2,\widetilde{q}^{4-2s}}\cdots\widetilde{Y}_{2,1})$
equals to
\begin{align*}
(s=1)\quad & \widetilde{Y}_{2,1}+Y_{1,-q^{-1}}Y_{2,-1}^{-1}+Y_{1,-q}^{-1}=\widetilde{Y}_{2,1}(1+A_{2,\widetilde{q}}^{-1}+A_{2,\widetilde{q}}^{-1}A_{1,-1}^{-1}),\\
(s>1)\quad & \widetilde{Y}_{2,\widetilde{q}^{2-2s}}\widetilde{Y}_{2,\widetilde{q}^{4-2s}}\cdots\widetilde{Y}_{2,\widetilde{q}^{-4}}\cdot\chi_{q}(\widetilde{Y}_{2,\widetilde{q}^{-2}}\widetilde{Y}_{2,1})\\
 & \quad=\widetilde{Y}_{2,\widetilde{q}^{2-2s}}\widetilde{Y}_{2,\widetilde{q}^{4-2s}}\cdots\widetilde{Y}_{2,1}(1+A_{2,\widetilde{q}}^{-1}+A_{2,\widetilde{q}}^{-1}A_{1,-1}^{-1}+A_{2,\widetilde{q}}^{-1}A_{1,-1}^{-1}A_{2,\widetilde{q}^{-1}}^{-1})
\end{align*}
\end{prop}

Observe that in the last two examples their normalized $q$-characters,
that is $q$-characters divided by their highest $\ell$-weights,
stabilize at $s=1$ and $s=2$ respecitvely. This super phenomenon
already manifests at the level of finite type, as the crystal \cite{BKK}
of the underlying $\mathring{\mathcal{U}}(\epsilon)$-module $V(M^{s})$
of $L(Y_{M,q^{2-2s}}Y_{M,q^{4-2s}}\cdots Y_{M,1})$ stabilizes from
$s\geq N$.
\begin{rem}
A simple module of the highest $\ell$-weight of the form $Y_{i,cq^{2-2s}}Y_{i,cq^{4-2s}}\cdots Y_{i,c}$
or $\widetilde{Y}_{j,c}\widetilde{Y}_{j,cq^{-2}}\cdots\widetilde{Y}_{j,cq^{2-2s}}$
is an analogue of Kirillov-Reshetikhin modules for quantum affine
algebras. A tableau sum formula is known for their $q$-characters
\cite[Theorem 2.4]{Z5}.
\end{rem}

\section{Transfer matrix construction of the $q$-character map and applications\label{sec:FR-FM-qchar}}

\subsection{Transfer matrix construction of the $q$-character map\label{subsec:qchar-constr-transferM}}

The multiplicative formula for the universal $R$-matrix allows us
to define the $q$-character map in a different way \cite{FR2}, motivated
from the construction of transfer matrices of quantum integrable systems.
From now on, we consider a slight extension of $\mathcal{U}(\epsilon)$
generated by $e_{i},f_{i}$ ($i\in I$) and $k_{\lambda}$ ($\lambda\in\mathring{P}$)
subject to 
\[
k_{0}=1,\quad k_{\lambda}k_{\mu}=k_{\lambda+\mu},\quad k_{\lambda}e_{i}k_{\lambda}^{-1}=\mathbf{q}(\lambda,\alpha_{i})e_{i},\quad k_{\lambda}f_{i}k_{\lambda}^{-1}=\mathbf{q}(\lambda,\alpha_{i})^{-1}f_{i},
\]
where the original generator $k_{i}$ is just $k_{\alpha_{i}}=k_{\delta_{i}}k_{\delta_{i+1}}^{-1}$.
As we are already concerning $\mathfrak{gl}$-weighted modules, this
does not affect our discussions in the last section.

We first recall the basic properties of the universal $R$-matrix,
which can be verified as in \cite[Section 4]{Lb}.
\begin{lem}[{cf. \cite[Proposition 4.2.2]{Lb}}]
\label{lem:R-mat-quasitri} The universal $R$-matrix $\mathcal{R}$
satisfies the following identities
\begin{align*}
\Delta^{\mathrm{op}}(x) & =\mathcal{R}\Delta(x)\mathcal{R}^{-1}\quad\quad{}^{\forall}x\in\mathcal{U}(\epsilon),\\
(\Delta\otimes\mathrm{id})\mathcal{R}= & \mathcal{R}^{13}\mathcal{R}^{\mathrm{23}},\quad(\mathrm{id}\otimes\Delta)\mathcal{R}=\mathcal{R}^{13}\mathcal{R}^{12}
\end{align*}
where for $A=\sum_{i}A_{i}^{\prime}\otimes A_{i}^{\prime\prime}\in\mathcal{U}(\epsilon)\otimes\mathcal{U}(\epsilon)$,
we use the notation $A^{12}=\sum_{i}A_{i}^{\prime}\otimes A_{i}^{\prime\prime}\otimes1\in\mathcal{U}(\epsilon)^{\otimes3}$
and similarly $A^{23}$, $A^{13}$.
\end{lem}

Given a $\mathcal{U}(\epsilon)$-module $V$ and a formal variable
$z$, we let $V_{z}=V\otimes\mathbf{k}[z^{\pm1}]$ with a $\mathcal{U}(\epsilon)$-module
structure
\[
e_{i}(v\otimes F)=(e_{i}v)\otimes z^{\delta_{i,0}}F,\quad f_{i}(v\otimes F)=(f_{i}v)\otimes z^{-\delta_{i,0}}F,\quad k_{\lambda}(v\otimes F)=(k_{\lambda}v)\otimes F\quad\text{for }v\in V,\,F\in\mathbf{k}[z^{\pm1}].
\]
Denote by $\pi_{V_{z}}:\mathcal{U}(\epsilon)\rightarrow\mathrm{End}(V)[z^{\pm1}]$
the associated algebra homomorphism. We introduce the corresponding
transfer matrix
\[
t_{V}(z)=\mathrm{tr}_{V}\left((\pi_{V_{z}}\otimes\mathrm{id})\mathcal{R}\right)\in\mathcal{U}^{-}(\epsilon)\left\llbracket z\right\rrbracket .
\]
Following the standard argument we obtain the usual properties for
transfer matrices (\textit{cf.} \cite[Lemma 2]{FR2}):
\begin{enumerate}
\item For any finite-dimensional $\mathcal{U}(\epsilon)$-modules $V$ and
$W$, we have $[t_{V}(z),t_{W}(w)]=0$.
\item If there is a short exact sequence $0\rightarrow V\rightarrow W\rightarrow V^{\prime}\rightarrow0$,
then $t_{W}(z)=t_{V}(z)+t_{V^{\prime}}(z)$.
\item $t_{V\otimes W}(z)=t_{V}(z)t_{W}(z)$ and $t_{V_{a}}(z)=t_{V}(az)$
for any $a\in\mathbf{k}^{\times}$, where $V_{a}$ is the specialization
of $V_{z}$ at $z=a$.
\end{enumerate}
Consider the subalgebra $\widetilde{\mathcal{U}}(\epsilon)$ generated
by $x_{i,r}^{\pm}$, $k_{i}^{\pm1}$ and $h_{i,r}$ for $i\in\mathring{I}$
and $r\leq0$, which contains $\mathcal{U}^{-}(\epsilon)$. The triangular
decomposition (\ref{eq:GQGDr-tri-decomp}) with respect to loop generators
restricts to $\widetilde{\mathcal{U}}(\epsilon)$ as 
\[
\widetilde{\mathcal{U}}(\epsilon)\cong\widetilde{\mathcal{U}}^{\mathrm{Dr},-}(\epsilon)\otimes\widetilde{\mathcal{U}}^{\mathrm{Dr},0}(\epsilon)\otimes\mathcal{\widetilde{U}}^{\mathrm{Dr},+}(\epsilon)
\]
where $\widetilde{\mathcal{U}}^{\mathrm{Dr},\pm}(\epsilon)$ (resp.
$\widetilde{\mathcal{U}}^{\mathrm{Dr},0}(\epsilon)$) is generated
by $x_{i,r}^{\pm}$ (resp. $k_{i}^{\pm1}$ and $h_{i,r}$) for $i\in\mathring{I}$
and $r\leq0$. Namely, we have
\[
\widetilde{\mathcal{U}}(\epsilon)\cong\widetilde{\mathcal{U}}^{\mathrm{Dr},0}(\epsilon)\oplus\left(\widetilde{\mathcal{U}}^{\mathrm{Dr},-}(\epsilon)_{0}\cdot\widetilde{\mathcal{U}}^{\mathrm{Dr},0}(\epsilon)\right)\oplus\left(\widetilde{\mathcal{U}}^{\mathrm{Dr},0}(\epsilon)\cdot\mathcal{\widetilde{U}}^{\mathrm{Dr},+}(\epsilon)_{0}\right)
\]
where $\widetilde{\mathcal{U}}^{\mathrm{Dr},\pm}(\epsilon)_{0}$ is
the augmentation ideal of $\widetilde{\mathcal{U}}^{\mathrm{Dr},\pm}(\epsilon)$,
and we let $\mathbf{h}_{q}$ be the projection onto $\widetilde{\mathcal{U}}^{\mathrm{Dr},0}(\epsilon)$
along the other two components.

We redefine the $q$-character map
\begin{align*}
\chi_{q}:K(\mathcal{F}(\epsilon)) & \longrightarrow\widetilde{\mathcal{U}}^{\mathrm{Dr},0}(\epsilon)\left\llbracket z\right\rrbracket \\
V & \longmapsto\mathbf{h}_{q}(t_{V}(z))
\end{align*}
which is a well-defined ring homomorphism due to the properties above
and the argument in \cite[Lemma 3]{FR2}. As in the non-super case
\cite[Section 3.3]{FR2}, we claim that this agrees with the definition
of $\chi_{q}(V)$ above as the generating function of the $\ell$-weights
of $V$.

Indeed, recall the multiplicative decomposition of $\mathcal{R}=\mathcal{R}^{+}\mathcal{R}^{0}\mathcal{R}^{-}\overline{\Pi}$.
Since any term with non-trivial factor from $\mathcal{R}^{-}$ is
removed under the projection $\mathbf{h}_{q}$ and similarly $\mathcal{R}^{+}$
by taking the trace, we obtain
\[
\mathcal{\chi}_{q}(V)=\mathrm{tr}_{V}\left[\exp\left(-\sum_{r>0}\sum_{i\in I}\frac{r(q-q^{-1})^{2}}{q_{i}^{r}-q_{i}^{-r}}\pi_{V}(h_{i,r})\otimes\widetilde{h}_{i,-r}z^{r}\right)(\pi_{V}\otimes\mathrm{id})\overline{\Pi}\right]
\]
where 
\begin{equation}
\widetilde{h}_{i,-r}=\sum\widetilde{C}_{ji}^{-r}h_{j,-r}.\quad(\text{note: }\widetilde{C}_{ji}^{r}=\widetilde{C}_{ji}^{-r})\label{eq:def-htilde_ir}
\end{equation}
By Theorem \ref{thm:integral-l-wt}, an $\ell$-weight of $V\in\mathcal{F}(\epsilon)$
can be written as
\[
\prod_{i<M}\left(\prod_{k,l}Y_{i,a_{i,k}}Y_{i,b_{i,l}}^{-1}\right)\cdot\left(\prod_{k,l,k^{\prime},l^{\prime}}Y_{M,a_{M,k}}Y_{M,b_{M,l}}^{-1}\widetilde{Y}_{M,c_{M,k^{\prime}}}\widetilde{Y}_{M,d_{M,l^{\prime}}}^{-1}\right)\cdot\prod_{j>M}\left(\prod_{k,l}\widetilde{Y}_{j,a_{j,k}}\widetilde{Y}_{i,b_{j,l}}^{-1}\right),
\]
which means that the eigenvalue of $\psi_{i}^{\pm}(z)$ is given by
\begin{align*}
 & \prod_{k}q_{i}\frac{1-q_{i}^{-1}a_{i,k}z}{1-q_{i}a_{i,k}z}\cdot\prod_{l}\left(q_{i}\frac{1-q_{i}^{-1}b_{i,l}z}{1-q_{i}b_{i,l}z}\right)^{-1} & \text{if }i\neq M,\\
 & \prod_{k}q\frac{1-q^{-1}a_{M,k}z}{1-qa_{M,k}z}\cdot\prod_{l}\left(q\frac{1-q^{-1}b_{M,l}z}{1-qb_{M,l}z}\right)^{-1}\cdot\prod_{k^{\prime}}\widetilde{q}\frac{1-\widetilde{q}^{-1}c_{M,k^{\prime}}z}{1-\widetilde{q}c_{M,k^{\prime}}z}\cdot\prod_{l^{\prime}}\left(\widetilde{q}\frac{1-\widetilde{q}^{-1}d_{M,l^{\prime}}z}{1-\widetilde{q}d_{M,l^{\prime}}z}\right)^{-1} & \text{if }i=M.
\end{align*}
Accordingly, the eigenvalue of $h_{i,r}$ is
\begin{equation}
\begin{cases}
\frac{q_{i}^{r}-q_{i}^{-r}}{r(q-q^{-1})}\left(\sum_{k}a_{i,k}^{r}-\sum_{l}b_{i,l}^{r}\right) & \text{if }i\neq M,\\
\frac{q^{r}-q^{-r}}{r(q-q^{-1})}\left(\sum_{k}a_{M,k}^{r}-\sum_{l}b_{M,l}^{r}\right)+\frac{\widetilde{q}^{r}-\widetilde{q}^{-r}}{r(q-q^{-1})}\left(\sum_{k^{\prime}}c_{M,k^{\prime}}^{r}-\sum_{l^{\prime}}d_{M,l^{\prime}}^{r}\right) & \text{if }i=M.
\end{cases}\label{eq:typ-eigenval-h}
\end{equation}
\[
\]

Now substituting the eigenvalue of $h_{i,r}$ into the formula for
$\chi_{q}(V)$ above, each factor $Y_{i,a}$ and $\widetilde{Y}_{i,a}$
correspond to the following series
\begin{align}
Y_{i}(az) & =k_{\delta_{1}+\cdots+\delta_{i}}^{-1}\exp\left(-(q-q^{-1})\sum_{r>0}\widetilde{h}_{i,-r}z^{r}a^{r}\right)\quad\text{for }i\leq M,\label{eq:def-Y-series}\\
\widetilde{Y}_{i}(az) & =\begin{cases}
k_{\delta_{i+1}+\cdots+\delta_{M+N}}\exp\left(-(q-q^{-1})\sum_{r>0}(-1)^{r-1}\widetilde{h}_{i,-r}z^{r}a^{r}\right) & \text{for }i=M,\\
k_{\delta_{i+1}+\cdots+\delta_{M+N}}\exp\left(-(q-q^{-1})\sum_{r>0}\widetilde{h}_{i,-r}z^{r}a^{r}\right) & \text{for }i>M,
\end{cases}\label{eq:def-tY-series}
\end{align}
Note also that $Y_{M}(az)\widetilde{Y}_{M}(-az)=k_{-\delta_{1}-\cdots-\delta_{M}+\delta_{M+1}+\cdots\delta_{M+N}}$
is independent of $a$. Therefore, we may identify the series $Y_{i}(az),\widetilde{Y}_{i}(az)$
with the variables $Y_{i,a},\widetilde{Y}_{i,a}$ in $\mathscr{Y}(\epsilon)$
to conclude that both definitions of $\chi_{q}(V)$ coincide.
\begin{rem}
\label{rem:def-A-Cartan}Similarly, if we let
\[
A_{i}(az)=k_{i}^{-1}\exp\left(-(q-q^{-1})\sum_{s>0}h_{i,-n}z^{n}a^{n}\right),
\]
then $A_{i}(az)$ can be written as a product of $Y_{j}(az)^{\pm1}$
and $\widetilde{Y}_{j}(az)^{\pm1}$ as determined by the matrix $C(q,-q^{-1})$
and this recovers the definition of $A_{i,a}$ in terms of $Y_{i,a}$
(\ref{eq:def-A-qCartan}).
\end{rem}

\subsection{\label{subsec:univ-coeff-R}Universal coefficients of $R$-matrices
and a denominator formula}

Given $V,W\in\mathcal{F}(\epsilon)$, the universal $R$-matrix $\mathcal{R}$
defines a $\mathcal{U}(\epsilon)$-module homomorphism (see \cite[Section 3.2]{KL})
\[
\mathcal{R}_{V,W}(z):V_{z}\otimes W\longrightarrow\mathbf{k}\left\llbracket z\right\rrbracket \otimes(W\otimes V_{z}).
\]
Assume $V,W$ are simple and let $v,w$ be the highest $\ell$-weight
vector of $V,W$ respectively. Then there exists $a_{V,W}(z)\in\mathbf{k}\left\llbracket z\right\rrbracket $
such that $\mathcal{R}_{V,W}(z)(v\otimes w)=a_{V,W}(z)(w\otimes v)$,
which allows us to define a \textit{normalized $R$-matrix}
\[
\mathcal{R}_{V,W}^{\mathrm{norm}}(z)=a_{V,W}(z)^{-1}\mathcal{R}_{V,W}(z):V_{z}\otimes W\longrightarrow\mathbf{k}(z)\otimes(W\otimes V_{z})
\]
which is the unique $\mathbf{k}(z)\otimes\mathcal{U}(\epsilon)$-module
homomorphism mapping $v\otimes w$ to $w\otimes v$. The formal power
series $a_{V,W}(z)$ is called the \textit{universal coefficient}
of $V$ and $W$.

As an application of the transfer matrix construction of $\chi_{q}$,
we obtain a general method to compute the universal coefficient of
two simple $\mathcal{U}(\epsilon)$-modules in $\mathcal{F}(\epsilon)$
\cite[Section 4.3]{FR2}. Let $\Psi,\Psi^{\prime}$ be the highest
$\ell$-weights of simple $V,W\in\mathcal{F}(\epsilon)$ respectively.
From the computation below, it is clear that $a_{V,W}(z)$ is bimultiplicative
with respect to $\Psi$ and $\Psi^{\prime}$, namely 
\begin{align*}
a_{L(\Psi_{1}\Psi_{2}),L(\Psi^{\prime})}(z) & =a_{L(\Psi_{1}),L(\Psi^{\prime})}(z)a_{L(\Psi_{2}),L(\Psi^{\prime})}(z),\\
a_{L(\Psi),L(\Psi_{1}^{\prime}\Psi_{2}^{\prime})}(z) & =a_{L(\Psi),L(\Psi_{1}^{\prime})}(z)a_{L(\Psi),L(\Psi_{2}^{\prime})}(z).
\end{align*}
Hence, from the highest $\ell$-weight classification of the simples
in $\mathcal{F}(\epsilon)$, it is enough to compute it when $\Psi$
and $\Psi^{\prime}$ are either $Y_{i,a}$ or $\widetilde{Y}_{j,b}$.
We also write for $f,g\in\mathbf{k}\left\llbracket z\right\rrbracket $
\[
f\equiv g\quad\text{if }f=cg\text{ for some }c\in\mathbf{k}[z^{\pm1}]^{\times}=\{kz^{t}\,|\,k\in\mathbf{k}^{\times},t\in\mathbb{Z}\}
\]
and always work with this equivalence relation, as for the future
application we only need $a_{V,W}(z)$ up to a multiplication by an
element of $\mathbf{k}[z^{\pm1}]^{\times}$.

Let us recall the argument of \cite[Section 4.2,3]{FR2}. First, since
$w$ is a highest $\ell$-weight vector, it is an eigenvector of $t_{V}(z)$:
\[
t_{V}(z)\cdot w=\mathbf{h}_{q}(t_{V}(z))\cdot w=\chi_{q}(V)\cdot w.
\]
Here we regard $\chi_{q}(V)$ as an element of $\widetilde{\mathcal{U}}^{\mathrm{Dr},0}(\epsilon)\left\llbracket z\right\rrbracket $,
or more precisely a Laurent polynomial in the formal power series
$Y_{i}(az)$ and $\widetilde{Y}_{j}(bz)$ defined in the last section.
Next, take bases of generalized eigenvectors of $\psi_{i}^{\pm}(z)$
for $V$ and $W$ containing $v$ and $w$. From the definition of
$t_{V}(z)$, the eigenvalue of $t_{V}(z)$ on $w$ equals the sum
of diagonal entries of $\mathcal{R}_{V,W}(z)$. Such an entry corresponds
to a vector $v^{\prime}\otimes w$ for a basis vector $v^{\prime}\in V$
and so to a monomial occurring in $\chi_{q}(V)$ (counted with multiplicity).
In particular, we find the universal coefficient $a_{V,W}(z)$ which
is just the diagonal entry corresponding to $v\otimes w$ by evaluating
on $w$ the series $m(z)\in\widetilde{\mathcal{U}}^{\mathrm{Dr},0}(\epsilon)\left\llbracket z\right\rrbracket $
obtained by replacing each $Y_{i,a}$, $\widetilde{Y}_{j,b}$ in the
highest $\ell$-weight $\Psi$ of $V$ with $Y_{i}(az)$, $\widetilde{Y}_{j}(bz)$
respectively. 

The action of $m(z)$ on $w$ can be explicitly computed thanks to
the formula for $\widetilde{C}^{r}$. Suppose $\Psi=Y_{i,a}$ so that
$m(z)=Y_{i}(az)$ (\ref{eq:def-Y-series}) and $\Psi^{\prime}=Y_{j,b}$,
$j\leq M$ (resp.$\Psi^{\prime}=\widetilde{Y}_{j,b}$, $j\geq M$).
The latter means that $h_{k,-r}$ acts on $w$ by the scalar
\[
\delta_{k,j}\frac{q^{r}-q^{-r}}{r(q-q^{-1})}b^{-r}\quad(\text{resp. }\delta_{k,j}\frac{\widetilde{q}^{r}-\widetilde{q}^{-r}}{r(q-q^{-1})}b^{-r})
\]
and then 
\[
\widetilde{h}_{i,-r}\cdot w=\sum_{k}\widetilde{C}_{k,i}^{-r}h_{k,-r}\cdot w=\frac{q^{r}-q^{-r}}{r(q-q^{-1})}b^{-r}\widetilde{C}_{ji}^{-r}\quad(\text{resp. }\frac{\widetilde{q}^{r}-\widetilde{q}^{-r}}{r(q-q^{-1})}b^{-r}\widetilde{C}_{ji}^{-r}).
\]
From the formulas in Proposition \ref{prop:inv-qCartan}, $(q^{r}-q^{-r})\widetilde{C}_{ji}^{-r}$
(resp. $(\widetilde{q}^{r}-\widetilde{q}^{-r})\widetilde{C}_{ji}^{-r}$)
is of the form $f_{ji}^{r}/(Q^{r}-Q^{-r})$, where 
\[
f_{ji}^{r}=\sum_{k}\left(m_{ji}^{+}(k)(q^{k})^{r}+m_{ji}^{-}(k)(-q^{k})^{r}\right)
\]
for some $m_{ij}^{\pm}(k)\in\mathbb{Z}$ that are independent of $r$
(in fact, for each $i,j$, either $m_{ji}^{+}(k)=0$ for all $k$
or $m_{ji}^{-}(k)=0$ for all $k$) and $Q=q^{M}\widetilde{q}^{N}$.
It remains to substitute everything to (\ref{eq:def-Y-series}) to
conclude that each monomial $m_{ji}^{\pm}(k)(\pm q^{k})^{r}$ contributes
to the eigenvalue the factor
\[
(ab^{-1}(\pm q^{k})Q^{-1};Q^{-2})^{m_{ji}(k)},\quad\text{where }(x;y)_{\infty}=\prod_{m=0}^{\infty}(1-xy^{m}).
\]
The other case $\Psi=\widetilde{Y}_{i,a}$ ($i\geq M$) can be done
similarly, except that one has to be careful when $i=M$ as $\widetilde{Y}_{M}(z)$
is slight different from $\widetilde{Y}_{i}(z)$ for $i>M$ (\ref{eq:def-tY-series}).

For the universal coefficient for fundamental representations, let
us simplify the notation $a_{ij}(z)=a_{L(Y_{i,1}),L(Y_{j,1})}(z)$
($1\leq i,j\leq M$), $a_{i\widetilde{j}}(z)=a_{L(Y_{i,1}),L(\widetilde{Y}_{j,1})}(z)$
($1\leq i\leq M$, $M\leq j\leq n-1$), and so on. Then we obtain
the following formula by Proposition \ref{prop:inv-qCartan}.
\begin{prop}
\label{prop:univ-coeff-fund}The universal coefficients for fundamental
representations are given as follows:

\begin{align*}
a_{ij}(z) & \equiv\frac{(q^{\left|i-j\right|}Q^{-2}z;Q^{-2})(q^{-\left|i-j\right|}z;Q^{-2})}{(q^{i+j}Q^{-2}z;Q^{-2})(q^{-i-j}z;Q^{-2})}, & (1\leq i,j\leq M)\\
a_{i,\widetilde{j}}(z)=a_{\widetilde{j},i}(z) & \equiv\frac{((-1)^{M-j+1}q^{i-j+2N}z;Q^{-2})((-1)^{M-j+1}q^{-i+j-2M}z;Q^{-2})}{((-1)^{M-j+1}q^{i+j-2M}z;Q^{-2})((-1)^{M-j+1}q^{-i-j+2N}z;Q^{-2})}, & (i\leq M\leq j)\\
a_{\widetilde{i},\widetilde{j}}(z) & \equiv\frac{((-1)^{i+j}q^{\left|i-j\right|}z;Q^{-2})((-1)^{i+j}q^{-\left|i-j\right|}Q^{-2}z;Q^{-2})}{((-1)^{i+j}q^{i+j-4M}z;Q^{-2})((-1)^{i+j}q^{-i-j+4M}Q^{-2}z;Q^{-2})}. & (M\leq i,j\leq n-1)
\end{align*}
\end{prop}

On the other hand, there is another way to read the universal coefficient
from the denominator of normalized $R$-matrices \cite[Appendix A]{AK}.
Recall that the \textit{denominator} $d_{V,W}(z)$ of the normalized
$R$-matrix $\mathcal{R}_{V,W}^{\mathrm{norm}}(z)$ for simple $V,W\in\mathcal{F}(\epsilon)$
is the monic polynomial of minimal degree such that
\[
\mathrm{im}\left(d_{V,W}(z)\mathcal{R}_{V,W}^{\mathrm{norm}}(z)\right)\subset W\otimes V_{z}.
\]
We will use a similar notation $d_{i,j}(z)$, $d_{i,\widetilde{j}}(z)$,
... for the denominators for fundamental representations. The argument
relies on the following simple observations.
\begin{lem}
For $i,j\in\{1,\dots,M\}\sqcup\{\widetilde{M},\dots,\widetilde{n-1}\}$,
we have $d_{ji}(z)\equiv\overline{d_{ij}(z^{-1})}$ where $\overline{\cdot}$
is a ring automorphism of $\mathbf{k}[z^{\pm1}]$ defined by $\overline{q}=q^{-1},\,\overline{z}=z$.
\end{lem}

\begin{proof}
The algebra $\mathcal{U}(\epsilon)$ has a ring automorphism $a\mapsto\overline{a}$,
where
\[
\overline{e_{i}}=e_{i},\overline{f_{i}}=f_{i},\,\overline{k_{i}}=k_{i}^{-1},\,\overline{q}=q^{-1}.
\]
If we denote by $\overline{M}$ the pullback of a $\mathcal{U}(\epsilon)$-module
through the automorphism, we have
\[
\overline{M\otimes N}=\overline{N}\otimes\overline{M}\text{ and }\overline{L(Y_{i,1})}=L(Y_{i,1}),\,\overline{L(\widetilde{Y}_{j,1})}=L(\widetilde{Y}_{j,1})
\]
from which the assertion follows.
\end{proof}
\begin{lem}[{cf. \cite[Proposition A.1]{AK}}]
 For simple $V,W\in\mathcal{F}(\epsilon)$, we have
\[
a_{V,W}(z)a_{V,W^{*}}(z)\equiv\frac{d_{V,W}(z)}{d_{W^{*},V}(z^{-1})}.
\]
\end{lem}

\begin{proof}
The same proof as in \cite[Proposition A.1]{AK} applies here (except
we are using the left dual) which relies on the crossing symmetry
of the universal $R$-matrix. The latter formally follows from the
quasitriangularity (Lemma \ref{lem:R-mat-quasitri}), see for example
\cite[Proposition 6.3.2, 9.5.2]{EFK}. 
\end{proof}
\begin{lem}
\label{lem:double-dual}For $V\in\mathcal{F}(\epsilon)$, we have
$V^{**}\cong V_{Q^{-2}}$.
\end{lem}

\begin{proof}
For $x\in\mathcal{U}(\epsilon)_{\beta}$ with $\beta=\sum_{i\in I}c_{i}\alpha_{i}\in Q$,
we have
\[
S^{2}(x)=\prod_{i\in I}(q_{i}q_{i+1})^{-c_{i}}x
\]
which implies $S^{2}(\psi_{i,\pm r}^{\pm})=Q^{-2r}\psi_{i,\pm r}^{\pm}$
for $r>0$ and hence the statement follows.
\end{proof}
Applying the second lemma to $V,W$ and then to $V,W^{*}$, we obtain
a $q$-difference equation
\begin{equation}
\frac{a_{V,W}(z)}{a_{V,W}(Q^{2}z)}=\frac{a_{V,W}(z)}{a_{V,W^{**}}(z)}\equiv\frac{d_{V,W}(z)d_{W^{**},V}(z^{-1})}{d_{W^{*},V}(z^{-1})d_{V,W^{*}}(z)}.\label{eq:q-DE-univ-coeff}
\end{equation}
Therefore, if we know the denominators $d_{V,W}(z)$ and $d_{V,W^{*}}(z)$,
then we find the universal coefficient $a_{V,W}(z)$ by solving the
equation.

While the denominator is much harder to compute than the universal
coefficient, by comparing the methods above for $a_{V,W}(z)$ one
may deduce information on the denominator $d_{V,W}(z)$. Here we give
an exact formula for $d_{i,\widetilde{j}}(z)$ for $i\leq M\leq j$
(for the quantum affine superalgebra $U(\epsilon)$ it is also computed
directly from the explicit structure of fundamental representations
\cite[Section 3]{Z3}). First, we recall the denominator formula for
polynomial representations which can be lifted from the non-super
case \cite[Theorem 6.5]{OS} through the super duality functor, see
\cite[Section 4.2]{KL}. Note that in our convention (including \cite{KL})
$z$ is equal to $z^{-1}$ in \cite{AK,OS}.
\begin{prop}
\label{prop:denom-R-++--}For $M=L(Y_{i,1})$ and $N=L(Y_{1,q^{j-1}}Y_{1,q^{j-3}}\cdots Y_{1,q^{1-j}})$,
we have
\[
d_{M,N}(z)=d_{N,M}(z)=z-q^{-i-j}.
\]
Dually, we also have $d_{\widetilde{n-1},\widetilde{n-j}}(z)=z-\widetilde{q}^{-j-1}$.
\end{prop}

First we put $V=L(Y_{i,1})$ and $W=L(\widetilde{Y}_{j,1})$ into
(\ref{eq:q-DE-univ-coeff}), where $W^{*}=L(Y_{1,\widetilde{q}^{n-j-1}Q^{-1}}Y_{1,\widetilde{q}^{n-j-3}Q^{-1}}\cdots Y_{1,\widetilde{q}^{1+j-n}Q^{-1}})$.
If we write $d_{i,\widetilde{j}}(z)=\prod_{\nu}(z-p_{\nu})$, together
with the proposition, the $q$-difference equation reads
\begin{align*}
\frac{a_{i,\widetilde{j}}(z)}{a_{i,\widetilde{j}}(Q^{2}z)} & \equiv\frac{\prod_{\nu}(z-p_{\nu})(z^{-1}-Q^{2}\overline{p}_{\nu}^{-1})}{(z^{-1}-(-1)^{M-j-1}q^{-i+j-2N})(z-(-1)^{M-j-1}q^{-i+j-2M})}\\
 & \equiv\frac{\prod_{\nu}(1-p_{\nu}^{-1}z)(1-\overline{p}_{\nu}^{-1}Q^{2}z)}{(1-(-1)^{M-j-1}q^{-i+j-2N}z)(1-(-1)^{M-j-1}q^{i-j+2M}z)},
\end{align*}
which is (uniquely up to equivalence $\equiv$) solved by
\[
a_{i,\widetilde{j}}(z)=\frac{((-1)^{M-j-1}q^{-i+j-2N}Q^{-2}z;Q^{-2})((-1)^{M-j-1}q^{i-j+2M}Q^{-2}z;Q^{-2})}{\prod_{\nu}(p_{\nu}^{-1}Q^{-2}z;Q^{-2})(\overline{p}_{\nu}^{-1}z;Q^{-2})}.
\]
Compared with the second formula in Proposition \ref{prop:univ-coeff-fund},
as all the factors are of the form $(az;Q^{-2})$, we obtain $d_{i,\widetilde{j}}(z)=z-p_{ij}$
for some $p_{ij}\in\mathbf{k}$ and 
\[
\{(-1)^{j-M+1}q^{i+j-2M},(-1)^{j-M+1}q^{-i-j+2N}\}=\{p_{ij}^{-1}Q^{-2},\overline{p}_{ij}^{-1}\}.
\]
At this stage we cannot claim further, that is both 
\begin{equation}
d_{i,\widetilde{j}}(z)=z-(-1)^{j-M+1}q^{i+j-2M}\quad\text{and}\quad d_{i,\widetilde{j}}(z)=z-(-1)^{j-M+1}q^{-i-j+2N}\label{eq:denom-dichotomy}
\end{equation}
are plausible. Note that when $i=1$ and $j=n-1$, 
\[
(-1)^{j-M+1}q^{i+j-2M}=Q^{-1}=(-1)^{j-M+1}q^{-i-j+2N}
\]
so there is no ambiguity and we may conclude $d_{1,\widetilde{n-1}}(z)=z-Q^{-1}$. 

To determine the value of $p_{ij}$ for general $i$ and $j$, we
invoke fusion construction to proceed inductively from the base case
$i=1$, $j=n-1$. We shall use the following fact.
\begin{lem}[{\cite[Lemma C.15]{AK}}]
 Suppose we have a surjective map $V_{1}\otimes V_{2}\twoheadrightarrow V$
for simple $\mathcal{U}(\epsilon)$-modules $V,V_{1}$ and $V_{2}$.
Then for any simple $W$, we have
\[
\frac{d_{V_{1},W}(z)d_{V_{2},W}(z)a_{V,W}(z)}{a_{V_{1},W}(z)a_{V_{2},W}(z)d_{V,W}(z)},\,\frac{d_{W,V_{1}}(z)d_{W,V_{2}}(z)a_{W,V}(z)}{a_{W,V_{1}}(z)a_{W,V_{2}}(z)d_{W,V}(z)}\in\mathbf{k}[z^{\pm1}].
\]
\end{lem}

Consider the following surjective map
\[
L(\widetilde{Y}_{n-1,Q^{-1}q^{i+1}})\otimes L(Y_{i+1,q})\overset{1\otimes\iota}{\longrightarrow}L(\widetilde{Y}_{n-1,Q^{-1}q^{i+1}})\otimes L(Y_{1,q^{i+1}})\otimes L(Y_{i,1})\overset{\mathrm{ev}\otimes1}{\longrightarrow}L(Y_{i,1})
\]
where $\iota$ is the first map of the first exact sequence in Proposition
\ref{prop:SES-fund-rep} and $\mathrm{ev}$ is the evaluation for
$L(\widetilde{Y}_{n-1,Q^{-1}q^{i+1}})\cong L(Y_{1,q^{i+1}})^{*}$.
The lemma then asserts
\[
\frac{d_{\widetilde{n-1},\widetilde{j}}(Q^{-1}q^{i+1}z)d_{i+1,\widetilde{j}}(qz)a_{i,\widetilde{j}}(z)}{a_{\widetilde{n-1},\widetilde{j}}(Q^{-1}q^{i+1}z)a_{i+1,\widetilde{j}}(qz)d_{i,\widetilde{j}}(z)}\in\mathbf{k}[z^{\pm1}].
\]
Substituting the formulas for $d_{\widetilde{n-1},\widetilde{j}}(z)$,
$a_{i,\widetilde{j}}(z)$, $a_{i+1,\widetilde{j}}(z)$ and $a_{\widetilde{n-1},\widetilde{j}}(z)$,
this is equivalent to
\begin{equation}
\frac{(z-(-1)^{M-j+1}q^{-i-j-2+2M})\cdot d_{i+1,\widetilde{j}}(qz)}{d_{i,\widetilde{j}}(z)}\in\mathbf{k}[z^{\pm1}].\label{eq:denom-ind-hor}
\end{equation}
Similarly, from the composition
\[
L(\widetilde{Y}_{j-1,1})\otimes L(Y_{1,\widetilde{q}^{j-n}Q})\longrightarrow L(\widetilde{Y}_{j,\widetilde{q}})\otimes L(\widetilde{Y}_{n-1,\widetilde{q}^{j-n}})\otimes L(Y_{1,\widetilde{q}^{j-n}Q})\longrightarrow L(\widetilde{Y}_{j,\widetilde{q}})
\]
we deduce that
\begin{equation}
\frac{(z-(-1)^{M-j+1}q^{-i-j+2+2M})\cdot d_{i,\widetilde{j-1}}(\widetilde{q}z)}{d_{i,\widetilde{j}}(z)}\in\mathbf{k}[z^{\pm1}].\label{eq:denom-ind-vert}
\end{equation}
Note that for these to be true, $d_{i,\widetilde{j}}(z)$ must be
equal to one of the factors of each numerators, as $d_{i,\widetilde{j}}(z)$
is of degree 1.
\begin{thm}
\label{thm:denom-R-+-}For $i\leq M\leq j$, the denominator $d_{i,\widetilde{j}}(z)$
of the normalized $R$-matrix
\[
\mathcal{R}_{i,\widetilde{j}}^{\mathrm{norm}}(z):L(Y_{i,z})\otimes L(\widetilde{Y}_{j,1})\longrightarrow L(\widetilde{Y}_{j,1})\otimes L(Y_{i,z})
\]
is given by
\[
d_{i,\widetilde{j}}(z)=z-(-1)^{j-M+1}q^{i+j-2M}.
\]
\end{thm}

\begin{proof}
Recall that $d_{1,\widetilde{n-1}}(z)=z-(-1)^{N}q^{N-M}$. We first
observe by comparing the dichotomy of $d_{i,\widetilde{j}}(z)$ (\ref{eq:denom-dichotomy})
and the constraints (\ref{eq:denom-ind-hor}, \ref{eq:denom-ind-vert})
that 
\begin{align*}
i+j\neq2M-1\text{ nor }N\neq M-1 & \,\Longrightarrow\,d_{i,\widetilde{j}}(z)\neq z-(-1)^{M-j+1}q^{-i-j-2+2M}\,\Longrightarrow\,d_{i+1,\widetilde{j}}(z)\equiv d_{i,\widetilde{j}}(q^{-1}z),\\
i+j\neq2M+1\text{ nor }N\neq M+1 & \,\Longrightarrow\,d_{i,\widetilde{j}}(z)\neq z-(-1)^{M-j+1}q^{-i-j+2+2M}\,\Longrightarrow\,d_{i,\widetilde{j-1}}(z)\equiv d_{i,\widetilde{j}}(\widetilde{q}^{-1}z).
\end{align*}
Consider the graph defined as follows:
\begin{align*}
\{\text{vertices}\} & =\{(i,j)\in\mathbb{Z}^{2}\,|\,1\leq i\leq M\leq j\leq n-1\}\\
\{\text{edges}\} & =\{(i,j)-(i+1,j)\,|\,i+j\neq2M-1\}\cup\{(i,j)-(i,j-1)\,|\,i+j\neq2M+1\}.
\end{align*}
Suppose $N>M+1$ (so $N\neq M-1$ as well). From the observation,
if $(1+a,n-1-b)$ is connected to $(1,n-1)$, then we have $d_{1+a,\widetilde{n-1-b}}(z)=d_{1,n-1}(q^{-a}\widetilde{q}^{-b}z)$
as we desired. Hence it is enough to see that the graph is connected,
which is obvious.

In the case of $N=M+1$, it may happen that
\[
d_{i,\widetilde{j}}(z)=z-(-1)^{j-M+1}q^{-i-j+2N}=z-(-1)^{M-j+1}q^{-i-j+2+2M}.
\]
On the other hand, we have 
\[
d_{1+a,\widetilde{n-1}}(z)\equiv d_{1,\widetilde{n-1}}(q^{-a}z)\equiv z-(-1)^{N}q^{a+N-M}
\]
which falls into the first one of the dichotomy 
\[
d_{i,\widetilde{j}}(z)=z-(-1)^{j-M+1}q^{i+j-2M}\quad\text{or}\quad z-(-1)^{j-M+1}q^{-i-j+2N}.
\]
Note that these two are distinct if and only if $i+j\neq2M+1$, and
so for $j=n-1$ and $i\neq1$ we still have $d_{i,\widetilde{j-1}}(z)\equiv d_{i,\widetilde{j}}(\widetilde{q}^{-1}z)$.
Repeating this argument, we may consider the same graph to deduce
the same conclusion. Now the remaining case $N\leq M-1$ is parallel.
\end{proof}

\subsection{\label{subsec:restriction-rk12}Restriction to rank 1 and 2}

Let $J$ be a subset of $\mathring{I}$ of the form $\{p,p+1,\dots,p^{\prime}-1\}$
for some $p<p^{\prime}$, $\epsilon_{J}=(\epsilon_{p},\epsilon_{p+1},\dots,\epsilon_{p^{\prime}})$
the subsequence of $\epsilon$ and $\overline{J}=\mathring{I}\setminus J$.
Consider the subalgebra $\mathcal{U}(\epsilon)_{J}$ of $\mathcal{U}(\epsilon)$
generated by $x_{j,k}^{\pm}$, $h_{j,r}$ and $k_{\delta_{l}}^{\pm1}$
($j\in J$, $k\in\mathbb{Z}$, $r\in\mathbb{Z}\setminus\{0\}$, $l=p,\dots,p^{\prime}$).
We set
\[
\mathscr{Y}(\epsilon_{J})=\mathbb{Z}[Y_{i,a}^{\pm1},\widetilde{Y}_{j,a}^{\pm1}]_{p\leq i\leq M\leq j\leq p^{\prime}-1,a\in\mathbf{k}^{\times}}/(Y_{M,a}\widetilde{Y}_{M,-a}=Y_{M,b}\widetilde{Y}_{M,-b})_{a,b\in\mathbf{k}^{\times}}
\]
which is defined formally in the same way even if the numbers of $0$
and $1$ in $\epsilon_{J}$ are the same. Here we understand there
is no $Y_{i,a}$-variable if $p>M$ and hence no $D=Y_{M,a}\widetilde{Y}_{M,-a}$.
According to the construction in the previous section, given a $\mathcal{U}(\epsilon)$-module
$V$, the $q$-character of the restriction of $V$ to $\mathcal{U}(\epsilon)_{J}$
can be obtained by applying to $\chi_{q}(V)$ the specialization 
\begin{align*}
\beta_{J}:\mathscr{Y}(\epsilon) & \longrightarrow\mathscr{Y}(\epsilon_{J}),\\
Y_{i,a},\widetilde{Y}_{i,a} & \longmapsto\begin{cases}
Y_{i,a},\widetilde{Y}_{i,a} & \text{if }i\in J\\
1 & \text{if }i\in\overline{J}.
\end{cases}
\end{align*}

As in \cite[Section 3]{FM}, we can refine this procedure as follows.
Let $\mathcal{U}^{0}(\epsilon)_{J}^{\perp}$ be the subalgebra generated
by $k_{\delta_{l}}$ and $\widetilde{h}_{i,r}$ for $l\in\mathbb{I}\setminus\{p,\dots,p^{\prime}\}$,
$i\in\overline{J}$ and $r\in\mathbb{Z}\setminus\{0\}$, where $\widetilde{h}_{i,r}$
is defined in (\ref{eq:def-htilde_ir}). From the commutation relations
(\ref{eq:GQG-Drinfeld-hh}) and (\ref{eq:GQG-Drinfeld-hx}), we know
that $\mathcal{U}^{0}(\epsilon)_{J}^{\perp}$ commutes with $\mathcal{U}(\epsilon)_{J}$
inside $\mathcal{U}(\epsilon)$. The idea is to consider simultaneously
the eigenvalues of $\widetilde{h}_{i,r}$ for $i\in\overline{J}$
so that no information is lost in the course of restriction.

To this end, we first extend $\mathscr{Y}(\epsilon_{J})$ to 
\begin{align*}
\mathscr{Y}(\epsilon)^{(J)} & =\begin{cases}
\mathscr{Y}(\epsilon_{J})\otimes\mathbb{Z}[Z_{j,b}^{\pm1}]_{j\in\overline{J},b\in\mathbf{k}^{\times}} & \text{if }M\in J,\\
\mathscr{Y}(\epsilon_{J})\otimes\mathbb{Z}[Z_{j,b}^{\pm1},D^{\pm1}]_{j\in\overline{J},b\in\mathbf{k}^{\times}} & \text{if }M\in\overline{J}.
\end{cases}
\end{align*}
Then we define a ring homomorphism $\tau_{J}:\mathscr{Y}(\epsilon)\longrightarrow\mathscr{Y}(\epsilon)^{(J)}$
by
\[
\tau_{J}(Y_{i,a})=\begin{cases}
Y_{i,a}\prod_{j\in\overline{J}}\prod_{k\in\mathbb{Z}}Z_{j,aq^{k}}^{p_{ij}^{+}(k)}Z_{j,-aq^{k}}^{p_{ij}^{-}(k)}\\
\prod_{j\in\overline{J}}\prod_{k\in\mathbb{Z}}Z_{j,aq^{k}}^{p_{ij}^{+}(k)}Z_{j,-aq^{k}}^{p_{ij}^{-}(k)},
\end{cases}\quad\tau_{J}(\widetilde{Y}_{i,a})=\begin{cases}
\widetilde{Y}_{i,a}\prod_{j\in\overline{J}}\prod_{k\in\mathbb{Z}}Z_{j,aq^{k}}^{-p_{ij}^{-}(k)}Z_{j,-aq^{k}}^{-p_{ij}^{+}(k)} & i\in J\\
D\prod_{j\in\overline{J}}\prod_{k\in\mathbb{Z}}Z_{j,aq^{k}}^{-p_{ij}^{-}(k)}Z_{j,-aq^{k}}^{-p_{ij}^{+}(k)} & i\in\overline{J},
\end{cases}
\]
where $p_{ij}^{\pm}(k)\in\mathbb{Z}$ are given by
\[
\widetilde{C}_{ij}^{r}=\frac{1}{\det C(q^{r},\widetilde{q}^{r})}\sum_{k\in\mathbb{Z}}(p_{ij}^{+}(k)q^{kr}+p_{ij}^{-}(k)(-q^{k})^{r}).
\]
We note that $\tau_{J}(Y_{M,a}\widetilde{Y}_{M,-a})=D$ so it is well-defined.
The naive restriction map $\beta_{J}$ is then the composition of
$\tau_{J}$ and the specialization $Z_{j,b}=1$ for all $j,b$.
\begin{lem}[{cf. \cite[Lemma 3.4]{FM}}]
\label{lem:restriction-module} Suppose we write $\tau_{J}(\chi_{q}(V))=\sum P_{k}Q_{k}$
where $P_{k}\in\mathbb{Z}[Y_{i,a}^{\pm1}]_{i\in J,a\in\mathbf{k}^{\times}}$
and $Q_{k}$ is a pairwise distinct monomial in $Z_{j,b}^{\pm1}$
for $j\in\overline{J}$ and $b\in\mathbf{k}^{\times}$. Then the restriction
of $V$ to $\mathcal{U}(\epsilon)_{J}$ is isomorphic to $\bigoplus_{k}V_{k}$
for some $\mathcal{U}(\epsilon)_{J}$-modules $V_{k}$ with $\chi_{q}(V_{k})=P_{k}$.
\end{lem}

\begin{proof}
(See \cite[Section 3.2]{FM}) Recall that the eigenvalue of $h_{j,r}$
(\ref{eq:typ-eigenval-h}) associated to the $\ell$-weight $Y_{i,a}$
(resp. $\widetilde{Y}_{i,a}$) is 
\[
\delta_{ij}\frac{q^{r}-q^{-r}}{r(q-q^{-1})}a^{r}\quad(\text{resp. }\delta_{ij}\frac{\widetilde{q}^{r}-\widetilde{q}^{-r}}{r(q-q^{-1})}a^{r})
\]
and accordingly the eigenvalue of $\widetilde{h}_{j,r}$ is
\[
\sum_{k\in I}\widetilde{C}_{kj}^{r}\delta_{ki}\frac{q^{r}-q^{-r}}{r(q-q^{-1})}a^{r}=\widetilde{C}_{ij}^{r}\frac{q^{r}-q^{-r}}{r(q-q^{-1})}a^{r}\quad(\text{resp. }\widetilde{C}_{ij}^{r}\frac{\widetilde{q}^{r}-\widetilde{q}^{-r}}{r(q-q^{-1})}a^{r})
\]
which we rewrite
\begin{align*}
 & \frac{q^{r}-q^{-r}}{r(q-q^{-1})\det C(q^{r},\widetilde{q}^{r})}\sum_{k\in\mathbb{Z}}\left(p_{ij}^{+}(k)(aq^{k})^{r}+p_{ij}^{-}(k)(-aq^{k})^{r}\right)\\
(\text{resp. } & \frac{q^{r}-q^{-r}}{r(q-q^{-1})\det C(q^{r},\widetilde{q}^{r})}\sum_{k\in\mathbb{Z}}\left(-p_{ij}^{+}(k)(-aq^{k})^{r}p_{ij}^{-}(k)(aq^{k})^{r}\right)).
\end{align*}
Therefore, compared with the definition above, $\tau_{J}(Y_{i,a})$
(resp. $\tau_{J}(\widetilde{Y}_{i,a})$) encodes the collection of
eigenvalues of $\widetilde{h}_{j,r}$ for $j\in\overline{J}$ and
of $h_{j,r}$ for $j\in J$ on the $\ell$-weight $Y_{i,a}$ (resp.
$\widetilde{Y}_{i,a}$). The extra generator $D$ in the case of $M\notin J$
compensates the difference between the eigenvalue of $k_{\delta_{l}}$'s
on $Y_{M,a}$ and $\widetilde{Y}_{M,-a}$. Now the statement follows
by restricting $V$ to $\mathcal{U}(\epsilon)_{J}\otimes\mathcal{U}^{0}(\epsilon)_{J}^{\perp}$
as in the proof of \cite[Lemma 3.4]{FM}.
\end{proof}
\begin{lem}[{cf. \cite[Lemma 3.3]{FM}}]
 \label{lem:restriction-inj}
\begin{enumerate}
\item If $J=\{j\}$ for $j\neq M$ or $J=\{M,M\pm1\}$, then $\tau_{J}$
is injective. 
\item If $J=\{M\}$, then $\tau_{J}$ has the kernel generated by $Y_{M-1,a}\widetilde{Y}_{M+1,a}$
for $a\in\mathbb{C}^{\times}$, where we assume $Y_{M-1,a}=1$ if
$M=1$ and $\widetilde{Y}_{M+1,a}=1$ if $N=1$.
\end{enumerate}
\end{lem}

\begin{proof}
Let $D^{r}$ be the diagonal matrix whose $(i,i)$-th entry is $(-1)^{(r-1)\epsilon_{i}}$.
We shall consider the determinant of the submatrix of $\det C(q^{r},\widetilde{q}^{r})\cdot D^{r}\widetilde{C}^{r}$
obtained by removing the $j$-th rows and columns for $j\in J$. If
$J=\{j\}$ for some $j\in\mathring{J}$, it is just the $(j,j)$-th
entry of $C(q^{r},\widetilde{q}^{r})$ up to a non-zero scalar multiple.
Hence the submatrix is of full rank, equivalently $\tau_{\{j\}}$
is injective, if and only if $j\neq M$. 

For $J=\{M-1,M\}$, we use the Schur complement formula for block
matrices:
\begin{align*}
\text{for }A & =\begin{pmatrix}W & X\\
Y & Z
\end{pmatrix}\text{ with invertible }W,\quad\det A=\det W\cdot\det(Z-YW^{-1}X).
\end{align*}
Take as $A$ the matrix obtained from $C(q^{r},\widetilde{q}^{r})\cdot D^{r}$
by permuting the rows and columns so that the $(M-1)$-th and $M$-th
rows and columns are moved to the first and the second respectively,
and $W$ its principal $2\times2$ submatrix, which is nothing but
$\begin{pmatrix}q^{r}+q^{-r} & -1\\
-1 & 0
\end{pmatrix}$. Since both $A$ and $W$ are invertible, so is $Z-YW^{-1}X$. Then
from the Gaussian elimination of $A$, it turns out that the submatrix
we are seeking is (up to scalar multiple) exactly the inverse of $Z-YW^{-1}X$
and so invertible, as desired. The other case $J=\{M,M+1\}$ is parallel.

Finally, the last argument also shows that if we remove only the $M$-th
column then it has a $1$-dimensional cokernel. Together with the
formulas for $\widetilde{C}^{r}$, this implies the second statement.
\end{proof}
\begin{prop}
\label{prop:restriction-Ainv} Assume $J=\{j\}\text{ or }\{M,M\pm1\}$
and consider the following diagram, where the vertical map on the
right is the multiplication by $\overline{A}_{j,a}^{-1}=\beta_{J}(A_{j,a}^{-1})$,
$j\in J$:\begin{equation}\label{eq:comm-diag-restriction}\begin{tikzcd}
\mathscr{Y}(\epsilon) \arrow[r,"\tau_{J}"] \arrow[d] & \mathscr{Y}(\epsilon)^{(J)} \arrow[d, "\beta_J (A_{j,a}^{-1})"] \\
\mathscr{Y}(\epsilon) \arrow[r,"\tau_{J}"] & \mathscr{Y}(\epsilon)^{(J)}
\end{tikzcd}\end{equation}Then the multiplication by $A_{j,a}^{-1}$ (as a map $\mathscr{Y}(\epsilon)\rightarrow\mathscr{Y}(\epsilon)$)
makes the diagram commutative, and such map is unique unless $J=\{M\}$.
\end{prop}

The refined homomorphism $\tau_{J}$ allows us to reduce the various
problems, including the computation of $q$-characters, to rank 1
and rank 2 cases. First let $j\neq M$ and consider the restriction
of a given $\mathcal{U}(\epsilon)$-module to $\mathcal{U}(\epsilon)_{J}\cong U_{q_{j}}(\widehat{\mathfrak{sl}}_{2})$.
By Lemma \ref{lem:restriction-module} it is a direct sum of $U_{q_{j}}(\widehat{\mathfrak{sl}}_{2})$-modules,
and if two vectors $v,w$ in such a module has $\ell$-weights (over
$U_{q_{j}}(\widehat{\mathfrak{sl}}_{2})$) differing by $\overline{A}_{j,a}^{-1}$,
then their $\ell$-weights over $\mathcal{U}(\epsilon)$ also differ
by $A_{j,a}^{-1}$ by Proposition \ref{prop:restriction-Ainv}.

Suppose $j=M$ and restrict again to $\mathcal{U}(\epsilon)_{J}\cong\mathcal{U}(01)$
whose representation theory is summarized in Appendix \ref{sec:GQG-RT-rank1}.
We may take a similar approach, and since every simple module is of
highest $\ell$-weight, it is enough to consider the following case:
suppose we have two $\ell$-weight vectors $v,w$ in a $\mathcal{U}(\epsilon)_{J}$-module
such that $w\in\sum_{k\in\mathbb{Z}}x_{M,k}^{-}v$ so that their $\ell$-weights
differ by $\overline{A}_{M,a}^{-1}=Y_{a}^{-1}\widetilde{Y}_{a}^{-1}=D^{-1}$.
While $\tau_{J}$ is not injective in this case and the ratio of their
$\ell$-weights over $\mathcal{U}(\epsilon)$ can be of the form
\[
A_{j,a}^{-1}\cdot\prod_{i}(Y_{M-1,b_{i}}\widetilde{Y}_{M+1,b_{i}})^{\pm1},
\]
Theorem \ref{thm:integral-l-wt} ensures us that we do not have the
second, unexpected factor. 

However, as the $q$-character of the prime simple $\mathcal{U}(01)$-modules
in $\mathcal{F}(01)$ (\textit{i.e.} $V(i,j)_{a}$ in Appendix \ref{sec:GQG-RT-rank1})
are of the form $\Psi+\Psi D^{-1}$ regardless of the highest $\ell$-weight
$\Psi$, we cannot determine the exact value of the spectral parameter
$a$ by restriction to rank $1$. This can be settled by restricion
to $\mathcal{U}(\epsilon)_{\{M-1,M\}}\cong\mathcal{U}(001)$ (when
$M=1$, one may consider $J=\{M,M+1\}$ and $\mathcal{U}(011)$-modules
instead). Indeed, since every simple $\mathcal{U}(001)$-module in
$\mathcal{F}(001)$ can be obtained as a subquotient of a tensor product
of fundamental representations whose $q$-characters are completely
known, we can determine $a$ given the $\ell$-weight of $v$. 

From our discussion, we have the following lattice property of the
spectral parameters in the $\ell$-weights of fundamental representations,
and hence of simple modules in $\mathcal{F}(\epsilon)$.
\begin{prop}
We have for $1\leq r\leq M$
\[
\chi_{q}(L(Y_{r,a}))\in Y_{r,a}\cdot\mathbb{Z}[A_{i,aq^{2k+r-i}}^{-1},A_{j,aq^{M-i}\widetilde{q}^{2k+j-M}}^{-1}]_{i\leq M\leq j,k\in\mathbb{Z}},
\]
and for $M\leq r^{\prime}\leq n-1$
\[
\chi_{q}(L(\widetilde{Y}_{r^{\prime},a}))\in\widetilde{Y}_{r^{\prime},a}\cdot\mathbb{Z}[A_{i,aq^{2k+M-i}\widetilde{q}^{M-r^{\prime}}}^{-1},A_{j,a\widetilde{q}^{2k+j-r^{\prime}}}^{-1}]_{i\leq M\leq j,k\in\mathbb{Z}}.
\]
\end{prop}

\[
\]
This motivates us to introduce the following subcategory which plays
a role of the `skeletal' subcategory in $\mathcal{F}(\epsilon)$ (\textit{cf.}
Proposition \ref{prop:denom-R-++--} and Theorem \ref{thm:denom-R-+-}).
\begin{defn}
The category $\mathcal{F}_{\mathbb{Z}}(\epsilon)$ consists of finite-dimensional
$\mathcal{U}(\epsilon)$-modules whose $\ell$-weights are monomials
in $Y_{i,q^{i-1+2k}}^{\pm1}$ ($i=1,\dots,M$) and $\widetilde{Y}_{j,(-1)^{j-M+1}q^{j-1+2k}}^{\pm1}$
($j=M,\dots,n-1$) for $k\in\mathbb{Z}$. 

Equivalently, it is the full abelian subcategory of $\mathcal{F}(\epsilon)$
whose simple objects are $L(\Psi)$ for monomials $\Psi$ in $Y_{i,q^{i-1+2m}}$
and $\widetilde{Y}_{j,(-1)^{j-M+1}q^{j-1+2m}}$.
\end{defn}

\begin{rem}
If we consider the subcategory whose simple modules are those whose
highest $\ell$-weights are monomials in $Y_{i,q^{i-1+2m}}$ ($i=1,\dots,M$,
$m\in\mathbb{Z}$) only, then this is exactly the category $\mathcal{C}_{J}(\epsilon)$
studied in \cite[Section 6.1]{KL} where nice dualities such as a
quantum affine super analogue of Schur--Weyl duality and a super
duality are available. Our category $\mathcal{F}_{\mathbb{Z}}(\epsilon)$
is then the smallest full monoidal Serre subcategory of $\mathcal{F}(\epsilon)$
that contains $\mathcal{C}_{J}(\epsilon)$ and closed under left and
right duals.

On the other hand, through the algebra isomorphism $\tau$, $\mathcal{F}_{\mathbb{Z}}(\epsilon)$
corresponds to the category of finite-dimensional $U(\epsilon)$-modules
whose $\ell$-weights are monomials in $\mathrm{Y}_{i,q^{i-1+2k}}^{\pm1}$
($i\in\mathring{I}$, $k\in\mathbb{Z}$), see Remark \ref{rem:compare-fund-GQG-qasa}.
\end{rem}

\subsection{Frenkel--Mukhin-type algorithm for $q$-characters\label{subsec:FM-algorithm}}

We are ready to propose a super analogue of Frenkel--Mukhin algorithm
to compute the $q$-character of simple modules in $\mathcal{F}(\epsilon)$.
Since the flow of the algorithm is the same as non-super cases, we
refer the reader to the original paper \cite[Section 5.5]{FM} for
the detailed description and focus on the new rule on the `expansion
in the $M$-th direction' that comes from the representation theory
of $\mathcal{U}(01)$ (Appendix \ref{sec:GQG-RT-rank1}).

The algorithm roughly proceeds as follows. Begin from a given $\ell$-weight
$\Psi$ which we `color' with the $\mathring{I}$-tuple $(0,0,\dots,0)$
and give the multiplicity 1 in order to obtain the $q$-character
of $L(\Psi)$. We also assume $\Psi$ is a monomial in $Y_{i,a}$
and $\widetilde{Y}_{j,a}$ for $i\leq M\leq j$ and $a\in\mathbf{k}^{\times}$
as we are concerning the category $\mathcal{F}(\epsilon)$. At each
step, we take an $\ell$-weight $m$ with color $(s_{1},\dots,s_{n-1})$
and multiplicity $s$ which is the highest (in the usual partial order
on $\mathring{Q}$) among those given in previous steps with color
$(s_{1}^{\prime},\dots,s_{n-1}^{\prime})$ and multiplicity $s^{\prime}$
such that $s_{i}<s$ for some $i\in\mathring{I}$.

If $s_{i}<s$ for $i\neq M$, then we do as in the non-super case.
The algorithm fails if $\beta_{\{i\}}(m)$ is not dominant; otherwise
we add $\ell$-weights obtained by \textit{expanding} $m$ \textit{in
the $i$-th direction}. Namely, suppose the $q$-character of $\mathcal{U}(\epsilon)_{\{i\}}\cong U_{q_{i}}(\widehat{\mathfrak{sl}}_{2})$-module
$L(\beta_{\{i\}}(m))$, which is finite-dimensional by assumption,
is given by
\[
\chi_{q}(L(\beta_{\{i\}}(m)))=\beta_{\{i\}}(m)\cdot\sum_{k}c_{k}\prod_{l}\overline{A}_{i,a_{k,l}}^{-1}\in\beta_{\{i\}}(m)\cdot\mathbb{Z}_{\geq0}[\overline{A}_{i,a}^{-1}]_{a\in\mathbf{k}^{\times}}
\]
where $\overline{A}_{i,a}=\beta_{\{i\}}(A_{i,a})$ is the simple $\ell$-root
for $\mathcal{U}(\epsilon)_{\{i\}}$. For each $\ell$-weight $\beta_{\{i\}}(m)\cdot\prod\overline{A}_{i,a_{k,l}}^{-1}$
other than $\beta_{\{i\}}(m)$, we add the lifted $\ell$-weight $m\cdot\prod A_{i,a_{k,l}}^{-1}$
(Proposition \ref{prop:restriction-Ainv}). If the latter has not
been added in previous steps, we add it with color $(\dots,0,(s-s_{i})c_{k},0,\dots)$
and multiplicity $(s-s_{i})c_{k}$; if there already exists with color
$(t_{1},\dots,t_{n-1})$ and multiplicity $t$, then we change its
$i$-th color to $t_{i}+(s-s_{i})c_{k}$ and its multiplicity to $\max\{t,t_{i}+(s-s_{i})c_{k}\}$.

Now assume $i=M$. The rule of adding new $\ell$-weights are exactly
the same, but there are two differences. First, we do not have a notion
of dominance for $\beta_{\{M\}}(m)$, as the $\mathcal{U}(01)$-module
$L(\beta_{\{M\}}(m))$ is always finite-dimensional. Hence we never
fail when we expand any $\ell$-weight $m$ in the $M$-th direction.
Second, suppose again the $q$-character of $L(\beta_{\{M\}}(m))$
is given by
\[
\chi_{q}(L(\beta_{\{M\}}(m)))=\beta_{\{M\}}(m)\cdot\sum_{k}c_{k}\overline{D}^{-d_{k}}.
\]
To lift each $\ell$-weight $\beta_{\{M\}}(m)\cdot\overline{D}^{-d_{k}}$
to $m\cdot\prod_{l=1}^{d_{k}}A_{M,a_{k,l}}^{-1}$, we have to determine
the spectral parameters $a_{k,l}$. 

As explained in the last section, in order to do that we shall take
into account the restriction to $J=\{M-1,M\}$ (to $\{M,M+1\}$ if
$M=1$, where one should consider the representation theory of $\mathcal{U}(011)$
analogous to Section \ref{subsec:rank2-case}) at the same time. Recall
from the end of Appendix \ref{sec:GQG-RT-rank1} that we may decompose
uniquely
\[
\beta_{\{m\}}(m)=\overline{D}^{-s}\prod_{r}(\overline{Y}_{M,a_{r}}\overline{Y}_{M,a_{r}q^{2}}\cdots\overline{Y}_{M,a_{r}q^{2i_{r}-2}})\prod_{s}(\overline{\widetilde{Y}}_{M,b_{s}}\overline{\widetilde{Y}}_{M,b_{s}\widetilde{q}^{2}}\cdots\overline{\widetilde{Y}}_{M,b_{s}\widetilde{q}^{2j_{s}-2}})
\]
for some $s\in\mathbb{Z}$, $i_{r},j_{s}\in\mathbb{Z}_{>0}$ and $a_{r},b_{S}\in\mathbf{k}^{\times}$
such that each factors are pairwise in general position (see the appendix
for definition), so that
\[
L(\beta_{\{m\}}(m))\cong(D)^{\otimes-s}\otimes\left(\bigotimes_{r}L(\overline{Y}_{M,a_{r}}\cdots\overline{Y}_{M,a_{r}q^{2i_{r}-2}})\right)\otimes\left(\bigotimes_{s}L(\overline{\widetilde{Y}}_{M,b_{s}}\cdots\overline{\widetilde{Y}}_{M,b_{s}\widetilde{q}^{2j_{s}-2}})\right).
\]
Each factor is 2-dimensional and has a very simple $q$-character
\begin{align*}
\chi_{q}L(\overline{Y}_{M,a_{r}}\cdots\overline{Y}_{M,a_{r}q^{2i_{r}-2}}) & =\overline{Y}_{M,a_{r}}\cdots\overline{Y}_{M,a_{r}q^{2i_{r}-2}}(1+\overline{D}^{-1}),\\
\chi_{q}L(\overline{\widetilde{Y}}_{M,b_{s}}\cdots\overline{\widetilde{Y}}_{M,b_{s}\widetilde{q}^{2j_{s}-2}}) & =\overline{\widetilde{Y}}_{M,b_{s}}\cdots\overline{\widetilde{Y}}_{M,b_{s}\widetilde{q}^{2j_{s}-2}}(1+\overline{D}^{-1}).
\end{align*}
We are in a position to determine $a\in\mathbf{k}^{\times}$ to add
the $\ell$-weight $Y_{M,a_{r}}\cdots Y_{M,a_{r}q^{2i_{r}-2}}A_{M,a}^{-1}$
that lifts $\overline{Y}_{M,a_{r}}\cdots\overline{Y}_{M,a_{r}q^{2i_{r}-2}}\overline{D}^{-1}$.
From the $q$-character of the $\mathcal{U}(001)$-module $L(\beta_{\{M-1,M\}}(Y_{M,a_{r}}Y_{M,a_{r}q^{2}}\cdots Y_{M,a_{r}q^{2i_{r}-2}}))$
from Section \ref{subsec:rank2-case}, we see that $a=aq^{2i_{r}-1}$
in this case. Similarly, $\overline{\widetilde{Y}}_{M,b_{s}}\cdots\overline{\widetilde{Y}}_{M,b_{s}\widetilde{q}^{2j_{s}-2}}\overline{D}^{-1}$
should be lifted to $\widetilde{Y}_{M,b_{s}}\cdots\widetilde{Y}_{M,b_{s}\widetilde{q}^{2j_{s}-2}}A_{M,a\widetilde{q}^{2j_{s}-1}}^{-1}$.
Therefore, we understand which $\ell$-weights should be added by
expanding $m$ in $M$-th direction, and we assign the color and multiplicity
in the same way as above.

Below we give several examples of the implementation of the algorithm,
where we omit the color $(s_{i})$ for $\ell$-weights. We depict
the result of the algorithm by a directed graph whose vertices are
the $\ell$-weights with multiplicity obtained through the algorithm
and arrows are of the form $\Psi\overset{i,a}{\longrightarrow}\Psi^{\prime}$
to represent $\Psi^{\prime}=\Psi A_{i,a}^{-1}$.
\begin{example}
Consider $\mathcal{U}(\epsilon_{3|2})$-modules $L(Y_{1,1})$ and
$L(\widetilde{Y}_{4,1})$.

\begin{equation*}\begin{tikzcd}
Y_{1,1} \arrow[d,"{1,q}"]   &&  \widetilde{Y}_{4,1} \arrow[d,"{4,-q^{-1}}"] \\
Y^{-1}_{1,q^2} Y_{2,q} \arrow[d,"{2,q^2}"] && \widetilde{Y}_{3,-q^{-1}} \widetilde{Y}^{-1}_{4,q^-2} \arrow[d,"{3,q^{-2}}"] \\
Y^{-1}_{2,q^3} Y_{3,q^2} \arrow[d,"{3,q^3}"]  && Y_{2,q^{-2}} Y^{-1}_{3,q^{-1}}  \arrow[d,"{2,q^{-1}}"] \\
\widetilde{Y}^{-1}_{3,-q^2} \widetilde{Y}_{4,q^3}  \arrow[d,"{4,-q^2}"]  && Y_{1,q^{-1}} Y^{-1}_{2,1} \arrow[d,"{1,1}"] \\
\widetilde{Y}^{-1}_{4,q}  &&Y^{-1}_{1,q} 
\end{tikzcd}\end{equation*}
\end{example}

\begin{example}
Consider the $\mathcal{U}(\epsilon_{3|1})$-module $L(Y_{3,1})$.
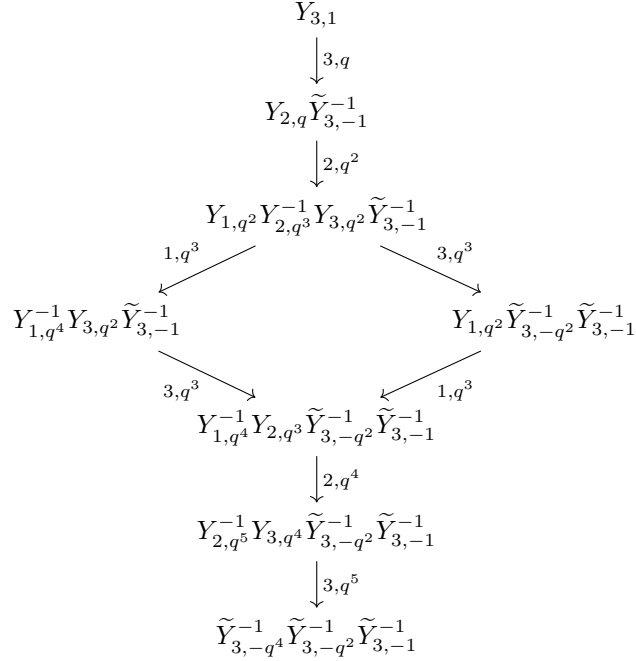
\begin{equation*}\begin{tikzcd}[column sep=-.1cm]
& Y_{3,1} \arrow[d,"{3,q}"] & \\
& Y_{2,q} \widetilde{Y}^{-1}_{3,-1} \arrow[d,"{2,q^2}"]   & \\
& Y_{1,q^2} Y^{-1}_{2,q^3} Y_{3,q^2} \widetilde{Y}^{-1}_{3,-1} \arrow[ld,"{1,q^3}"']\arrow[rd,"{3,q^3}"] & \\
Y^{-1}_{1,q^4} Y_{3,q^2} \widetilde{Y}^{-1}_{3,-1} \arrow[rd,"{3,q^3}"']&& Y_{1,q^2}  \widetilde{Y}^{-1}_{3,-q^2} \widetilde{Y}^{-1}_{3,-1}  \arrow[ld,"{1,q^3}"]\\
&  Y^{-1}_{1,q^4}  Y_{2,q^3}  \widetilde{Y}^{-1}_{3,-q^2} \widetilde{Y}^{-1}_{3,-1} \arrow[d,"{2,q^4}"] &\\
& Y^{-1}_{2,q^5}  Y_{3,q^4} \widetilde{Y}^{-1}_{3,-q^2} \widetilde{Y}^{-1}_{3,-1} \arrow[d,"{3,q^5}"] &\\
& \widetilde{Y}^{-1}_{3,-q^4} \widetilde{Y}^{-1}_{3,-q^2} \widetilde{Y}^{-1}_{3,-1} &
\end{tikzcd}\end{equation*}

When we expand at $\Psi=Y_{1,q^{2}}Y_{2,q^{3}}^{-1}Y_{3,q^{2}}\widetilde{Y}_{3,-1}^{-1}$
in the 3rd direction, as $\beta_{\{3\}}(\Psi)=Y_{3,q^{2}}\widetilde{Y}_{3,-1}^{-1}=Y_{3,q^{2}}Y_{3,1}D^{-1}$
we consider the 2-dimensional $\mathcal{U}(01)$-module $L(Y_{q^{2}}Y_{1}D^{-1})\cong L(Y_{q^{2}}Y_{1})\otimes D^{-}$.
Hence we add just one $\ell$-weight $Y_{1,q^{2}}Y_{2,q^{3}}^{-1}Y_{3,q^{2}}Y_{3,1}A_{3,q^{3}}^{-1}D^{-1}$,
where the spectral parameter $q^{3}$ of $A^{-1}$ is determined from
the $q$-character of the $\mathcal{U}(001)$-module $L(Y_{2,q^{2}}Y_{2,1})$. 
\end{example}

\begin{example}
The algorithm also applies to modules beyond fundamental representations.
For instance, the $q$-characters of a $5$-dimensional KR module
$L(Y_{1,q^{-2}}Y_{1,1})$ and a $7$-dimensional $L(Y_{1,q^{-4}}Y_{1,q^{-2}}Y_{1,1})$
over $\mathcal{U}(001)$ are computed correctly:
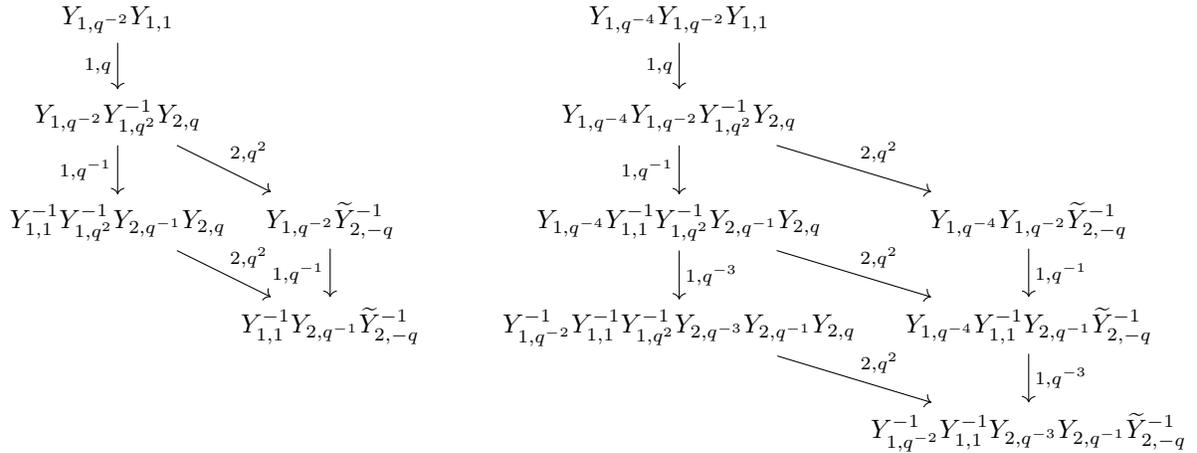
\begin{equation*}\begin{tikzcd}[column sep=-.1cm]
Y_{1,q^{-2}} Y_{1,1}   \arrow[d,"{1,q}"'] & &\quad\quad  &   Y_{1,q^{-4}}Y_{1,q^{-2}} Y_{1,1}  \arrow[d,"{1,q}"']& \\
Y_{1,q^{-2}}Y^{-1}_{1,q^2}Y_{2,q} \arrow[d,"{1,q^{-1}}"']  \arrow[rd,"{2,q^2}"]  && \quad\quad  &  Y_{1,q^{-4}}Y_{1,q^{-2}}Y^{-1}_{1,q^2}Y_{2,q} \arrow[d,"{1,q^{-1}}"']  \arrow[rd,"{2,q^2}"]&\\
Y^{-1}_{1,1}Y^{-1}_{1,q^2}Y_{2,q^{-1}}Y_{2,q}  \arrow[rd,"{2,q^2}"] & Y_{1,q^{-2}}\widetilde{Y}^{-1}_{2,-q} \arrow[d,"{1,q^{-1}}"']   & \quad\quad & Y_{1,q^{-4}}Y^{-1}_{1,1}Y^{-1}_{1,q^2}Y_{2,q^{-1}}Y_{2,q} \arrow[d,"{1,q^{-3}}"] \arrow[rd,"{2,q^2}"]  &  Y_{1,q^{-4}}Y_{1,q^{-2}}\widetilde{Y}^{-1}_{2,-q} \arrow[d,"{1,q^{-1}}"] \\
& Y^{-1}_{1,1}Y_{2,q^{-1}}\widetilde{Y}^{-1}_{2,-q}    &\quad\quad&   Y^{-1}_{1,q^{-2}}Y^{-1}_{1,1}Y^{-1}_{1,q^2}Y_{2,q^{-3}}Y_{2,q^{-1}}Y_{2,q}  \arrow[rd,"{2,q^2}"] & Y_{1,q^{-4}}Y^{-1}_{1,1}Y_{2,q^{-1}} \widetilde{Y}^{-1}_{2,-q} \arrow[d,"{1,q^{-3}}"]  \\
&     & \quad\quad  &&  Y^{-1}_{1,q^{-2}}Y^{-1}_{1,1}Y_{2,q^{-3}}Y_{2,q^{-1}}\widetilde{Y}^{-1}_{2,-q}
\end{tikzcd}\end{equation*}One can easily generalize it to obtain the $q$-character of $L(Y_{1,q^{2-2s}}Y_{1,q^{4-2s}}\cdots Y_{1,1})$
which is $(2s+1)$-dimensional.
\end{example}

\begin{example}
\label{exa:FM-alg-fails}Consider $8$-dimensional $\mathcal{U}(\epsilon_{2|1})$-modules
$L(Y_{1,1}\widetilde{Y}_{2,-q^{-1}})$ and $L(Y_{1,1}\widetilde{Y}_{2,-q})$.\begin{equation*}\begin{tikzcd}[column sep=-.1cm]
&Y_{1,1}\widetilde{Y}_{2,-q^{-1}} \arrow[d,"{1,q^3}"'] \arrow[rd,"{2,q^{-2}}"]&&  \quad\quad && Y_{1,1}\widetilde{Y}_{2,-q} \arrow[ld,"{1,q}"']\arrow[rd,"{2,1}"] & \\
&Y^{-1}_{1,q^2} Y_{2,q} \widetilde{Y}_{2,-q^{-1}} \arrow[ld,"{2,q^2}"']\arrow[rd,"{2,q^{-2}}"']&   Y_{1,1}Y_{1,q^{-2}} Y^{-1}_{2,q^{-1}}  \arrow[d,"{1,q}"]&  \quad\quad &Y^{-1}_{1,q^2}Y_{2,q} \widetilde{Y}_{2,-q}  && Y^{2}_{1,1}Y^{-1}_{2,q} \arrow[ld,"{1,q}"'] \\
\widetilde{Y}_{2,-q^{-1}}\widetilde{Y}^{-1}_{2,-q} \arrow[rd,"{2,q^{-2}}"'] &&Y_{1,q^{-2}} Y^{-1}_{1,q^2} Y_{2,q} Y^{-1}_{2,q^{-1}} \arrow[ld,"{2,q^2}"'] \arrow[d,"{1,q^{-1}}"] &  \quad\quad && 2Y_{1,1}Y^{-1}_{1,q^2} \arrow[ld,"{1,q}"'] & \\
&Y_{1,q^{-2}} Y^{-1}_{2,q^{-1}}\widetilde{Y}^{-1}_{2,-q} \arrow[d,"{1,q}"']& Y^{-1}_{1,1} Y^{-1}_{1,q^2} Y_{2,q} \arrow[ld,"{2,q^2}"]&  \quad\quad  & Y^{-2}_{1,q^2}  Y_{2,q} \arrow[d,"{2,q^2}"] & &  \\
&  Y^{-1}_{1,1} \widetilde{Y}^{-1}_{2,-q} &&   \quad\quad  &Y^{-1}_{1,q^2} \widetilde{Y}^{-1}_{2,-q} &&
\end{tikzcd}\end{equation*}

For the first one, the algorithm does not fail and yields the correct
$q$-character. On the other hand, it fails for the other one at its
lowest $\ell$-weight $\Psi=Y_{1,q^{2}}^{-1}\widetilde{Y}_{2,-q}^{-1}$,
as it has the color $(0,1)$ with multiplicity $1$ but $\beta_{\{1\}}(\Psi)=Y_{1,q^{2}}^{-1}$
is not dominant. Indeed, the algorithm could not capture the missing
$\ell$-weight $Y_{1,1}Y_{2,q}^{-1}\widetilde{Y}_{2,-q}^{-1}$ which
is `dominant' in the sense that it gives rise to a finite-dimensional
simple module.
\end{example}

As the last example suggests, it is hard to determine for which $\ell$-weight
$\Psi$ this algorithm is \textit{well-defined}, that is it never
fails and ends in finite steps, and gives us the correct $q$-character
$\chi_{q}(L(\Psi))$. More precisely, in all known examples, the algorithm
is not well-defined when it fails to capture a `dominant' $\ell$-weight
of $L(\Psi)$. Indeed, in the non-super cases one sufficient condition
is that the module $L(\Psi)$ is \textit{minuscule,} that is the only
dominant $\ell$-weight of $L(\Psi)$ is the highest one \cite[Theorem 5.9]{FM}.
In our super case the situation is more subtle as the usual notion
of dominance as above is not strong enough (\textit{cf.} Remark \ref{rem:intble-fails}),
which results in more examples where the algorithm fails. Still, it
is possible to apply the argument therein to fundamental representations
of $\mathcal{U}(\epsilon)$ as follows.
\begin{thm}
For the fundamental representations $L(Y_{i,a})$ and $L(\widetilde{Y}_{j,a})$,
the algorithm is well-defined and gives the correct $q$-character.
\end{thm}

\begin{proof}
We may draw a directed graph from $\chi_{q}(L(Y_{i,a}))$ as follows.
Its vertex set is the multiset of the $\ell$-weights of $\chi_{q}(L(Y_{i,a}))$.
For each $i\in\mathring{I}$ we first partition the vertex set so
that each component forms a $q$-character of a simple $\mathcal{U}(\epsilon)_{\{i\}}$-modules,
and then we have an arrow $\Psi\overset{i}{\longrightarrow}\Psi^{\prime}$
whenever $\Psi$ and $\Psi^{\prime}$ belong to the same component
and $\Psi^{\prime}=\Psi A_{i,b}^{-1}$ for some $b\in\mathbf{k}^{\times}$. 

We claim that the highest $\ell$-weight is the only vertex in the
graph without incoming arrows. Indeed, let $m$ be such a vertex.
This first implies that for each $i\neq M$, $\beta_{\{i\}}(m)$ is
dominant, that is a monomial in $\overline{Y}_{i,b}$ for $b\in\mathbf{k}^{\times}$.
Hence we may write $m$ in the following form 
\[
m=\prod_{r\in R}Y_{i_{r},b_{r}}\cdot\prod_{t\in T_{+}}Y_{M,c_{t}^{+}}\prod_{t\in T_{-}}Y_{M,c_{t}^{-}}^{-1}D^{k}\cdot\prod_{s\in S}\widetilde{Y}_{j_{s},d_{s}},\quad i_{r}<M<j_{s}
\]
so that its weight is
\[
\sum_{r}(\delta_{1}+\cdots+\delta_{i_{r}})+(\left|T_{+}\right|-\left|T_{-}\right|)(\delta_{1}+\cdots+\delta_{M})+k(\delta_{1}+\cdots+\delta_{M}-\delta_{M+1}-\cdots-\delta_{n})-\sum_{s}(\delta_{j_{s}+1}+\cdots+\delta_{n}).
\]
We invoke the well-known weight structure of the underlying $\mathring{\mathcal{U}}(\epsilon)$-module
$V(1^{i})$ of $L(Y_{i,a})$, namely any weight is of the form
\[
\sum_{k\in\mathbb{I}}\lambda_{k}\delta_{k}\in\bigoplus\mathbb{Z}_{\geq0}\delta_{k},\quad\sum\lambda_{k}=i,\,\lambda_{k}\in\{0,1\}\text{ if }k\leq M.
\]
Looking at the coefficients of $\delta_{1},\delta_{M},\delta_{M+1}$
and $\delta_{n}$ of the weight of $m$, we obtain inequalities
\[
0\leq\left|R\right|+\left|T_{+}\right|-\left|T_{-}\right|\leq1,\quad0\leq\left|T_{+}\right|-\left|T_{-}\right|+k\leq1,\quad k\leq0,\quad-k-\left|S\right|\ge0
\]
respectively. 

Suppose $S\neq\emptyset$ so that $k<0$. Then $\left|T_{+}\right|-\left|T_{-}\right|\geq-k>0$,
which forces $R=\emptyset$ and $\left|T_{+}\right|-\left|T_{-}\right|=1=-k$
and then $\left|S\right|=1$. Hence $m$ is of the form $m=Y_{M,c}D^{-1}\widetilde{Y}_{j,d}$
with the weight $\delta_{M+1}+\cdots+\delta_{j}$. However, we know
that $V(1^{a})$ has the unique vector $\ket{\mathbf{e}_{M+1}+\cdots+\mathbf{e}_{j}}$
(see Section \ref{subsec:fund-repn} for the notation) of such weight,
and this vector is obtained as 
\[
\ket{\mathbf{e}_{M+1}+\cdots+\mathbf{e}_{j}}=f_{M}\ket{\mathbf{e}_{M}+\mathbf{e}_{M+1}+\cdots+\mathbf{e}_{j}}.
\]
Hence $m$ has an incoming arrow with label $M$, which contradicts
our assumption.

Therefore, we must have $S=\emptyset$ and the sum of all coefficients
in the weight of $m$ is
\[
\sum i_{r}+(\left|T_{+}\right|-\left|T_{-}\right|+k)M-Nk
\]
which equals to $i$. Since $k\leq0$ and $i\leq M$, this forces
$\left|T_{+}\right|-\left|T_{-}\right|+k=0$ and so $m$ is either
$Y_{i,b}$ (if $k=0$) or $\prod_{t\in T_{+}}Y_{M,c_{t}^{+}}\prod_{t\in T_{-}}Y_{M,c_{t}^{-}}^{-1}D^{-1}$
with $\left|T_{+}\right|-\left|T_{-}\right|=1$ (if $k=-1$). In the
latter case, its weight is $\delta_{M+1}+\cdots+\delta_{n}$ which
again implies $m$ has an incoming arrow with label $M$ as in the
previous paragraph. We conclude $m=Y_{i,a}$, as desired.

The remaining induction argument of the proof of \cite[Theorem 5.9]{FM}
relies on this tree property of the directed graph, not on the specific
representation theory of $U_{q}(\widehat{\mathfrak{sl}}_{2}).$ Therefore
we may repeat it to conclude, and the proof for $L(\widetilde{Y}_{j,a})$
is parallel.
\end{proof}
More generally we have the following conjecture supported by a number
of examples, which asserts that a naive definition of dominance works
for polynomial modules (and also for their duals). 
\begin{conjecture}
Let $\Psi$ be a monomial in $Y_{i,a}$ for $1\leq i\leq M$, $a\in\mathbf{k}^{\times}$,
so that $L(\Psi)$ is a polynomial module. If $L(\Psi)$ has no $\ell$-weight
other than the highest one that is a monomial in $Y_{i,a}$, $\widetilde{Y}_{j,a}$
and $D^{\pm1}$ ($i\leq M<j$, $a\in\mathbf{k}^{\times}$), then the
algorithm is well-defined.
\end{conjecture}

\appendix

\section{Existence of the braid action in rank 2 \label{sec:well-defd-T_i-rank2}}

In this section, we complete the proof of the existence of the braid
action $T_{i}$ \cite{Machida,Yu} by considering the remaining case
of $n=3$ and $M,N>0$. Thanks to the obvious symmetry it is enough
to consider the case $\epsilon=(001)$. Since the verification of
the braid relation is exactly the same as in the general case, we
only check that the formula (\ref{eq:def-braid-action}) for $T_{i}$
defines a well-defined algebra homomorphism.

First, the existence of $T_{1}$ can be shown exactly as in \cite[Chapter 37]{Lb}
(whose $T_{1,1}^{\prime\prime}$ corresponds to our $T_{1}$). Indeed,
an automorphism $T_{1}$ can be defined on any integrable $\mathcal{U}(001)$-modules
as in \cite[Chapter 5]{Lb}, which uniquely lifts to an algebra automorphism.
That $T_{1}$ is given by the formula (\ref{eq:def-braid-action})
can be verified as in \cite[37.2.2]{Lb}, thanks to the even Serre
relation (\ref{eq:GQG-DJ-evenSerre}) at $i=1$.

For $T_{2}:\mathcal{U}(001)\longrightarrow\mathcal{U}(010)$, we shall
verify that $T_{2}$ defined by (\ref{eq:def-braid-action}) preserve
the defining relations of $\mathcal{U}(001)$. Up to the computation
in \cite{Yu}, it suffices to check the quintic Serre relation
\[
\begin{array}{c}
[e_{0},[e_{2},[e_{0},[e_{2},e_{1}]_{q}]_{-1}]]_{-q^{-1}}=[e_{2},[e_{0},[e_{2},[e_{0},e_{1}]_{q}]_{-1}]]_{-q^{-1}}\end{array}.
\]
First consider the image of the lefthand side of the relation under
$T_{2}$. Since $T_{2}[e_{2},e_{1}]_{q}=-qe_{1}$, 
\[
T_{2}[e_{0},[e_{2},[e_{0},[e_{2},e_{1}]_{q}]_{-1}]]_{-q^{-1}}=-q[T_{2}e_{0},[T_{2}e_{2},[T_{2}e_{0},e_{1}]_{-1}]]_{-q^{-1}}.
\]
Recall that $T_{2}e_{2}=-f_{2}k_{2}$. It is easy to check that
\[
[f_{2},[T_{2}e_{0},e_{1}]_{-1}]=(e_{1}e_{0}-q^{-1}e_{0}e_{1})k_{2}^{-1},
\]
and so we have
\begin{align*}
-q[T_{2}e_{0},[T_{2}e_{2},[T_{2}e_{0},e_{1}]_{-1}]]_{-q^{-1}} & =q[T_{2}e_{0},(e_{1}e_{0}-q^{-1}e_{0}e_{1})]_{-q^{-1}}.
\end{align*}

Next, for the righthand side, we first compute
\[
\Delta\coloneqq T_{2}[e_{2},[e_{0},e_{1}]_{q}]_{-1}=q^{-1}e_{0}T_{2}(e_{1})-qT_{2}(e_{1})e_{0}+T_{2}(e_{0})e_{1}+q^{2}e_{1}T_{2}(e_{0})
\]
by invoking 
\[
T_{2}(e_{1}e_{2}-q^{-1}e_{2}e_{1})=e_{1},\quad T_{2}(e_{0}e_{2}+qe_{2}e_{0})=e_{0}.
\]
Then the image of the righthand side is
\[
[T_{2}e_{2},[T_{2}e_{0},\Delta]]_{-q^{-1}}=q^{-1}[f_{2},[T_{2}e_{0},\Delta]]k_{2}.
\]
One can readily check that $[f_{2},\Delta]=0$ and then
\begin{align*}
 & =q^{-1}[[f_{2},T_{2}e_{0}],\Delta]k_{2}\\
 & =q^{-1}[e_{0},\Delta]\\
 & =qe_{0}e_{1}T_{2}(e_{0})-qe_{1}T_{2}(e_{0})e_{0}+q^{-1}e_{0}T_{2}(e_{0})e_{1}-q^{-1}T_{2}(e_{0})e_{1}e_{0}\\
 & \quad+q^{-2}e_{0}^{2}T_{2}(e_{1})-q^{-2}e_{0}T_{2}(e_{1})e_{0}-e_{0}T_{2}(e_{1})e_{0}+T_{2}(e_{1})e_{0}^{2}.
\end{align*}

It remains to check that 
\begin{align*}
q[T_{2}e_{0}, & (e_{1}e_{0}-q^{-1}e_{0}e_{1})]_{-q^{-1}}\\
= & qe_{0}e_{1}T_{2}(e_{0})-qe_{1}T_{2}(e_{0})e_{0}+q^{-1}e_{0}T_{2}(e_{0})e_{1}-q^{-1}T_{2}(e_{0})e_{1}e_{0}\\
 & +q^{-2}e_{0}^{2}T_{2}(e_{1})-q^{-2}e_{0}T_{2}(e_{1})e_{0}-e_{0}T_{2}(e_{1})e_{0}+T_{2}(e_{1})e_{0}^{2}.
\end{align*}
Indeed, after the substitution $T_{2}(e_{0})=e_{2}e_{0}-q^{-1}e_{0}e_{2}$
and $T_{2}(e_{1})=e_{2}e_{1}+qe_{1}e_{2}$ in both sides and the cancellation
of common terms, we are left with
\begin{align*}
qe_{2}e_{0}e_{1}e_{0} & -e_{2}e_{0}^{2}e_{1}-q^{-1}e_{1}e_{0}^{2}e_{2}+q^{-2}e_{0}e_{1}e_{0}e_{2}\\
=e_{0}e_{1}e_{0} & e_{2}-q^{-1}e_{2}e_{0}e_{1}e_{0}+q^{-1}e_{0}^{2}e_{1}e_{2}+e_{2}e_{1}e_{0}^{2}
\end{align*}
which easily follows from the even Serre relation at $i=0$.

\section{Construction of the surjective map $\mathcal{B}$\label{sec:GQG-Drinfeld-constr}}

In this section, we explain a construction of the surjective algebra
homomorphism $\mathcal{B}:\mathcal{U}^{\mathrm{Dr}}(\epsilon)\rightarrow\mathcal{U}(\epsilon)$
following \cite{B2,Y}. Since the proofs are completely parallel with
those in \cite[Section 8]{Y}, we focus on describing the image of
loop generators and formulas adapted to $\mathcal{U}(\epsilon)$.
See Section \ref{subsec:aff-PBW} for (extended) affine Weyl groups
and corresponding braid action.

As in \cite{B2}, we define $\mathcal{B}$ on the loop generators
as
\begin{align*}
x_{i,k}^{+} & \longmapsto o(i)^{k}T_{\varpi_{i}}^{-k}(e_{i}),\quad x_{i,k}^{-}\longmapsto o(i)^{k}T_{\varpi_{i}}^{k}(f_{i}),\quad(k\in\mathbb{Z})\\
\psi_{i,r}^{+} & \longmapsto o(i)^{r}(q-q^{-1})k_{i}\overline{\psi}_{i,r},\,\psi_{i,0}^{+}\longmapsto k_{i},\quad\psi_{i,-r}^{-}\longmapsto\Omega(\mathcal{B}(\psi_{i,r}^{+}))\quad(r\geq0)
\end{align*}
where $o(i)=(-1)^{i}$ and 
\[
\overline{\psi}_{i,r}=c^{-r/2}\left(T_{\varpi_{i}}^{r}(-k_{i}^{-1}f_{i})e_{i}-\mathbf{q}(\alpha_{i},\alpha_{i})^{-1}e_{i}T_{\varpi_{i}}^{r}(-k_{i}^{-1}f_{i})\right)\in c^{-r/2}\mathcal{U}^{+}(\epsilon)\quad(r>0).
\]
The image of $h_{i,r}$ is determined accordingly by (\ref{eq:GQG-psi-halfVO}).
Below, for notational convenience, for $x\in\mathcal{U}^{\mathrm{Dr}}(\epsilon)$
we will just denote by the same symbol $x$ its image $\mathcal{B}(x)$
in $\mathcal{U}(\epsilon)$.

To show that $\mathcal{B}$ is a well-defined algebra homomorphism,
one first has to verify commutation relations between loop generators
(\ref{eq:GQG-Drinfeld-hh}-\ref{eq:GQG-Drinfeld-q-locality}). This
partly relies on the computation in the subalgebra of rank 1 affine
type, which is still valid in our case. For each $i\in\mathring{I}$,
let $\mathcal{U}_{i}$ be the subalgebra of $\mathcal{U}(\epsilon)$
generated by $e_{i},f_{i},k_{i},T_{\varpi_{i}}(-k_{i}^{-1}f_{i}),T_{\varpi_{i}}(-e_{i}k_{i})$
and $ck_{i}^{-1}$, and also $U_{q_{i}}(\widehat{\mathfrak{sl}}_{2})$
the quantum affine algebra associated to $\mathfrak{sl}_{2}$ with
quantum parameter $q_{i}$ (see for example \cite[Theorem 2.3]{CP}). 
\begin{prop}[{cf. \cite[Proposition 3.8]{B2}}]
 For $i\in\mathring{I}$ such that $p(i)=0$, there exists an algebra
isomorphism $h_{i}:U_{q_{i}}(\widehat{\mathfrak{sl}}_{2})\rightarrow\mathcal{U}_{i}$
given by
\begin{align*}
h_{i}(e_{1}) & =e_{i},\quad h_{i}(f_{1})=f_{i},\quad h_{i}(k_{1})=k_{i},\\
h_{i}(e_{0}) & =T_{\varpi_{i}}(-k_{i}^{-1}f_{i}),\quad h_{i}(f_{0})=T_{\varpi_{i}}(-e_{i}k_{i}),\quad h_{i}(k_{0})=T_{\varpi_{i}}(k_{i}^{-1})=ck_{i}^{-1}.
\end{align*}
Moreover, we have $T_{i}\restriction_{\mathcal{U}_{i}}=h_{i}\circ T_{1}\circ h_{i}^{-1}$
and $T_{\varpi_{i}}\restriction_{\mathcal{U}_{i}}=h_{i}\circ T_{\varpi_{1}}\circ h_{i}^{-1}$
where $T_{1},T_{\varpi_{1}}$ are braid symmetries on $U_{q_{i}}(\widehat{\mathfrak{sl}}_{2})$.
\end{prop}

For $i,j\in\mathring{I}$ and $k>0$, put $Q_{ij,k}=\frac{\mathbf{q}(\alpha_{i},\alpha_{j})^{k}-\mathbf{q}(\alpha_{i},\alpha_{j})^{-k}}{q-q^{-1}}$
and $b_{ij}=\mathbf{q}(\alpha_{i},\alpha_{j})c^{-1/2}$. Note that
in the statement (2) below, we do not know yet whether $T_{\varpi_{i}}(\overline{\psi}_{i,r})=\overline{\psi}_{i,r}$
when $p(i)=1$.
\begin{lem}[{cf. \cite[Lemma 8.3.2]{Y}}]
\begin{enumerate}
\item Let $i\in I$ be such that $p(i)=0$, $r>0$ and $m\in\mathbb{Z}$.
Then $T_{\varpi_{i}}(\overline{\psi}_{i,r})=\overline{\psi}_{i,r}$
and we have
\begin{align*}
[\overline{\psi}_{i,r},T_{\varpi_{i}}^{m}(f_{i})] & =-c^{1/2}Q_{ii,1}\left((q-q^{-1})\sum_{k=1}^{r-1}b_{ii}^{1-k}T_{\varpi_{i}}^{m+k}(f_{i})\overline{\psi}_{i,r-k}+b_{ii}^{1-r}T_{\varpi_{i}}^{m+r}(f_{i})\right),\\{}
[\overline{\psi}_{i,r},T_{\varpi_{i}}^{m}(e_{i})] & =c^{-1/2}Q_{ii,1}\left((q-q^{-1})\sum_{k=1}^{r-1}b_{ii}^{k-1}T_{\varpi_{i}}^{m-k}(e_{i})\overline{\psi}_{i,r-k}+b_{ii}^{r-1}T_{\varpi_{i}}^{m-r}(e_{i})\right).
\end{align*}
\item For $i,j\in\mathring{I}$ with $i\neq j$, $r>0$ and $s,m\in\mathbb{Z}$,
we have
\begin{align*}
[T_{\varpi_{i}}^{s}(\overline{\psi}_{i,r}),T_{\varpi_{j}}^{m}(f_{j})] & =c^{1/2}Q_{ij,1}\left((q-q^{-1})\sum_{k=1}^{r-1}(-b_{ij})^{1-k}T_{\varpi_{j}}^{m+k}(f_{j})T_{\varpi_{i}}^{s}(\overline{\psi}_{i,r-k})+(-b_{ij})^{1-r}T_{\varpi_{j}}^{m+r}(f_{j})\right),\\{}
[T_{\varpi_{i}}^{s}(\overline{\psi}_{i,r}),T_{\varpi_{j}}^{m}(e_{j})] & =-c^{-1/2}Q_{ij,1}\left((q-q^{-1})\sum_{k=1}^{r-1}(-b_{ij})^{k-1}T_{\varpi_{j}}^{m-k}(e_{j})T_{\varpi_{i}}^{s}(\overline{\psi}_{i,r-k})+(-b_{ij})^{r-1}T_{\varpi_{j}}^{m-r}(e_{j})\right).
\end{align*}
\end{enumerate}
\end{lem}

\begin{lem}[{cf. \cite[Lemma 8.3.3]{Y}}]
For $i\in\mathring{I}$ and $r\in\mathbb{Z}$, we have
\begin{align*}
[T_{\varpi_{i}}^{r}(f_{i}),f_{i}]_{\mathbf{q}(\alpha_{i},\alpha_{i})^{-1}} & =-[T_{\varpi_{i}}(f_{i}),T_{\varpi_{i}}^{r-1}(f_{i})]_{\mathbf{q}(\alpha_{i},\alpha_{i})^{-1}}\\{}
[e_{i},T_{\varpi_{i}}^{r}(e_{i})]_{\mathbf{q}(\alpha_{i},\alpha_{i})^{-1}} & =-[T_{\varpi_{i}}^{r-1}(e_{i}),T_{\varpi_{i}}(e_{i})]_{\mathbf{q}(\alpha_{i},\alpha_{i})^{-1}}
\end{align*}
Moreover, if $p(i)=0$, then $[T_{\varpi_{i}}^{r}(f_{i}),f_{i}]_{\mathbf{q}(\alpha_{i},\alpha_{i})^{-1}}=0=[T_{\varpi_{i}}^{r}(e_{i}),e_{i}]_{\mathbf{q}(\alpha_{i},\alpha_{i})^{-1}}$.
\end{lem}

\begin{lem}[{cf. \cite[Lemma 8.3.4, 8.3.5]{Y}}]
For $i\in\mathring{I}$ such that $p(i)=1$, we have $T_{\varpi_{i}}(\overline{\psi}_{i,r})=\overline{\psi}_{i,r}$
and
\[
[\overline{\psi}_{i,r},T_{\varpi_{i}}^{m}(e_{i})]=0=[\overline{\psi}_{i,r},T_{\varpi_{i}}^{m}(f_{i})].
\]
\end{lem}

\begin{lem}[{cf. \cite[Lemma 8.3.6]{Y}}]
 For $i,j\in\mathring{I}$ and $r,l>0$, we have $[\overline{\psi}_{i,r},\overline{\psi}_{j,l}]=0$.
\end{lem}

From this results, one can then verify (\ref{eq:GQG-Drinfeld-q-Heis}-\ref{eq:GQG-Drinfeld-quadSerre}).
The remaining Serre relations (\ref{eq:GQG-Drinfeld-eSerre}) and
(\ref{eq:GQG-Drinfeld-oSerre}) can be deduced inductively by applying
$T_{\varpi_{i}}^{\pm1}$ to the Serre relations between $x_{i,0}^{+}=e_{i}$'s,
as in the proof of \cite[Theorem 4.7]{B2}, concluding our construction
of an algebra homomorphism $\mathcal{B}$ from $\mathcal{U}^{\mathrm{Dr}}(\epsilon)$
to $\mathcal{U}(\epsilon)$.
\begin{rem}
While it is assumed in \cite[Section 8]{Y} that $n>3$, the same
argument applies to $n=3$ as well.
\end{rem}

Finally to prove the surjectivity of $\mathcal{B}$, we argue as in
\cite[Theorem 12.11]{D2} to show that $e_{0}$ is contained in the
image of $\mathcal{B}$. First, by Proposition \ref{prop:Tw-on-ef}
we have $T_{\varpi_{i}}(x)=x$ for any $x\in\mathcal{U}_{j}$, $j\neq i$,
and so the image of $\mathcal{B}$ is closed under $T_{\varpi_{i}}$
for each $i\in\mathring{I}$. Recall the following reduced expression
of $t_{\varpi_{1}}$
\[
t_{\varpi_{1}}=\tau s_{n-1}s_{n-2}\cdots s_{1}.
\]
Consider $f=T_{1}^{-1}T_{2}^{-1}\cdots T_{n-2}^{-1}(f_{n-1})\in\mathcal{U}(\epsilon)$,
which is made of $f_{i}=x_{i,0}^{-}$ for $i=1,\dots,n-1$ and hence
contained in the image of $\mathcal{B}$. Then 
\[
\mathrm{im}\mathcal{B}\ni T_{\varpi_{1}}(f)=ZT_{n-1}(f_{n-1})=Z(-e_{n-1}k_{n-1})=-e_{0}k_{0}=-e_{0}c\prod_{i=1}^{n-1}k_{i}^{-1}
\]
and consequently $e_{0}\in\mathrm{im}\mathcal{B}$ as well.
\begin{rem}
\label{rem:GQG-Dr-e_0}The last argument also yields an explicit description
of $e_{0}$ in terms of loop generators. Indeed, by definition of
$T_{i}^{-1}$ we have
\begin{align*}
f & =T_{1}^{-1}T_{2}^{-1}\cdots T_{n-2}^{-1}(f_{n-1})\\
 & =[f_{1},[\cdots,[f_{n-2},f_{n-1}]_{\mathbf{q}_{\epsilon s_{1}\cdots s_{n-3}}(\alpha_{n-2},\alpha_{n-1})^{-1}}\cdots]_{\mathbf{q}_{\epsilon s_{1}}(\alpha_{2},\alpha_{3}+\cdots+\alpha_{n-1})^{-1}}]_{\mathbf{q}_{\epsilon}(\alpha_{1},\alpha_{2}+\cdots+\alpha_{n-1})^{-1}}\\
 & =[f_{1},[\cdots,[f_{n-2},f_{n-1}]_{q_{n-1}}\cdots]_{q_{3}}]_{q_{2}}.
\end{align*}
Since $T_{\varpi_{1}}(f_{i})=f_{i}$ for $i=2,\dots,n-1$, we obtain
\begin{align*}
e_{0} & =-T_{\varpi_{1}}(f)c^{-1}\prod_{i=1}^{n-1}k_{i}\\
 & =-o(1)[x_{1,1}^{-},[f_{2},[\cdots,[f_{n-2},f_{n-1}]_{q_{n-1}}\cdots]_{q_{4}}]_{q_{3}}]_{q_{2}}c^{-1}\prod_{i=1}^{n-1}k_{i}.
\end{align*}
\[
\]
\end{rem}

\section{A rank 1 case $\mathcal{U}(01)$\label{sec:GQG-RT-rank1}}

In this section we summarize basic facts on the generalized quantum
groups $\mathcal{U}(01)$ associated to $\widehat{\mathfrak{sl}}_{1|1}$
and their finite-dimensional representations. We will keep the conventions
from Section \ref{sec:RT-GQG}, for instance our base field is $\mathbf{k}$
and every $\mathcal{U}(01)$-module below is finite-dimensional and
$\mathring{P}$-graded.
\begin{defn}
The generalized quantum group $\mathcal{U}(\epsilon)$ for $\epsilon=(01)$
is defined as the $\mathbf{k}$-algebra generated by $x_{k}^{\pm}$
($k\in\mathbb{Z}$), and $\psi_{\pm r}^{\pm}$ ($r\geq0$) subject
to the relations
\begin{align*}
\psi_{0}^{+}\psi_{0}^{-}=\psi_{0}^{-}\psi_{0}^{+} & =1,\quad[\psi_{r}^{\pm},\psi_{s}^{\pm}]=[\psi_{r}^{+},\psi_{s}^{-}]=0,\\
\psi_{r}^{\pm}x_{k}^{+}=(q_{1}q_{2})^{\pm1} & x_{k}^{+}\psi_{r}^{\pm},\quad\psi_{r}^{\pm}x_{k}^{-}=(q_{1}q_{2})^{\mp1}x_{k}^{-}\psi_{r}^{\pm}\\{}
[x_{k}^{+},x_{l}^{-}] & =\frac{\psi_{k+l}^{+}-\psi_{k+l}^{-}}{q-q^{-1}},\\
x_{k}^{\pm}x_{l}^{\pm} & +x_{l}^{\pm}x_{k}^{\pm}=0
\end{align*}
where we assume $\psi_{\pm m}^{\pm}=0$ for $m<0$. We also let $\mathring{\mathcal{U}}(01)$
be the subalgebra generated by $e=x_{0}^{+}$, $f=x_{0}^{-}$ and
$k^{\pm1}=\psi_{0}^{\pm}$. 

The algebra $\mathcal{U}(10)$ (and its representation theory below)
is obtained by exchaning the role of $q$ and $\widetilde{q}=-q^{-1}$
everywhere and hence we will focus on the case $\epsilon=(01)$. 
\end{defn}

Unlike the general case above, $\mathcal{U}(01)$ is defined directly
in terms of Drinfeld presentation. Indeed, as $\mathfrak{sl}_{1|1}$
is a nilpotent Lie superalgebra, its loop algebra $L\mathfrak{sl}_{1|1}$
cannot be generated by (finitely many) Chevalley generators. Still,
by the same reason, the structure of $\mathcal{U}(01)$ is exceptionally
simple even compared with $U_{q}(\widehat{\mathfrak{sl}}_{2})$. For
example, one can readily prove the following basic results.
\begin{prop}
The algebra $\mathcal{U}(01)$ has a Hopf algebra structure given
by
\begin{align*}
\Delta: & \begin{cases}
x^{+}(z)_{\pm}\longmapsto\psi^{\pm}(z)\otimes x^{+}(z)_{\pm}+x^{+}(z)_{\pm}\otimes1,\\
x^{-}(z)_{\pm}\longmapsto1\otimes x^{-}(z)_{\pm}+x^{-}(z)_{\pm}\otimes\psi^{\pm}(z),\\
\psi^{\pm}(z)\longmapsto\psi^{\pm}(z)\otimes\psi^{\pm}(z),
\end{cases}\\
S: & \begin{cases}
x^{+}(z)_{+}\longmapsto-\psi^{+}(z)^{-1}x^{+}(z)_{+},\quad x^{+}(z)_{-}\longmapsto-\psi^{-}(z)^{-1}x^{+}(z)_{-},\\
x^{-}(z)_{+}\longmapsto-x^{-}(z)_{+}\psi^{+}(z)^{-1},\quad x^{-}(z)_{-}\longmapsto-x^{-}(z)_{-}\psi^{-}(z)^{-1},\\
\psi^{\pm}(z)\longmapsto\psi^{\pm}(z)^{-1}
\end{cases}
\end{align*}
where we set 
\begin{align*}
x^{+}(z)_{+}= & \sum_{k\geq0}x_{k}^{+}z^{k},\:x^{+}(z)_{-}=\sum_{k<0}x_{k}^{+}z^{k},\\
x^{-}(z)_{+}= & \sum_{k>0}x_{k}^{-}z^{k},\;x^{+}(z)_{+}=\sum_{k\leq0}x_{k}^{-}z^{k},\\
\psi^{\pm}(z) & =\sum_{r\geq0}\psi_{\pm r}^{\pm}z^{r}.
\end{align*}
Note that the inverse of $\psi^{\pm}(z)$ is well-defined in $\mathcal{U}(01)\left\llbracket z\right\rrbracket $.
\end{prop}

\begin{prop}
Let $X^{\pm}$ (resp. $H$) be the subalgebra of $\mathcal{U}(01)$
generated by $x_{k}^{\pm}$ for $k\in\mathbb{Z}$ (resp. $\psi_{\pm r}^{\pm}$
for $r\geq0$). As a vector space, we have
\[
\mathcal{U}(01)\cong X^{+}\otimes H\otimes X^{-}\cong X^{-}\otimes H\otimes X^{+}.
\]
\end{prop}

\begin{prop}
For each $a\in\mathbf{k}^{\times}$, there exists an algebra homomorphism
\begin{align*}
\mathrm{ev}_{a}:\mathcal{U}(01) & \longrightarrow\mathring{\mathcal{U}}(01)\\
x_{l}^{+},\,x_{l}^{-},\,\psi_{0}^{\pm} & \longmapsto a^{l}k^{l}e,\,a^{l}fk^{l},\,k^{\pm1}\quad\text{respectively,}\\
\psi_{\pm s}^{\pm} & \longmapsto\pm(q-q^{-1})a^{s}k^{s}(ef-(-1)^{s}fe)\quad(s>0).
\end{align*}
\end{prop}

From the triangular decomposition, one can develop the theory of highest
(and similarly lowest) $\ell$-weight modules for $\mathcal{U}(01)$
as well. Since the purpose of this appendix is to give a basis for
the Frenkel--Mukhin algorithm in Section , we will focus on the category
$\mathcal{F}(01)$ defined analogously and the $q$-characters of
simple modules contained in it. Before that, we record one observation
that follows from the comultiplication formula for $\psi^{\pm}(z)$,
whose proof can be found in \cite[Lemma 3.7]{Z2}.
\begin{lem}[{cf. \cite[Lemma 3.7]{Z2}}]
 Let $V_{+}$ (resp. $V_{-}$) be a highest (resp. lowest) $\ell$-weight
$\mathcal{U}(01)$-module generated by a highest (resp. lowest) $\ell$-weight
vector $v_{+}$ (resp. $v_{-}$). Then $V_{+}\otimes V_{-}$ is generated
by $v_{+}\otimes v_{-}$.
\end{lem}

Given $i,j\in\mathbb{Z}$, let $K(i,j)$ be the 2-dimensional $\mathring{\mathcal{U}}(01)$-module
spanned by $v_{0}$ and $v_{1}$ of weight $i\delta_{1}+j\delta_{2}$
and $(i-1)\delta_{1}+(j+1)\delta_{2}$ respectively, with the action
given by
\[
ev_{0}=0,\,fv_{0}=v_{1},\quad ev_{1}=\frac{q_{1}^{i}q_{2}^{-j}-q_{1}^{-i}q_{2}^{j}}{q-q^{-1}}v_{0},\,fv_{1}=0.
\]
This is an analogue of the Kac module of $\mathfrak{sl}_{1|1}$ (see
for example \cite[Section 2.1.1]{CW}). Moreover, $K(i,j)$ is irreducible
if and only if $i+j=0$. We let $V(i,j)$ be the (unique) simple quotient
of $K(i,j)$, namely
\[
V(i,j)=\begin{cases}
K(i,j)/\mathbf{k}v_{1} & \text{if }i+j=0\\
K(i,j) & \text{otherwise}.
\end{cases}
\]
They exhaust all finite-dimensional $P$-graded $\mathring{\mathcal{U}}(01)$-modules.

For $a\in\mathbf{k}^{\times}$ and a $\mathring{\mathcal{U}}(01)$-module
$V$, the evaluation module $V_{a}$ is defined to be the $\mathcal{U}(01)$-module
obtained by pulling back $V$ through $\mathrm{ev}_{a}$. Since $K(i,j)_{a}$
is $2$-dimensional, it is very easy to compute its $\ell$-weights
using the formula for $\mathrm{ev}_{a}(\psi_{\pm r}^{\pm})$ above:
\[
\psi^{\pm}(z)v_{0}=q_{1}^{i}q_{2}^{-j}\frac{1-q_{1}^{-i}q_{2}^{j}az}{1-q_{1}^{i}q_{2}^{-j}az}v_{0},\quad\psi^{\pm}(z)v_{1}=q_{1}^{i-1}q_{2}^{-j-1}\frac{1+q_{1}^{-i+1}q_{2}^{j+1}az}{1+q_{1}^{i-1}q_{2}^{-j-1}az}v_{1}.
\]
Its simple quotient is also the evaluation module $V(i,j)_{a}$ of
$V(i,j)$, and hence we know its $q$-character. In the terminology
of Section \ref{subsec:classification-irrep} and \ref{subsec:fund-repn},
$V(1,0)_{a}$ (resp. $V(0,-1)_{a}$) is the fundamental representation
$L(Y_{a})$ (resp. $L(\widetilde{Y}_{a})$), where we are omitting
the index $i$ for $Y_{i,a}$ as we are in a rank 1 case. Also, $V(1,-1)_{a}=D^{+}$
is the 1-dimensional module independent of $a\in\mathbf{k}^{\times}$
and $D^{+}\otimes-$ has the inverse $D^{-}\otimes-$ for $D^{-}=V(-1,1)_{a}$.
\begin{lem}
For $i,j\in\mathbb{Z}$ and $a\in\mathbf{k}^{\times}$, we have $V(i,j)_{a}^{*}\cong V(-i+1,-j-1)_{-a}$. 
\end{lem}

Let us denote the highest $\ell$-weight of $V(i,j)_{a}$ by
\[
S_{i,j}(a)(z)=q_{1}^{i}q_{2}^{-j}\frac{1-q_{1}^{-i}q_{2}^{j}az}{1-q_{1}^{i}q_{2}^{-j}az}\in\mathbf{k}(z)
\]
and say $S_{i,j}(a)$ and $S_{i^{\prime},j^{\prime}}(b)$ are \textit{in
general position} if $a/b\neq(q_{1}^{i+i^{\prime}}q_{2}^{-j-j^{\prime}})^{\pm1}$.
We assume that $S_{1,-1}(a)=Y_{a}\widetilde{Y}_{-a}=D$ is in general
position with any other $S_{i,j}(b)$. Note that this condition is
much weaker than the one for $U_{q}(\widehat{\mathfrak{sl}}_{2})$
\cite{CP} which can be understood from the fact that $V(i,j)_{a}$
(which plays the role of Kirillov--Reshetikhin modules here) is at
most 2-dimensional. The following theorem is a super analogue of \cite[Theorem 4.8]{CP}
and allows us to compute the $q$-character of any given finite-dimensional
simple $\mathcal{U}(01)$-module. Although it was given and proved
for modules over the Borel subalgebra of $U_{q}(\widehat{\mathfrak{gl}}_{1|1})$,
with two lemmas above the same argument works for our case as well.
\begin{thm}[{cf. \cite[Theorem 4.2]{Z2}}]
 Given a sequence $(S_{i_{k},j_{k}}(a_{k}))_{k=1,\dots,t}$, the
tensor product $V(i_{1},j_{1})_{a_{1}}\otimes\cdots\otimes V(i_{k},j_{k})_{a_{k}}$
is a simple $\mathcal{U}(01)$-module if and only if the entries of
the sequence are pairwise in general position. 
\end{thm}

Now we explain how to compute the $q$-character of any simple module
in $\mathcal{F}(01)$, whose highest $\ell$-weights $\Psi$ is a
monomial in $Y_{a}^{\pm1}$ and $\widetilde{Y}_{a}^{\pm}$ for $a\in\mathbf{k}^{\times}$.
First, $\Psi$ can be represented uniquely in the following form:
\[
\Psi=D^{-s}\prod Y_{a_{i}}\prod\widetilde{Y}_{b_{j}},\quad a_{i}\neq-b_{j}\,^{\forall}i,j,\ s\in\mathbb{Z}.
\]
Next, we rewrite the factor $\prod Y_{a_{i}}$ as a product of $\ell$-weights
of the form $S_{i_{k},0}(c_{i_{k}})$ which are pairwise in general
position. This can be done as in the case of $U_{q}(\widehat{\mathfrak{sl}}_{2})$
\cite[Section 4.7]{CP} as the condition for being in a general position
is weakened. Similarly we can do it for $\prod\widetilde{Y}_{b_{j}}$
in terms of $S_{0,-j_{l}}(d_{i_{l}})$ to obtain
\[
\Psi=D^{-s}\prod S_{i_{k},0}(c_{i_{k}})\prod S_{0,-j_{l}}(d_{j_{l}}).
\]
Then any two $\ell$-weights $S_{i_{k},0}(c_{i_{k}})$ and $S_{0,-j_{l}}(d_{j_{l}})$
should be in general position, otherwise it violates the above condition
$a_{i}\neq-b_{j}$ for all $i,j$. Therefore, we conclude
\[
\chi_{q}(L(\Psi))=D^{-s}\prod(S_{i_{k},0}(c_{i_{k}})+S_{i_{k},0}(c_{i_{k}})D^{-1})\prod(S_{0,-j_{l}}(d_{j_{l}})+S_{0,-j_{l}}(d_{j_{l}})D^{-1})
\]
where $D=A_{a}=Y_{a}\widetilde{Y}_{a}$, independent of $a$.

\end{document}